\documentclass[reqno, a4paper]{amsart}
\usepackage{amssymb, mathdots}
\usepackage{mathabx}
\usepackage{mathrsfs}
\usepackage[utf8]{inputenc}
\usepackage{amsmath,  graphicx, tensor}
\usepackage{tikz}
\usetikzlibrary{matrix, arrows.meta}
\usetikzlibrary{cd}

\usepackage{enumerate}
\usepackage{hyperref}
\DeclareSymbolFont{bbold}{U}{bbold}{m}{n}
\DeclareSymbolFontAlphabet{\mathbbold}{bbold}
\def\qmod#1#2{{\hbox{}^{\displaystyle{#1}}}\!\big/\!\hbox{}_{
\displaystyle{#2}}}

 \def\psp#1#2%
  {\mathop{}%
   \mathopen{\vphantom{#2}}^{#1}%
   \kern-\scriptspace%
    \hskip -0.3mm{#2}} 
 \def\psb#1#2%
  {\mathop{}%
   \mathopen{\vphantom{#2}}_{#1}%
   \kern-\scriptspace%
    \hskip -0.0mm{#2}} 
 \def\pscr#1#2#3%
  {\mathop{}%
   \mathopen{\vphantom{#3}}^{#1}_{#2}%
   \kern-\scriptspace%
    \hskip -0.3mm{#3}} 

\def\C{{\mathbb C}}

\def\K{{\mathbb K}}

\def\N{{\mathbb N}}

\def\P{{\mathbb P}}

\def\R{{\mathbb R}}

\def\Z{{\mathbb Z}}

\def\cringle{\mathaccent23}
\def\union{\mathop{\bigcup}}

\def\textmap#1{\mathop{\vbox{\ialign{
                                  ##\crcr
      ${\scriptstyle\hfil\;\;#1\;\;\hfil}$\crcr
      \noalign{\kern 1pt\nointerlineskip}
      \rightarrowfill\crcr}}\;}}
\def\bigtextmap#1{\mathop{\vbox{\ialign{
                                  ##\crcr
      ${\hfil\;\;#1\;\;\hfil}$\crcr
      \noalign{\kern 1pt\nointerlineskip}
      \rightarrowfill\crcr}}\;}}
      
\newcommand{\cal}{\mathcal}
\def\textlmap#1{\mathop{\vbox{\ialign{
                                  ##\crcr
      ${\scriptstyle\hfil\;\;#1\;\;\hfil}$\crcr
      \noalign{\kern-1pt\nointerlineskip}
      \leftarrowfill\crcr}}\;}}
 
\def\ag{{\mathfrak a}}

\def\dg{{\mathfrak d}}

\def\fg{{\mathfrak f}}
\def\g{{\mathfrak g}}
\def\hg{{\mathfrak h}}

\def\kg{{\mathfrak k}}
\def\lg{{\mathfrak l}}

\def\sg{{\mathfrak s}}
\def\tg{{\mathfrak t}}

\def\Ag{{\mathfrak A}}

\def\Cg{{\mathfrak C}}

\def\Fg{{\mathfrak F}}

\def\Ig{{\mathfrak I}}
\def\Jg{{\mathfrak J}}

\def\Og{{\mathfrak O}}

\def\Qg{{\mathfrak Q}}
\def\Sg{{\mathfrak S}}
\def\Tg{{\mathfrak T}}

\theoremstyle{remark}
\newtheorem{ex}{Example}[section]
\newtheorem{pb}{Problem}

\newtheorem*{cl}{Claim}
\newtheorem*{cond}{Compatibility condition}
\newtheorem{prop}{Property}

\newtheorem{sz}{Satz}[section]
\theoremstyle{remark}
\newtheorem{re}[sz]{Remark} 
\theoremstyle{plain}
\newtheorem{thr}[sz]{Theorem}
\newtheorem{pr}[sz]{Proposition}
\newtheorem{co}[sz]{Corollary}
\newtheorem{dt}[sz]{Definition}
\newtheorem{lm}[sz]{Lemma}

\def\End{\mathrm {End}}
\def\Aut{\mathrm {Aut}}

\def\SU{\mathrm {SU}}

\def\GL{\mathrm {GL}}
\def\gl{\mathrm {gl}}
\def\SL{\mathrm {SL}}

\def\Pic{\mathrm {Pic}}

\def\deg{\mathrm {deg}}
\def\Hom{\mathrm{Hom}}

\def\id{ \mathrm{id}}
\def\im{\mathrm{im}}

\def\p{\mathrm{p}}

\newcommand\smvee{{\hskip -0.1ex \raise 0.2ex\hbox{$\scriptscriptstyle\vee$}}\hskip -0,3ex}

\newcommand{\extpw}{\mathchoice{{\textstyle\bigwedge}}%
    {{\bigwedge}}%
    {{\textstyle\wedge}}%
    {{\scriptstyle\wedge}}}
\def\extp{{\extpw}\hspace{-2pt}}

\def\Ad{\mathrm{Ad}}
\def\trp#1{\tensor[^{\mathrm{t}}]{#1}{}}
\def\edf{\coloneq}
\def\hb{\hbox}
\def\fr{\frac}
\def\p{\partial}
\def\bp{\bar\partial}

\begin{document}

\title[Holomorphic bundles framed along a real hypersurface]{Holomorphic bundles framed along a real hypersurface and the Riemann-Hilbert problem}

\author{Andrei Teleman}
\address{Aix Marseille Univ, CNRS, I2M, Marseille, France.}
\email[Andrei Teleman]{andrei.teleman@univ-amu.fr}

\begin{abstract}
Let $X$ be a connected, compact complex manifold,  $S\subset X$ a separating real hypersurface, so $X$ decomposes as a union of compact complex manifolds with boundary $\bar X^\pm$ with $\bar X^+\cap \bar X^-=S$.   Let $\mathcal{M}$ be the  moduli space of $S$-framed holomorphic bundles, i.e. of pairs $(E,\theta)$ of fixed topological type consisting of a {\it holomorphic} bundle $E$ on $X$ and a trivialization $\theta$ -- belonging to a fixed Hölder regularity class $\mathcal{C}^{\kappa+1}$ -- of its restriction to  $S$.    

Our problem: compare, via the obvious restriction maps, the moduli space  $\mathcal{M}$  to the corresponding Donaldson's moduli spaces $\mathcal{M}^\pm$ of boundary framed formally holomorphic bundles on $\bar X^\pm$. The restrictions to $\bar X^\pm$ of an $S$-framed holomorphic bundle $(E,\theta)$   are boundary framed formally holomorphic bundles $(E^\pm,\theta^\pm)$ which induce,  via $\theta^\pm$, the same   tangential Cauchy-Riemann operator on the trivial bundle on $S$. Therefore one obtains a natural map from $\mathcal{M}$ into the fiber product $\mathcal{M}^-\times_\mathcal{C}\mathcal{M}^+$ over the space $\mathcal{C}$ of Cauchy-Riemann operators on the trivial bundle on $S$. 
Our main result states:  this map is a homeomorphism for $\kappa\in (0,\infty]\setminus\mathbb{N}$.    Note that, by theorems due to S. Donaldson and Z. Xi,  the moduli spaces   $\mathcal{M}^\pm$ can be  further identified with moduli spaces of boundary framed Hermitian Yang-Mills connections.

 The proof of our isomorphism theorem is based on a gluing principle for formally holomorphic bundles along a real hypersurface. The same gluing theorem can be used to give a complex geometric interpretation of the space of solutions of a large class of Riemann-Hilbert type problems. 
 
 We generalize these results  in two directions: first, we will replace the decomposition $X=\bar X^-\cup\bar X^+$ associated with a separating hypersurface by the  manifold with boundary $\widehat X_S$ obtained by cutting $X$ along  any (not necessarily separating) oriented  hypersurface $S$. Second, instead of vector bundles, we will consider principal $G$ bundles for an arbitrary  complex Lie group $G$. 
 
 We give explicit examples of moduli spaces of (boundary) framed holomorphic bundles and explicit formulae for the homeomorphisms provided by the general results.

 \end{abstract}
 

\maketitle

\tableofcontents

\setcounter{section}{-1}

\section*{Acknowledgemets}

The problems investigated in this article originate from a joint project in collaboration with Matei Toma devoted to the boundedness condition for torsion free coherent sheaves in the complex geometric (not necessarily algebraic) framework. I am indebted to him for the interesting questions he formulated, and also for his useful remarks  on preliminary versions of this article.     

I am also grateful to Alexander Borichev, Karl Oeljeklaus and Christine Laurent-Thiébaut for the time they invested in answering my questions, and  for their useful ideas, comments and suggestions. 

\section{Introduction}

 A fundamental problem in the theory of holomorphic bundles on compact complex manifolds is: understand, in the general (non-necessarily algebraic or Käherian) framework, the relation between convergence in the  space  of singular Hermitian-Einstein connections (Donaldson, Tian)  and convergence of sheaves in the sense of complex geometric deformation theory. Working on this problem in collaboration with Matei Toma, I noticed that Donaldson's article \cite{Do} -- which deals with the   correspondence between Hermitian-Einstein connections and holomorphic bundles on compact complex manifolds with boundary -- is relevant for our problem. Donaldson's article comes with a fundamental new idea: in the presence of a boundary, it's natural to consider infinite dimensional moduli spaces of {\it boundary framed} Hermitian Yang-Mills connections, respectively holomorphic bundles. 

A boundary framed Hermitian Yang-Mills connection on $\bar X$ is a triple $(E,A,\theta)$, where $E$ is a Hermitian vector bundle on $\bar X$, $A$  a Hermitian Yang-Mills connection on $E$, and $\theta$ a {\it unitary} trivialization of  $E_{\p\bar X}$. A boundary framed formally holomorphic vector bundle on $\bar X$ is a triple $(E,\delta,\theta)$, where $E$ is a differentiable vector bundle on $\bar X$, $\delta$ is a Dolbeault operator on $E$ satisfying the {\it formal integrability} condition $\delta^2=0$ (see \cite{Te} and section \ref{DolbeaultSect} in this article), and $\theta$ is a {\it differentiable} trivialization of  $E_{\p\bar X}$. 

Donaldson's theorem \cite[Theorem 1']{Do} yields an isomorphism between moduli spaces of gauge theoretical, respectively complex geometric boundary framed  objects. 
An  interesting  application of this isomorphism theorem: a  new proof of a fundamental factorization theorem in loop group theory  (see \cite[p. 100]{Do}). 

The manifolds with boundary which appear naturally in our complex geometry project are of the form $\bar X^\pm$ where $X^\pm\subset X$ are the open submanifolds obtained by cutting the given closed complex manifold $X$ along a separating real hypersurface $S\subset X$.  In our original joint project   we  focus on the case when $S$ is the boundary of a neighborhood of the bubbling locus of a weakly convergent sequence of Hermitian-Einstein connections. Relevant for the present article: in the presence of a  real hypersurface $S$ of a closed complex manifold $X$ it's natural to consider  moduli spaces of $S$-framed holomorphic bundles  {\it on the whole closed manifold} $X$, i.e. of holomorphic bundles $E$ on $X$ endowed with a differentiable trivialization $\theta$ on $S$. One should of  course fix the topological type of the pair $(E,\theta)$. 

Although infinite dimensional, such a moduli space can be constructed explicitly and studied using  techniques and methods from the classical deformation theory for analytic objects on compact complex spaces.  A joint article in preparation  \cite{TT} is dedicated to these moduli spaces and their role in our initial project.

The starting point of the present article is the natural problem: supposing that $S$ separates $X$, compare, via the obvious restriction maps, the moduli space  ${\cal M}$ of $S$-framed holomorphic bundles (of fixed topological type) on $X$, with the corresponding Donaldson's moduli spaces ${\cal M}^\pm$ of boundary framed holomorphic bundles on $\bar X^\pm$. The restrictions to $\bar X^\pm$ of an $S$-framed holomorphic bundle $(E,\theta)$ of rank $r$  are boundary framed formally holomorphic bundles $(E^\pm,\theta^\pm)$ which induce,  via $\theta^\pm$, the same   tangential Cauchy-Riemann operators on the trivial bundle of rank $r$ on $S$. Therefore one obtains a natural comparison map from ${\cal M}$ into the fiber product ${\cal M}^-\times_{\cal C}{\cal M}^+$ over the space ${\cal C}$ of Cauchy-Riemann operators on the trivial bundle of rank $r$ on $S$. 

At this point note that we will work in a more general framework: in the definitions of our framed objects  we will fix a regularity parameter $\kappa\in[0,\infty]$  and we will require ${\cal C}^\kappa$-regularity (see section \ref{Ckappa}) for $\delta$ in the definition of a boundary framed formally holomorphic vector bundle, and   ${\cal C}^{\kappa+1}$-regularity for $ \theta$ in the definitions of a (formally) holomorphic   $S$-framed (boundary framed) vector bundle.  The moduli spaces ${\cal M}$, ${\cal M}^\pm$ and the comparison map are defined in this more general framework.

The isomorphism  Theorem \ref{iso-moduli-vector-bundles-th} proved in this article states: {\it the comparison map ${\cal M}\to {\cal M}^-\times_{\cal C}{\cal M}^+$ is a homeomorphism for $\kappa\in(0,+\infty)\setminus\N$}. The results also holds for $\kappa=\infty$ if the considered moduli spaces  are endowed with  suitable topologies, see Remark \ref{kappa=infty}.    The meaning of this isomorphism theorem can be intuitively  expressed as a general principle: in the moduli theory for holomorphic bundles on closed complex manifolds, {\it framing} on a real hypersurface $S$ is equivalent to {\it cutting} along $S$.   Note that for $\dim_\C(X)=1$ the compatibility condition on the induced Cauchy-Riemann operators becomes void  so, on Riemann surfaces, the principle ``framing on $S$ is equivalent to cutting along $S$" becomes simply ${\cal M}={\cal M}^-\times {\cal M}^+$.

The difficult part of the isomorphism theorem  is the surjectivity, which follows from the gluing principle given by the crucial Theorem \ref{mainGnew}: let $U$ be a (not necessarily compact) complex manifold,  $S\subset U$ a closed, separating, smooth, real hypersurface, $\bar U^\pm$ be the corresponding manifolds with boundary,  $E$   a ${\cal C}^\infty$-bundle on $U$ and $\delta^\pm$ be {formally} integrable Dolbeault operators on $E_{\bar U^\pm}$ with coefficients in ${\cal C}^\kappa$ inducing the same tangential Cauchy-Riemann operators on $S$.  There exists  an automorphism $f_+$ of class ${\cal C}^{\kappa+1}$ of $E_{\bar U^+}$ which is the identity on $S$ such that $\delta^-$ and $f_+(\delta^+)$ glue together and give an integrable Dolbeault operator (so a holomorphic structure) on $E$. For $\kappa\in(0,+\infty)\setminus\N$ the proof makes use of Whitney's extension theorem for Lipschitz  spaces, which allows us to prove that  $f_+$ can be chosen to depend continuously on $(\delta^-,\delta^+)$. For $\kappa=\infty$ we use the ${\cal C}^\infty$ version of Whitney's extension theorem, which does {\it not} provide a continuous extension operator.

  Our gluing principle has other consequences: let   $E^\pm$ be ${\cal C}^\infty$ complex vector bundles on $\bar U^\pm$ and $\delta^\pm$  {formally}  integrable Dolbeault operators with coefficients in ${\cal C}^\kappa$ on $E^\pm$, and let $\upsilon:E^-_S\to E^+_S$ be a   bundle isomorphism of class ${\cal C}^{\kappa+1}$ (with  $\kappa\in(0,+\infty]\setminus\N$) such that the tangential Cauchy-Riemann operators $\delta^\pm_S$ induced by $\delta^\pm$  on $S$ agree via $\upsilon$.  Theorem \ref{new-main} shows, that, under these assumptions, the {\it topological}  bundle $E^\upsilon=E^-\coprod_\upsilon E^+$ on $U$  comes with a {\it canonical} holomorphic structure which extends the holomorphic structures defined by $\delta^\pm_S$ on $E^\pm_{U^\pm}$. Therefore, although the gluing isomorphism $\upsilon$ is only  of class ${\cal C}^{\kappa+1}$, if the above compatibility condition is satisfied, one can glue the {\it formally holomorphic} bundles $E^-$,  $E^+$ via $\upsilon$, and obtain a canonically defined {\it holomorphic} bundle on $U$. In particular, on Riemann surfaces, one can always (no compatibility condition needed) glue formally holomorphic bundles $E^\pm$ on $U^\pm$ via a  ${\cal C}^{\kappa+1}$ bundle isomorphism $\upsilon:E^-_S\to E^+_S$, and obtain a {\it holomorphic} vector bundle on $U$.

Consider the special case where $U=\P^1_\C=\C\cup\{\infty\}$, $S\subset \C$ is a closed curve, and $E^\pm$ are the trivial bundles on $\bar U^\pm$ (endowed with the standard Dolbeault operator $\bp$). An isomorphism $\upsilon$ as above is precisely the input data of the Riemann-Hilbert problem as stated in \cite[Kapitel X]{Hil}. Using this remark we show that a large class\footnote{Several authors state and study more general Riemann-Hilbert problems on $\P^1_\C$, where  $S$ is replaced by a piecewise differentiable, non-necessarily closed, "contour" in $\C$. These generalizations are related to Hilbert's 21-st problem \cite{Bo}. }    of   Riemann-Hilbert type problems, including Hilbert's original problem and matrix factorization problems (see Problem \ref{rho} in section \ref{RHonP1-section}),   can be reduced to a complex geometric problem for holomorphic vector bundles on $\P^1_\C$ (see Corollary \ref{RH-P1}).

Theorem \ref{new-main} can be easily extended to possibly non-separating closed, oriented real hypersurfaces $S$: one just replaces the disjoint union $\bar U^-\coprod\bar U^+$ by the manifold with boundary  $\widehat U_S$ obtained by cutting $U$ along $S$ (see section \ref{non-sep-section} and Fig. \ref{hatUS}).  This generalization  is Theorem \ref{new-main-non-sep}; it applies for instance when $S$ is a non-separating circle on an elliptic curve. This leads us to a  general Riemann-Hilbert type problem associated to a closed Riemann surface $X$ and an arbitrary (non-necessarily connected, non-necessarily separating)  smooth oriented closed curve $S\subset X$ (see Problem \ref{RHXD} in section \ref{RHonRiemannSurf-section})  and to a complex geometric approach to solve it (Corollary \ref{RH-X}).  In section \ref{RHXnD-section} we formulate and study a generalization of the Riemann-Hilbert problem for $n$ dimensional complex manifolds noting that, for $n\geq 2$,   the above compatibility condition   is needed.

Similarly, the moduli space isomorphism   ${\cal M}\simeq{\cal M}^-\times_{\cal C}{\cal M}^+$ can be generalized to the case of an oriented, not necessarily connected, not necessarily separating, real hypersurface $S\subset X$. The boundary $\widehat S$ of $\widehat X_S$ decomposes  as a disjoint union $S^-\cup S^+$ and comes with a canonical identification map $b:S^-\to S^+$. Let $E$ be a vector bundle on $X$ and   $\widehat{E}$ its pull back to $\widehat X_S$. A   formally  integrable Dolbeault operator $\dg$ on $\widehat E$  will be called {\it descendable}, if the tangential Cauchy-Riemann operators on $S^\pm\times \C^r$ induced via $\theta_{S^\pm}$ by   $\dg$ agree via $b$.  The first part of Theorem \ref{iso-moduli-vector-bundles-th} identifies the moduli space of $S$-framed holomorphic bundles (of a fixed topological type) on $X$ with the moduli space of descendable boundary framed formally holomorphic bundles (of the corresponding topological type) on $\widehat X_S$. The intuitive interpretation of this isomorphism is the same as in the separating case: framing on $S$ is equivalent to cutting along $S$.

 If $X$ is a Riemann surface, {\it any} boundary framed holomorphic bundle on $\widehat X_S$ is descendable. Therefore, {\it if $X$ is a closed Riemann surface, the moduli space of $S$-framed holomorphic bundles on $X$ (of a fixed topological type) can be identified with   the corresponding moduli space of boundary framed  formally holomorphic bundles on $\widehat X_S$.}

Suppose now that the  closed Riemann surface $X$ has been endowed with a Hermitian metric.  By Donaldson's isomorphism theorem \cite[Theorem 1']{Do}, the latter moduli space, in its turn,  can be identified with the corresponding moduli space of boundary framed  Hermitian Yang-Mills connections on $\widehat X_S$.  Composing the  two isomorphisms, one obtains an identification between the considered moduli space of $S$-framed holomorphic bundles on $X$ and the corresponding moduli space  of boundary framed   Hermitian Yang-Mills unitary connections on $\widehat X_S$. 
Theorem \ref{new-main-non-sep-G} generalizes  Theorems \ref{new-main},  \ref{new-main-non-sep}  to principal $G$-bundles $P$ endowed with (formally)  integrable bundle almost complex structures  (see section \ref{ACSsection}), where $G$ is an arbitrary complex Lie group. In this general framework the  role of the tangential Cauchy-Riemann operator $\delta_S$ is played by the  almost complex structure $J_S$ induced by a bundle almost complex structure $J$ on the pull back $\Tg_{P_S}\subset T_{P_S}$ of the canonical distribution $\Tg_S\edf T_S\cap J_U T_S$ of $S$.

The above results concerning Riemann-Hilbert problems and isomorphisms between moduli spaces of $S$-framed and boundary framed holomorphic bundles extend to the  framework of principal $G$-bundles.  Moreover, in the definition of our moduli spaces, one can use as framings on $S$ (or as boundary framings) differentiable bundle isomorphisms $\theta:\Phi\to P_S$ ($\theta:\Phi\to P_{\p\bar X}$), where $\Phi$ is a fixed, not necessarily trivial,  differentiable $G$-bundle on $S$ (on $\p\bar X$), see section  \ref{abstract-interpr}. In particular the isomorphism Theorem  \ref{iso-moduli-G-bundles-th} shows that the principle ``framing on $S$ is equivalent to cutting along $S$" generalizes to this framework.  In section \ref{ExamplesSection} we give explicit examples of isomorphisms provided by this theorem on Riemann surfaces and, in some cases, using classical theorems in complex analysis, we give explicit formulae for their inverses.

Of special interest is the case when $G$ is a complex reductive group, because, for such groups, we also have an analogue of Donaldson's isomorphism \cite[Theorem 1']{Do}: one just replaces the moduli space of boundary framed Hermitian Yang-Mills unitary connections by   the moduli space of boundary framed Hermitian Yang-Mills $K$-connections, where $K$ is a fixed maximal compact subgroup of $G$. Therefore, in this case  one can further identify the two moduli spaces     
 intervening in Theorem  \ref{iso-moduli-G-bundles-th} with a moduli space of boundary framed Hermitian Yang-Mills $K$-connections. Explicit examples of such identifications are given in section \ref{ExamplesSection}.  
\vspace{4mm}\\ 
{\bf Notations:}
For a differentiable manifold (possible with boundary) $M$,   a finite dimensional normed space $T$, a ${\cal C}^\infty$ vector bundle $E$ on $M$  and a locally trivial fiber bundle $\Phi$ on $M$ we will use the following notations:
\begin{itemize}
\item[-] ${\cal C}^\kappa(M,T)$: the space of $T$-valued maps of class ${\cal C}^\kappa$ on $M$,  see section \ref{Ckappa}.
 \item[-]	 $\Gamma^\kappa(M,E)$: the space of sections of class ${\cal C}^\kappa$ in $E$, see section \ref{Ckappa}.
  \item[-]	 $\Gamma^\kappa(M,\Phi)$: the space of sections of class ${\cal C}^\kappa$ in $\Phi$ in the sense of \cite[p. 38]{Pa}.
 \item[-] $\extp^{d}_{\;M}$:  the bundle of forms of degree $d$   on $M$. 
 \item[-]  $\extp^{p,q}_{\;M}$:  the bundle of forms of bidegree $(p,q)$ on a complex manifold $M$.
 \item[-] $A^d(M,E)\edf\Gamma^\infty(M,\extp^{d}_{\;M}\otimes E)$,   $A^{p,q}(M,E)\edf\Gamma^\infty(M,\extp^{p,q}_{\;M}\otimes E)$. 
 \item[-] $A^d(M,E)_\kappa\edf\Gamma^\kappa(M,\extp^{d}_{\;U}\otimes E)$, $A^{p,q}(M,E)_\kappa\edf\Gamma^\kappa(M,\extp^{p,q}_{\;U}\otimes E)$.

\end{itemize}

\section{Statement of  results}\label{results}
  
\subsection{Gluing holomorphic bundles along a real hypersurface }
Let $U$ be a differentiable manifold, and let $S\subset U$ be a  closed  real smooth hypersurface.
\subsubsection{Gluing holomorphic bundles along a separating real  hypersurface}

 Let $\kappa\in[0,+\infty]$. We will use the notation ${\cal C}^\kappa$  for the usual  $k$-th differentiability class  when $\kappa\in\N\cup\{\infty\}$, and the  Hölder  class ${\cal C}^{[\kappa],\kappa-[\kappa]}$ when $ \kappa\not\in\N\cup\{\infty\}$ (see section  \ref{Ckappa}).  Suppose  that $S$ separates $U$, i.e.   $U\setminus S$ decomposes as a disjoint union $U\setminus S=U^-\cup U^+$ with $\bar U^\pm=U^\pm\cup S$. Therefore $\bar U^\pm$ are   manifolds with boundary and $\p\bar U^+=\p\bar U^-=S$. Let $E^\pm$ be a  ${\cal C}^\infty$ complex vector bundle of rank $r$ on $\bar U^\pm$  and let $E^\pm_S$ be  its restriction to $S$.   
 Let  $\upsilon: E^-_S\to E^+_S$ be a bundle isomorphism of class ${\cal C}^{\kappa+1}$  and  $E^\upsilon\edf E^-\coprod_\upsilon E^+$ the {\it topological} bundle obtained by gluing $E^\pm$ along $S$ via $\upsilon$. 
 
 Suppose now that  $U$ is a complex manifold and $E^\pm$   have been endowed with  Dolbeault operators 
 $$\delta^\pm: \Gamma^{\kappa+1}(\bar U^\pm, E^\pm )\to \Gamma^{\kappa}(\bar U^\pm, \extp^{0,1}_{\;\bar U^\pm}\otimes E^\pm )$$
 with coefficients in ${\cal C}^\kappa$ which satisfy the {\it formal} integrability condition $F_{\delta^{\pm}}=0$, where $F_{\delta^{\pm}}$ is the $\End(E^\pm)$-valued (0,2)-form on $\bar U^\pm$ associated with $\delta^2$. When $\kappa\in[0,1)$, $F_{\delta^{\pm}}$ is a distribution supported by $\bar U^\pm$ in the sense of \cite[section I.1]{Me}, see section  \ref{FIC} in the appendix.

From now on throughout this section we will suppose $\kappa\in(0,+\infty]\setminus\N$. This condition is required  in several crucial arguments where we make use of the standard elliptic regularity for Hölder spaces, or of the Hölder version of the Newlander-Nirenberg theorem for principal bundles (see \cite{Te2} and section \ref{ACSsection} in this article).
 
 \begin{thr}\label{new-main}
 Let $\delta^\pm$ be a formally integrable Dolbeault operator with coefficients in ${\cal C}^\kappa$ on $E^\pm$ and let $\hg_\pm$ be the corresponding holomorphic structure on the underlying ${\cal C}^{\kappa+1}$ bundle of the restrictions $E^\pm_{U^\pm}$ to $U^\pm$. Suppose that the tangential Cauchy-Riemann operators $\delta^\pm_S$
  induced by $\delta^\pm$ agree via  $\upsilon$. Then 
  \begin{enumerate}
 \item  The topological bundle $E^\upsilon$ on $U$ admits a unique holomorphic reduction	$\hg^\upsilon$ extending $\hg^\pm$.
 \item For any local $\hg^\upsilon$-holomorphic section $U\stackrel{\hb{\tiny open}}{\supset}V\textmap{\sigma} E^\upsilon$, we have 
 $$\sigma|_{V\cap\bar U^\pm}\in \Gamma^{\kappa+1}(V\cap\bar U^\pm,E^\pm),$$
 i.e.  the restrictions $\sigma|_{V\cap\bar U^\pm}$ of $\sigma$ are of class ${\cal C}^{\kappa+1}$ up to the boundary.
 \end{enumerate}
 \end{thr}

 Therefore, although the gluing bundle isomorphism $\upsilon$ is supposed to be only  of class ${\cal C}^{\kappa+1}$ and the required compatibility condition concerns only the tangential operators $\delta^\pm_S$, we can glue together the two {\it formally}  holmorphic bundles $(E^\pm, \delta^\pm)$  along $S$ via $\upsilon$ and obtain a holomorphic bundle on $U$.

 \begin{re}
For a Dolbeault operator $\delta$ on a bundle $E^+$ on a manifold with boundary $\bar U^+$,   the   formal integrability condition $\delta^2=0$ does {\it not} imply integrability (existence of local frames solving the $\delta$-equation) at non pseudo-convex boundary points. In \cite{Te} we gave an example  of a bundle $E^+$ on a compact manifold $\bar U^+$ with pseudo-concave boundary with the property that a {\it generic} formally integrable Dolbeault operator on $E^+$ is   integrable at {\it no} boundary point. Theorem \ref{new-main} shows that the compatibility condition required in its hypothesis  implies local integrability of  {\it both} $\delta^\pm$ at {\it all}  points of $S$, without any pseudo-convexity condition.
	
 \end{re}

 Theorem \ref{new-main} gives:
 \begin{co}\label{loc-free-sheaves} Under the assumptions of Theorem \ref{new-main}, the  ${\cal O}_U$-module ${\cal E}$ defined by
\begin{equation}\label{first-sheaf} 
W\mapsto \left\{\begin{pmatrix}f^-\\f^+\end{pmatrix} \in  \begin{array}{c}\Gamma^0 (W\cap\bar U^-  ,E^-)\\\times\\ \Gamma^0 (W\cap\bar U^+ ,E^+)\end{array}\,\vline  \begin{array}{c} f^+|_{W\cap S}=\upsilon f^-|_{W\cap S}\,, \\ f^\pm  \hb{ is    $\hg^\pm$-holomorphic on } W\cap U^\pm\end{array}\hspace{-2mm}\right\}\phantom{.} 
\end{equation}
is  locally free  of rank $r$, and coincides with the apparently smaller sheaf	
\begin{equation}\label{second-sheaf}
\hspace{-2mm}W\mapsto \left\{\begin{pmatrix}f^-\\f^+\end{pmatrix} \in \begin{array}{c} \Gamma^{\kappa+1} (W\cap\bar U^- ,E^-)\\ \times \\\Gamma^{\kappa+1} (W\cap\bar U^+ ,E^+)\end{array}\,\vline   \begin{array}{c} f^+|_{W\cap S}=\upsilon f^-|_{W\cap S}\,, \\ f^\pm  \hb{ is    $\hg^\pm$-holomorphic on } W\cap U^\pm\end{array}\hspace{-2mm}  \right\}.
\end{equation}
 \end{co}

\vspace{2mm}

 \subsubsection{Gluing holomorphic bundles along  an oriented real hypersurface}\label{non-sep-section}

 Theorem \ref{new-main}, Corollary \ref{loc-free-sheaves} can be  extended to oriented, non-necessarily separating, non-necessarily connected, real hypersurfaces. Let $U$ be a complex manifold and $S\subset U$ be a closed, {\it oriented} real hypersurface. The  normal bundle $n_S\edf T_{U|S}/T_S$   of $S$ in $U$ comes with a distinguished orientation induced by the complex orientation of $U$ and the fixed orientation of $S$. Let $0_{n_S}$ be the zero section of $n_S$. The quotient $\widehat S\edf (n_S\setminus 0_{n_S})/\R_{>0}$ is a trivial  double cover of $S$,  so it decomposes as a disjoint union $\widehat S=S^+\cup S^-$, where $S^\pm$ are identified with $S$ via the cover map $\widehat S\to S$. Therefore we have an obvious identification $b:S^-\textmap{\simeq} S^+$.  The union 
 $$
 \widehat U_S\edf (U\setminus S)\cup \widehat S
 $$
 has a canonical structure of a complex manifold with boundary whose boundary is 
 $$\p \widehat U_S=\widehat S=S^-\cup S^+,$$
 and comes with an obvious surjective smooth map $p_S^U:\widehat U_S\to U$ extending the biholomorphic identification $\widehat U_S\setminus\widehat{S}=U\setminus S$; it will be called {\it the manifold with boundary obtained by cutting $U$ along $S$} (see fig. \ref{hatUS}).
 \begin{figure}[h]
\includegraphics[scale=0.5]{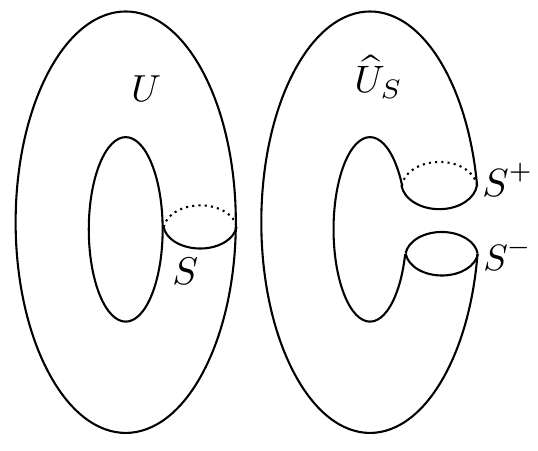}
\caption{$U$ and $\widehat U_S$.}
\label{hatUS}
\end{figure}
 In the special case considered above (when $S$ separates $U$) we have $\widehat U_S=\bar U^-\coprod \bar U^+$.

   Let $E$ be a  complex vector bundle of class ${\cal C}^\infty$ on $\widehat U_S$. We will denote by $E_{U\setminus S}$ the restriction of $E$ to $\widehat U_S\setminus\widehat{S}=U\setminus S$. Let $\upsilon:E_{S^-}\to b^*(E_{S^+})$ be a  bundle isomorphism of class ${\cal C}^{\kappa+1}$. 
 Identifying $E_{S^-}$ with $E_{S^+}$ via $\upsilon$ we obtain a topological bundle $E^\upsilon$ on $U$ whose pull back to $\widehat U_S$ is tautologically identified with $E$. \vspace{1mm}
 	
 Taking into account that Theorem \ref{new-main} has a local character with respect to $S$, we obtain:
 \begin{thr}
 \label{new-main-non-sep} Let $E$ be a ${\cal C}^\infty$ complex vector bundle on $\widehat U_S$, $\delta$ a formally integrable Dolbeault operator with coefficients in ${\cal C}^\kappa$ on $E$, and $\hg$ the corresponding holomorphic structure on the underlying ${\cal C}^{\kappa+1}$ bundle of $E_{U\setminus S}$. Let $\upsilon:E_{S^-}\to b^*(E_{S^+})$ be a   bundle isomorphism of class ${\cal C}^{\kappa+1}$.
 Suppose that the tangential Cauchy-Riemann operators $\delta^\pm_{S^\pm}$
  induced by $\delta$ agree via  $\upsilon$. Then 
  \begin{enumerate}
 \item  The topological bundle $E^\upsilon$ on $U$ admits a unique holomorphic reduction	$\hg^\upsilon$ extending $\hg$.
 \item For any local $\hg^\upsilon$-holomorphic section $U\stackrel{\hb{\tiny open}}{\supset}V\textmap{\sigma} E^\upsilon$, we have 
 $$\hat\sigma\edf\sigma\circ  p^V_S\in \Gamma^{\kappa+1}(\widehat{V}_{V\cap S},E),$$
 i.e.  the pull back $\hat \sigma$ of $\sigma$ via $p^V_S$ is of class ${\cal C}^{\kappa+1}$ up to the boundary.
 \end{enumerate}
  
 \end{thr}
 This can also be reformulated in terms of sheaves:
 \begin{co}\label{loc-free-sheaves-non-sep} Under the assumptions of Theorem \ref{new-main-non-sep}, the  ${\cal O}_U$-module ${\cal E}$ defined by
\begin{equation*} 
W\mapsto \bigg\{f \in  \Gamma^0 (\widehat W_{W\cap S},E)\,\vline  \begin{array}{c} (f|_{(S\cap W)^+})\circ b|_{(W\cap S)^-}=\upsilon\circ  (f|_{(W\cap S)^-})\,, \\ f  \hb{ is    $\delta$-holomorphic on } \widehat W_{W\cap S}\setminus \widehat S\end{array}\hspace{-1mm}\bigg\} 
\end{equation*}
is  locally free  of rank $r$, and coincides with the apparently smaller sheaf	
\begin{equation*}
W\mapsto \bigg\{f \in  \Gamma^{\kappa+1} (\widehat W_{S\cap W},E)\,\vline   \begin{array}{c} (f|_{(W\cap S)^+})\circ b|_{(W\cap S)^-}=\upsilon\circ  (f|_{(W\cap S)^-})\,, \\ f  \hb{ is    $\delta$-holomorphic on } \widehat W_{W\cap S}\setminus \widehat S\end{array}\hspace{-1mm}\bigg\}.
\end{equation*}
 \end{co}

 Let now $G$ be an arbitrary complex Lie  group. In the presence of a  principal $G$-bundle $P$ on $\widehat U_S$ and a    bundle isomorphism  $\upsilon:P_{S^-}\to b^*(P_{S^+})$ of class ${\cal C}^{\kappa+1}$, one can define the topological bundle $P^\upsilon$  as in the vector bundle case. 
Let  $\Tg_{P_{S^\pm}}\subset T_{P_{S^\pm}}$ be the pull-back of the canonical almost complex distribution $\Tg_S\edf T_S\cap J_U(T_S)$ of $S$. Using the definitions and notations explained in section  \ref{ACSsection} (see also \cite{Te2}) we have: 
 \begin{thr}
 \label{new-main-non-sep-G}
 Let $p:P\to \widehat U_S$ be a principal $G$-bundle on $\widehat U_S$ and $J$  a formally integrable bundle almost complex structure (bundle ACS) of class ${\cal C}^\kappa$ on $P$. Let $\upsilon:P_{S^-}\to  b^*(P_{S^+})$ be a   bundle isomorphism of class ${\cal C}^{\kappa+1}$. Suppose that the tangential almost complex structures $J_{S^\pm}$  induced by $J$ on the distributions $\Tg_{P_{S^\pm}}$ agree via $\upsilon$. Then
 \begin{enumerate}
\item  The topological bundle $P^\upsilon$ admits a unique holomorphic reduction	$\hg^\upsilon$ extending the holomorphic structure $\hg$ induced by $J$ on $P_{X\setminus S}$.
 \item The pull-back $\hat \tau$ of any local $\hg^\upsilon$-holomorphic section $\tau:V\to P^\upsilon$ is of class ${\cal C}^{\kappa+1}$ up to the boundary.
	
 \end{enumerate}

 \end{thr}

 Note that $\upsilon$  can be regarded as a section in a locally trivial fiber bundle over $S^-$. The ${\cal C}^{\kappa+1}$ condition on $\upsilon$ in Theorem \ref{new-main-non-sep-G} is meant in the sense of \cite[p. 38]{Pa}.
 
  \begin{re}
If  $G$ is a closed complex subgroup of $\GL(r,\C)$, any bundle ACS on $P$ will induce a Dolbeault operator  on the associated rank $r$ vector bundle, and the  compatibility condition ``the bundle ACS $J_{S^\pm}$ agree via $\upsilon$" required in Theorem \ref{new-main-non-sep-G} can be replaced by a compatibility condition for tangential Cauchy-Riemann operators as in Theorem \ref{new-main-non-sep}. We preferred  a  formulation which is general and intrinsic  in  terms of abstract complex Lie groups $G$ and principal $G$-bundles. \end{re}

 In the special case when $S$ separates $U$, we obtain as a special case the following generalization of Theorem \ref{new-main}:
 \begin{thr}\label{new-main-sep-G}  Let $P^\pm$ be a  ${\cal C}^\infty$ principal $G$-bundle  on $\bar U^\pm$  and let $P^\pm_S$ be  its restriction to $S$. Let $\upsilon:P^-_{S}\to P^+_{S}$ be a   bundle isomorphism of class ${\cal C}^{\kappa+1}$.
 Let $J^\pm$ be a formally integrable bundle ACS of class ${\cal C}^\kappa$ on $P^\pm$ and let $\hg_\pm$ be the corresponding holomorphic structure on the underlying ${\cal C}^{\kappa+1}$ bundle of the restrictions $P^\pm_{U^\pm}$ to $U^\pm$. Suppose that the tangential  almost complex structures $J^\pm_S$
  induced by $J^\pm$ on $\Jg_{P^\pm_S}$ agree via  $\upsilon$. Then 
  \begin{enumerate}
 \item  The topological bundle $P^\upsilon$ on $U$ admits a unique holomorphic reduction	$\hg^\upsilon$ extending $\hg^\pm$.
 \item For any local $\hg^\upsilon$-holomorphic section $U\stackrel{\hb{\tiny open}}{\supset}V\textmap{\tau} P^\upsilon$, we have 
 $$\tau|_{V\cap\bar U^\pm}\in \Gamma^{\kappa+1}(V\cap\bar U^\pm,P^\pm),$$
 i.e.  the restrictions $\tau|_{V\cap\bar U^\pm}$ of $\tau$ are of class ${\cal C}^{\kappa+1}$ up to the boundary.
 \end{enumerate}	
 \end{thr}

 \begin{re}\label{dim=1}
 The compatibility conditions on the tangential Cauchy-Riemann operators or tangential almost complex structures in Theorems \ref{new-main}, \ref{new-main-non-sep}, \ref{new-main-non-sep-G}, \ref{new-main-sep-G} are void when $\dim_\C(U)=1$ (i.e. when $U$ is a Riemann surface).  \end{re}

 \subsection{First applications: The Riemann-Hilbert problem}
 
The first applications of Theorems \ref{new-main}, \ref{new-main-non-sep}, \ref{new-main-non-sep-G} and their corollaries concern generalizations of the classical Riemann-Hilbert problem.
 
 \subsubsection{The Riemann-Hilbert problem on $\P^1_\C$} \label{RHonP1-section}
 
  We first illustrate Theorem \ref{new-main} in a  simple special case: let $U=\P^1_\C=\C\cup\{\infty\}$ and $U^+$ ($U^-$) be the connected component of $\P^1_\C\setminus S$ which contains (does not contain) $\infty$, where    $S\subset\C$ is a compact, connected smooth  curve. Let $\upsilon:S\to \GL(r,\C)$ be a  map of class ${\cal C}^{\kappa+1}$. Such a pair $(S,\upsilon)$ is the input data of a   Riemann-Hilbert   problem.
 \vspace{2mm}
 
 In the renowned  book chapter  \cite[Kapitel X. Riemanns Probleme in der Theorie der Funktionen einer komplexen Veränderlichen]{Hil}, Hilbert states and studies the following problem\footnote{In Hilbert's original problem, as stated in loc. cit,    $\upsilon$ is supposed to be  of class ${\cal C}^2$ and $S$ real analytic.}
 :
 \begin{pb}  [Riemann-Hilbert]  \label{RH}   Find the space of pairs $(f^-,f^+)$ of continuous  maps $f^\pm:\bar U^\pm\to \C$ which are holomorphic on $U^\pm$ and whose restrictions to $S$ satisfy the  condition $f^+_S=\upsilon f^-_S$.	
 \end{pb}
 
Hilbert also states and studies a meromorphic version of the problem: $f^+$ is still required to be holomorphic, but $f^-$ is allowed to be meromorphic with poles in $U^-$. 
%
%
%
%
 Several authors have stated interesting versions of the  Riemann-Hilbert problem; for instance in \cite{Ple}, \cite{Bo} one can find:
\begin{pb}\label{Ple}
Find the space of solutions $(f^-,f^+)$ of the Riemann-Hilbert problem with $f^-$ holomorphic on $U^-$ and $f^+$ meromorphic on $U^+$ with a single pole with prescribed singularity type (Laurent  coefficients of non-positive index) at $\infty$. 	
\end{pb}
Other authors (see for instance \cite{It}) are interested in matrix factorisation problems of the form:
\begin{pb}\label{It}    Find the space of pairs $(Y^-,Y^+)$  of  continuous  maps $Y^\pm:\bar U^\pm\to \gl(r,\C)$ which are holomorphic on $U^\pm$,  whose restrictions to $S$ satisfy the  condition $Y^+_S=  \upsilon Y^-_S$ and such that $Y_+(\infty)=I_r$.	
 \end{pb}
 
 More generally, let  $\rho:G\to \GL(V)$ be a representation of $G$ on  a finitely dimensional complex vector space $V$, $m\in\Z$,   and  $\gamma\in V[z]=\sum_{s\geq 0}\gamma_s z^s$ a $V$-valued polynomial.  Put $d\edf\deg(\gamma)\in\Z_{\geq -1}$ (we use the convention $\deg(\gamma)=-1$ for $\gamma=0$). 
Let $\zeta:\P^1_\C\setminus\{0\}\to \C$ be the standard coordinate of $\P^1_\C$ around $\infty$;  replacing formally $z$ by $\zeta^{-1}$ in the expression of $\gamma$, we obtain a 
  Laurent polynomial $\tilde\gamma=\sum_{s=-d}^0\tilde\gamma_s\zeta^s\in V[\zeta^{-1}]$  with $\tilde\gamma_s=\gamma_{-s}$. Regarding $\infty$ as an effective divisor on $\P^1_\C$, $\tilde\gamma$   can be interpreted as an  element of $H^0({\cal O}(d\infty)_{(d+1)\infty}\otimes V)$ with empty zero locus on the effective divisor  $(d+1)\infty$. Let $\upsilon:S\to G$ be a map of class ${\cal C}^{\kappa+1}$. We ask:
 \begin{pb}\label{rho}    Find the space of pairs $(Y^-,Y^+)$ of continuous maps 
 $$Y^-:\bar U^-\to V,\ Y^+:\bar U^+\setminus\{\infty\}\to V$$
   with $Y^-$ holomorphic on $U^-$, $Y^+$ holomorphic on $U^+\setminus \{\infty\}$ such  that $Y^+_S=\rho(\upsilon) Y^-_S$ and %
  $$\lim_{z\to\infty}(z^{d-m}Y^+(z)-\gamma(z))=0.\eqno{(C_\infty)}$$	
 \end{pb} 
 
The latter condition implies that $\infty$ is a non-essential singularity of $Y_+$; it is equivalent to the following condition on the Laurent series $\sum_{k\in\Z}b_k\zeta^k$ of $Y_+$ at $\infty$:
$$
b_s=\tilde \gamma_{m-d+s} \hb { for } s\leq d-m .\eqno{(C'_\infty)}
$$
Therefore the analytic condition $(C_\infty)$ has a purely complex geometric interpretation:
\begin{itemize}
\item For $d\geq 0$ (i.e. $\gamma\ne 0$)	  it requires that $Y^+$ extends as a section of the sheaf ${\cal O}(m\infty )\otimes_\C V$   on $U^+$ whose image  in   $H^0({\cal O}(m\infty)_{(d+1)\infty}\otimes V)$
%
via the obvious morphism is $z^{m-d}\otimes\tilde\gamma$.  
\item For $d=-1$ (i.e. $\gamma=0$) it just requires that $\tilde Y^+$ extends as a section of the sheaf ${\cal O}(m\infty )\otimes_\C V$   on $U^+$. This is the ``homogenous case", the case when the solution space is naturally a vector space.
\end{itemize}

Hilbert's original problem is obtained taking $\rho$  to be the canonical representation of $\GL(r,\C)$ on $\C^r$, $m=0$, and $\gamma=0$. Problem \ref{Ple} is obtained with the same $\rho$ taking $m=d$, and problem \ref{It} corresponds to the representation of $\GL(r,\C)$ on $\gl(r,\C)$ given by left multiplication, taking $m=0$ and $\gamma=$  the degree 0 polynomial $I_r$.

\vspace{2mm}

Let $P^\pm\edf \bar U^\pm\times G $ be the trivial $G$-bundle endowed with the standard (trivial) bundle ACS. A  map $\upsilon: S\to G$ of class ${\cal C}^{\kappa+1}$  can be regarded as a bundle isomorphism $P^-_S\to P^+_S$ of this class. By Theorem \ref{new-main-sep-G} and Remark \ref{dim=1}, for any such $\upsilon$ we have a well defined  holomorphic structure $\hg^\upsilon$ on the bundle $P^\upsilon$ over $\P^1_\C$. The obtained holomorphic bundle, which will still be denoted by $P^\upsilon$ to save on notations, comes with ${\cal C}^{\kappa+1}$-trivializations $\theta^\pm_\upsilon$ on $\bar U^\pm$ which are holomorphic on $U^\pm$. By Corollary \ref{loc-free-sheaves}, the  locally free sheaf ${\cal V}^\upsilon$ associated with the holomorphic vector bundle $P^\upsilon\times_\rho V$ is given by the {\it equivalent} formulae 
\begin{equation}\label{EGamma(W)}
\begin{split} 
  W&\mapsto\left\{\begin{pmatrix} f^-\\f^+\end{pmatrix} \in  \begin{array}{c} \Gamma^0(W\cap \bar U^-,V)\\ \times \\ \Gamma^0(W\cap \bar U^+,V)\end{array} \,\vline   \begin{array}{c} f^+|_S=\rho(\upsilon) f^-|_S,	\\
 f^\pm \hbox{ is holomorphic on }W\cap U^\pm\end{array}\hspace{-2mm}\right\}\phantom{.}\\
&= \left\{\begin{pmatrix} f^-\\f^+\end{pmatrix} \in \begin{array}{c} \Gamma^{\kappa+1}(W\cap \bar U^-,V)\\ \times \\ \Gamma^{\kappa+1}(W\cap \bar U^+,V)\end{array}\,\vline  \begin{array}{c} f^+|_S=\rho(\upsilon) f^-|_S,	\\
 f^\pm \hbox{ is holomorphic on }W\cap U^\pm\end{array}\hspace{-2mm}\right\}.
\end{split}
\end{equation}
Note also that  the trivialization $\theta^+$ induces  isomorphisms
$$H^0 ({\cal O}(d\infty)_{(d+1)\infty}\otimes V )\textmap{\simeq}H^0 ({\cal V}^\upsilon(d\infty)_{(d+1)\infty}),$$
so   $\tilde\gamma$ gives an element $\nu_{\tilde\gamma}^\upsilon\in H^0 ({\cal V}^\upsilon(d\infty)_{(d+1)\infty})$.
With these remarks we obtain:
\begin{co}\label{RH-P1}
Let $S\subset \C$ be a compact, connected smooth  curve. 
\begin{enumerate}
\item The map $\upsilon\mapsto (P^\upsilon,\theta^-_\upsilon,\theta^+_\upsilon)$ gives a bijection between the group ${\cal C}^{\kappa+1}(S,G)$ and the set of isomorphism classes of triples $(Q,\theta^-,\theta^+)$ consisting of a holomorphic principal $G$-bundle $Q$ on $\P^1_\C$ and   ${\cal C}^{\kappa+1}$-trivializations  $\theta^\pm$ of $Q$ on $\bar U^\pm$ which are holomorphic on $U^\pm$.	
\item  If $\gamma= 0$, the space of solutions of the  general Riemann-Hilbert Problem \ref{rho}  can be naturally identified with $H^0 (\P^1_\C, {\cal V}^\upsilon(m\infty))$.
\item  If $\gamma\ne 0$, the space of solutions of the  general Riemann-Hilbert Problem \ref{rho} is non-empty if and only if the image of $z^{m-d}\otimes\nu^\upsilon_{\tilde\gamma}$   via the connecting morphism %
\begin{equation*}
\begin{split}
\hspace*{12mm} H^0 ({\cal V}^\upsilon(m\infty)_{(d+1)\infty})=H^0\big(\P^1_\C, {\cal V}^\upsilon(m\infty)/{\cal V}^\upsilon(&(m-d-1)\infty)\big)\to \\ &\to H^1(\P^1,{\cal V}^\upsilon((m-d-1)\infty))	
\end{split}	
\end{equation*}
 vanishes. If this is the case, this space   has the natural structure of an affine space with model space $H^0\big(\P^1_\C, {\cal V}^\upsilon((m-d-1)\infty)\big)$, and can be naturally identified with the pre-image of $z^{m-d}\otimes\nu^\upsilon_{\tilde\gamma}$ via the natural morphism 
$$H^0(\P^1_\C, {\cal V}^\upsilon(m\infty))\to H^0 ({\cal V}^\upsilon(m\infty)_{(d+1)\infty}).
$$
\end{enumerate}
\end{co}
In particular, the space of solutions of the  original Riemann-Hilbert problem (of Problem \ref{RH}) is naturally isomorphic to the space $H^0(\P^1_\C, {\cal V}^\upsilon)$ associated with the canonical representation of $\GL(r,\C)$ on $\C^r$.
Hilbert's results   \cite[Sätze 27-30]{Hil} follow easily from Corollary \ref{RH-P1}.  Taking into account  formula (\ref {EGamma(W)}) we also obtain the following general regularity result: 
\begin{re}\label{regularity}
Any solution of a Riemann-Hilbert problem  with  $\upsilon$ of class ${\cal C}^{\kappa+1}$ is also of class ${\cal C}^{\kappa+1}$ up to the boundary. 	
\end{re}

By Grothendieck's  classification theorem \cite{Gro}, the sheaf  ${\cal V}^\upsilon$  splits as a direct sum of invertible sheaves, so we have 
$
{\cal V}^\upsilon\simeq \bigoplus_{j=1}^r {\cal O}(n_j)
$
with $n_j\in\Z$ and $\sum_{j=1}^r n_j= \deg({\cal V}^\upsilon)$. For the canonical representation of $\GL(r,\C)$ on $\C^r$ we have $\deg({\cal V}^\upsilon)=-\deg(\det(\upsilon))$. 
Therefore, once in possession of the complex geometric objects $(Q,\nu^\upsilon_\gamma)$ associated  with the input data $(\rho,\upsilon,\gamma)$, the corresponding Riemann-Hilbert problem can be approached  using elementary complex geometric methods. 
For instance,  one can  easily give examples of such data  for which the space of solutions is empty and, at least for small $r$ and standard representations, one can compute all  possible dimensions  of the space of solutions for a given  Grothendieck decomposition of  ${\cal V}^\upsilon$.

 A difficulty remains: make the bijection provided by  Corollary \ref{RH-P1} effective, i.e., for given $\upsilon$, determine {\it explicitly} a Grothendieck direct sum decomposition of ${\cal V}^\upsilon$ and the "position`` of $\nu_{\tilde\gamma}^\upsilon$ with respect to the summands. 
\subsubsection{The Riemann-Hilbert problem on Riemann surfaces}\label{RHonRiemannSurf-section}

The formalism and the results of section \ref{non-sep-section} allows us to formulate and approach with complex geometric methods a very general Riemann-Hilbert problem: Let $X$ be a closed Riemann surface, $S\subset X$ an {\it arbitrary} (non-necessarily connected, non-necessarily separating) closed,  oriented real 1-dimensional submanifold,  $\upsilon:S^-\to G$  a  map of class ${\cal C}^{\kappa+1}$, and $\rho:G\to \GL(V)$ a representation.
%
%
%
Let also $D$, $\Delta$ be   divisors on $X\setminus S$, with $\Delta\geq 0$, and fix a section $\gamma\in H^0({\cal O}(D)_\Delta\otimes V)$ which is nowhere vanishing\footnote{If $\gamma$ has non-empty zero locus on $\Delta$, one will obtain an equivalent problem associated with   a smaller pair $(D,\Delta)$.} on $\Delta$. 
\begin{pb}\label{RHXD}
Find the space of  meromorphic maps $Y:\widehat X_S\setminus\widehat{S} \dasharrow V$ which extend continuously around $\widehat S$, such that: 
\begin{enumerate}
	\item\label{Gamma-comp-5} $Y|_{S^+}\circ b=\rho(\upsilon) Y|_{S^-}$,
	\item   Via the obvious identification $\widehat X_S\setminus\widehat{S}=X\setminus S$, $Y$   extends as a section of ${\cal O}(D)\otimes V$, and the image of this extension in $H^0({\cal O}(D)_\Delta\otimes V)$ via the obvious morphism, is $\gamma$.
\end{enumerate}
\end{pb}
The ``homogenous case"  corresponds to the case $\Delta=0$ (the empty divisor).
For $D=\Delta=0$ one just obtains the space of continuous maps $Y:\widehat X_S\to V$ which satisfy  the $\upsilon$-compatibility condition (\ref{Gamma-comp-5}) and are holomorphic on $X\setminus S$.

Taking into account Theorem \ref{new-main-non-sep}, Corollary \ref{loc-free-sheaves-non-sep} and Remark  \ref{dim=1}, we obtain, as in Corollary \ref{RH-P1}, a map
$$
\upsilon\mapsto (P^\upsilon,\theta^\upsilon)
$$
which gives a bijection  between  the group ${\cal C}^{\kappa+1}(S^-,G)$ and the set of isomorphism classes of pairs $(Q,\theta)$, where $Q$ is a holomorphic principal  $G$-bundle   on the Riemann surface $X$, and $\theta$ is a  trivialization of class ${\cal C}^{\kappa+1}$ of the pull-back $(p_S^X)^*(Q)$ of $Q$ to $\widehat{X}_S$, which is holomorphic on $\widehat{X}_S\setminus\widehat{S}$.  We define the locally free sheaf ${\cal V}^\upsilon$  as in the previous section, and note that, via the trivialization $\theta^\upsilon$, $\gamma$ gives an element $\nu_\gamma^\upsilon\in H^0({\cal V}^\upsilon(D)_\Delta)$.
Using the explicit formulae for the sheaf ${\cal V}^\upsilon$ given by Corollary \ref{loc-free-sheaves-non-sep}, we obtain the following complex geometric interpretation of the space of solutions of the general Riemann-Hilbert Problem \ref{RHXD}:
\begin{co}\label{RH-X}
The space of solutions of   Problem \ref{RHXD} is non-empty if and only if the image of $ \nu^\upsilon_\gamma$   via the connecting morphism 
$$H^0({\cal V}^\upsilon(D)_\Delta)=H^0\big(X, {\cal V}^\upsilon(D)/{\cal V}^\upsilon(D-\Delta)\big)\to H^1(X,{\cal V}^\upsilon(D-\Delta))$$
 vanishes. If this is the case, this space   has the natural structure of an affine space with model space $H^0\big(X, {\cal V}^\upsilon(D-\Delta))$, and can be identified with the pre-image of $ \nu^\upsilon_\gamma$ via the natural morphism 
$H^0\big(X, {\cal V}^\upsilon(D)\big)\to H^0({\cal V}^\upsilon(D)_\Delta)$.
\end{co}

An interesting special case:
\begin{ex}\label{RHOnElliptic}
Let $X$ be a Riemann surface of genus 1, and $S$  a non-separating circle, as in Fig. \ref{hatUS}. We can assume that $X=\C^*/\langle  \alpha\rangle$, where $\alpha\in \C^*$ with $|\alpha|<1$, and $S$ is the image in $X$  of  $\Sigma\edf\p\bar D$, where $\bar D\subset\C$ is a smooth compact disk such that $0\in D$ and $\alpha\bar D\subset D$. Therefore, we can identify $\widehat X_S$, the Riemann surface with boundary obtained by cutting $X$ along $S$, with the annulus $\bar \Omega\edf\bar D\setminus \alpha D$ whose boundary is $\alpha \Sigma\cup \Sigma$. 
In this case the unknown of the Riemann-Hilbert  Problem \ref{RHXD} is a meromorphic map $Y:\Omega\dasharrow V$ on the open annulus $\Omega\edf D\setminus \alpha\bar D$ extending continuously around $\partial\bar\Omega$ and satisfying the compatibility condition:
$$\forall z\in \Sigma,\ Y(\alpha z)=\rho(\upsilon) Y(z).$$

Note that the holomorphic vector bundles on elliptic curves have been classified \cite{At}, so Corollary \ref{RH-X} allows one (in principle) to solve any Riemann-Hilbert problem of the considered type on an elliptic curve.
	
\end{ex}

\subsubsection{The Riemann-Hilbert problem in arbitrary dimension} \label{RHXnD-section}
  
Theorem \ref{new-main-non-sep-G} suggests a natural generalization of the  Riemann-Hilbert problem in arbitrary dimension (again for any complex Lie group $G$), and also a general complex geometric method to approach   it.

 Let $X$ be a connected, closed complex manifold and $S\subset X$ a general (non-necessarily connected, non necessarily separating) closed, oriented real hypersurface. Let $P$ be a differentiable principal $G$-bundle on the manifold with boundary $\widehat X_S$ obtained by cutting $X$ along $S$ (see section \ref{non-sep-section}), and let $J$ be a formally integrable bundle ACS on $P$ (see sections \ref{ACSsection}, \ref{FIC}).  Therefore, compared with previous generalizations, we start with an arbitrary, not necessarily trivial, formally holomorphic principal $G$-bundle $(P,J)$ on $\widehat X_S$.
 
  Let $\upsilon:P_{S^-}\to b^*(P_{S^+})$ be a   bundle isomorphism of class ${\cal C}^{\kappa+1}$.
 Suppose that the following compatibility condition holds: 
\begin{cond} The tangential almost complex structures $J_{S^\pm}$  induced by $J$ on the distributions $\Tg_{P_{S^\pm}}\subset T_{P_{S^\pm}}$  agree via  $\upsilon$. 
 \end{cond} 
 By Theorem \ref{new-main-non-sep-G} we obtain a holomorphic principal $G$-bundle $P^\upsilon$ on $X$  whose pull back $p_S^{X*}(P^\upsilon)$ to $\widehat X_S$ comes with a tautological    bundle isomorphism $P\stackrel{\theta^\upsilon}{\to} p_S^{X*}(P^\upsilon)$ of class ${\cal C}^{\kappa+1}$ which is   holomorphic  on  $\widehat X_S\setminus\widehat S$.  The map
$$
\upsilon\mapsto (P^\upsilon, \theta^\upsilon)
$$
defines a bijection between the set of     bundle isomorphisms $\upsilon:P_{S^-}\to b^*(P_{S^+})$ of class ${\cal C}^{\kappa+1}$ satisfying the above   compatibility condition on $S$ and the set of isomorphism classes of pairs $(Q,\theta)$ consisting of a holomorphic  $G$-bundle $Q$ on the   closed complex manifold $X$, and a   bundle isomorphism $\theta: P\to p_S^{X*}(Q)$ of class ${\cal C}^{\kappa+1}$ which is   holomorphic  on  $\widehat X_S\setminus\widehat S$.

  Let $\rho:G\to \GL(V)$ be a representation of $G$ on $V$, and ${\cal V}$ the locally free sheaf on $X\setminus S$ corresponding to the associated bundle $P_{X\setminus S}\times_\rho V$. 
  Let $Z\subset X\setminus S$ a (possibly empty) compact complex subspace and let $\gamma\in H^0({\cal V}_Z)$.    

\begin{pb}\label{RHXnD}
Find the space of continuous sections $Y\in \Gamma^0(\widehat X_S,P_E\times_\rho V) $	which are holomorphic on $\widehat X\setminus \widehat S$ such that 
\begin{enumerate}
	\item\label{Gamma-comp-6} $Y|_{S^+}\circ b=\rho(\upsilon) Y|_{S^-}$,
	\item the image of $Y$ in  $H^0({\cal V}_Z)$ via the obvious morphism is $\gamma$.
\end{enumerate}
\end{pb}

 Denoting by ${\cal V}^\upsilon$  the locally free sheaf on $X$ associated with the holomorphic bundle $P^\upsilon\times_\rho V$, we obtain an obvious identification ${\cal V}^\upsilon|_{X\setminus S}={\cal V}$. With these notations and remarks we obtain: 

\begin{co}\label{RH-coro-n-dim}
Suppose the above compatibility condition holds. The space of solutions of   Problem \ref{RHXnD} is non-empty if and only if the image of $ \gamma$   via the connecting morphism %
$$H^0\big({\cal V}^\upsilon_Z)=H^0\big(X, {\cal V}^\upsilon /{\cal V}^\upsilon\otimes {\cal I}_Z)\to H^1(X,{\cal V}^\upsilon\otimes {\cal I}_Z)$$
 vanishes. If this is the case, this space   has the natural structure of an affine space with model space $H^0\big(X, {\cal V}^\upsilon\otimes {\cal I}_Z)$, and can be identified with the pre-image of $\gamma$ via the natural morphism 
$H^0 (X, {\cal V}^\upsilon\big)\to H^0\big({\cal V}^\upsilon_Z)$.	
\end{co}

\subsection{Gauge theoretical applications: Isomorphisms of moduli spaces}
\label{IsoModuliSectionState}

In this article by complex manifold with boundary we will always mean  a submanifold with boundary $\bar X$ of a complex manifold $U$. In other words, the complex manifolds with boundary  we consider have a collar neighborhood in the sense of \cite{HiNa}. 

\subsubsection{Isomorphisms of moduli spaces of framed vector bundles} \label{iso-moduli-vector-bundles-section}

For a ${\cal C}^\infty$ vector bundle $E$ on a {\it compact} complex manifold with boundary $\bar X$, we denote by ${\cal I}^\kappa_E$ the space of {\it formally integrable} Dolbeault operators with coefficients in ${\cal C}^\kappa$ on $E$ and we define the moduli space
$$
{\cal M}_{\p\bar X}(E)\edf {{\cal I}^\kappa_E}/{{\cal G}_{\p\bar X}^E},
$$  
where ${\cal G}_{\p\bar X}^E$ is the gauge group
$$
{\cal G}_{\p\bar X}^E\edf \{f\in \Gamma^{\kappa+1}(X,\GL(E))|\ f_{\p\bar X}=\id_{E_{\p\bar X}}\}.   $$

Let now $X$ be a connected, closed complex manifold, $S\subset X$  an oriented closed,   smooth real hypersurface,  and $E$  a ${\cal C}^\infty$ vector bundle   on $X$. In this case  ${\cal I}^\kappa_E$ will stand for the space of {\it integrable} Dolbeault operators with coefficients in ${\cal C}^\kappa$ on $E$ and we define the moduli space
$$
{\cal M}_S(E)\edf {{\cal I}^\kappa_E}/{{\cal G}_S^E},
$$  
where  ${\cal G}_{S}^E$ is the gauge group
$$
{\cal G}_S^E\edf \{f\in \Gamma^{\kappa+1}(X,\GL(E))|\ f_S=\id_{E_S}\}.   $$
We denote by $\widehat E$   the pull-back of $E$ to the manifold with boundary $\widehat X_S$ obtained by cutting $X$ alongs $S$  via the canonical map $p^X_S:\widehat X_S\to X$ (see section \ref{non-sep-section}). In the special case when $S$ separates $X$, $\widehat X_S$ reduces to the disjoint union $\bar X^-\coprod \bar X^+$ of the corresponding pieces, and $\widehat E$ reduces to the the disjoint union $E^-\coprod E^+$ of the restrictions  $E^\pm\edf E|_{\bar X^\pm}$. The bundle $\widehat E$ comes with a  canonical   bundle isomorphism  
$\upsilon: \widehat E_{S^-}\to b^*(\widehat E_{S^+})$
of class  ${\cal C}^\infty$ induced by the obvious isomorphisms $\widehat E_{S^\pm}\to E_S$ which cover the identifications $S^\pm\stackrel{\simeq}{\to} S$.

\begin{dt}\label{descendable}
A formally integrable Dolbeault operator $\dg\in  {\cal I}^\kappa_{\widehat E}$ will be called descendable, if the tangential Cauchy-Riemann operators $\dg_{S^\pm}$ on $\widehat E_{S^\pm}$  agree via $\upsilon$ (are $\upsilon$-compatible). 	
\end{dt}
The pull-back $\widehat \delta$ to $\widehat E$ of any  integrable Dolbeault operator $\delta\in{\cal I}^\kappa_E$ is obviously descendable. 
 Let ${\cal I}^\kappa_{\widehat{E}\downarrow} \subset {\cal I}^\kappa_{\widehat{E}}$ 
 be the  (obviously ${\cal G}_{\p\widehat{X}_S}^{\widehat{E}}$-invariant) 
 subspace of descendable formally integrable Dolbeault operators on $\widehat E$, and let 
${\cal M}^\downarrow_{\p\widehat{X}_S}(\widehat E)$ be the corresponding closed subspace of ${\cal M}_{\p\widehat{X}_S}(\widehat E)$. In the special case when $S$ separates $X$, we have ${\cal I}^\kappa_{\widehat E}={\cal I}^\kappa_{E^-}\times{\cal I}^\kappa_{E^+}$, where  $E^\pm\edf E_{\bar X^\pm}$, and a pair $(\delta^-,\delta^+)$ is descendable if and only if the equality $\delta^-_S=\delta^+_S$ holds in the space, denoted ${\cal C}$, of Cauchy-Riemann operators with coefficients in ${\cal C}^\kappa$ on $E_S$.

\begin{thr}\label{iso-moduli-vector-bundles-th}
Suppose $\kappa\in(0,+\infty)\setminus\N$. With the above notations and assumptions, the pull-back map $\delta\mapsto \widehat \delta$ induces a homeomorphism
$$
{\cal M}_S(E)\to {\cal M}^\downarrow_{\p\widehat{X}_S}(\widehat E).
$$
In the special case when $S$ separates $X$, the restriction map $\delta\mapsto(\delta^-,\delta^+)$ induces a homeomorphism
$$
{\cal M}_S(E)\to {\cal M}_{\p\bar X^-}(E^-)\times_{\cal C}{\cal M}_{\p\bar X^+}(E^+).
$$	
\end{thr}

\subsubsection{Isomorphisms of moduli spaces of framed principal bundles}
\label{Iso-Moduli-G-section}

Let $G$ be a complex Lie group.  With the notations and under the assumptions of  section \ref{iso-moduli-vector-bundles-section} we replace:
 \begin{itemize}
 \item[-] $E$ by a ${\cal C}^\infty$ principal $G$ bundle $P$ on $X$ ($\bar X$).
 \item[-] ${\cal I}^\kappa_E$ by the space ${\cal I}^\kappa_P$ of (formally) integrable  bundle ACS of class ${\cal C}^\kappa$ on $P$.
 \item [-]	${\cal G}_S^E$ (${\cal G}_{\p\bar X}^E$) by respectively the gauge groups 
 \begin{equation*}
 \begin{split}
 {\cal G}_S^P&\edf \{f\in \Gamma^{\kappa+1}(X,P\times_\iota G)|\ f_S=\id\},\\
 {\cal G}_{\p\bar X}^P&\edf \{f\in \Gamma^{\kappa+1}(\bar X,P\times_\iota G)|\ f_{\p\bar X}=\id\}.
 \end{split}	
 \end{equation*}
 \item[-] ${\cal M}_S(E)$ (${\cal M}_{\p\bar X}(E)$) by respectively the moduli spaces 
 $${\cal M}_S(P)\edf {\cal I}^\kappa_P/{\cal G}_S^P, \hb{ respectively } {\cal M}_{\p\bar X}(P)\edf {\cal I}^\kappa_P/{\cal G}_{\p\bar X}^P.$$
  \end{itemize}

We also replace $\widehat E$ by its pull back $\widehat P$   to $\widehat X_S$, and, if $S$ separates $X$, we replace the restrictions $E^\pm$ by $P^\pm\edf P_{\bar X^\pm}$. 

In this principal bundle framework we also have a canonical   bundle isomorphism $\upsilon:\widehat P_{S^-}\to b^*(\widehat P_{S^+})$ of class  ${\cal C}^\infty$. An element $\Jg\in  {\cal I}^\kappa_{\widehat P}$ will be called descendable, if the induced tangential almost complex structures $\Jg_{S^\pm}$ on the distributions $\Tg_{\widehat{P}_{S^\pm}}\subset T_{\widehat{P}_{S^\pm}}$ (see Remark \ref{JS}) agree via $\upsilon$.  We denote by ${\cal I}^\kappa_{\widehat P\downarrow}\subset {\cal I}^\kappa_{\widehat P}$ the subspace of descendable formally integrable bundle ACS on ${\widehat P}$ and by ${\cal M}^\downarrow_{\p\widehat X_S}(\widehat P)$ its quotient by the gauge group ${\cal G}_{\p\widehat{X}_S}^{\widehat{P}}$. 
\begin{thr}\label{iso-moduli-G-bundles-th}
$\kappa\in(0,+\infty)\setminus\N$. With the above notations and assumptions, the pull-back map $J\mapsto \widehat J$ induces a homeomorphism
$$
{\cal M}_S(P)\to  {\cal M}^\downarrow_{\p\widehat X_S}(\widehat P).
$$
In the special case when $S$ separates $X$, the restriction map $J\mapsto(J_{P-},J_{P^+})$ induces a homeomorphism
$$
{\cal M}_S(P)\to {\cal M}_{\p\bar X^-}(P^-)\times_{\cal I}{\cal M}_{\p\bar X^+}(P^+).
$$
onto the fiber product of the moduli spaces ${\cal M}_{\p\bar X^\pm}(P^\pm)$ over the space ${\cal I}$ of almost complex structures of class ${\cal C}^\kappa$ on $\Tg_{P_S}$ which are $G$-invariant,  make the bundle epimorphism $\Tg_{P_S}\to \Tg_S$ $\C$-linear, and the parametrization of the $G$-orbits  pseudo-holomorphic.
	
\end{thr}

\begin{re}\label{kappa=infty}
Taking $\kappa=+\infty$ in Theorems \ref{iso-moduli-vector-bundles-th}, \ref{iso-moduli-G-bundles-th} we still obtain  homeomorphisms provided any moduli space ${\cal M}$ intervening in these theorems (but constructed with objects of class ${\cal C}^\infty$) is endowed with the initial topology  associated with the family of maps  $({\cal M}\to {\cal M}^\kappa)_{\kappa\in(0,+\infty)\setminus\N}$; here ${\cal M}^\kappa$ stands for the similar moduli space constructed using objects of class ${\cal C}^\kappa$.  
\end{re}

As pointed out in Remark \ref{dim=1}, the required compatibility conditions above $S$ become void on Riemann surfaces, so the isomorphisms Theorems \ref{iso-moduli-vector-bundles-th}, \ref{iso-moduli-G-bundles-th}	give:
\begin{re}\label{IsoModuliDim=1}
Suppose  $\dim(U)=1$. With the notations and under the assumptions above we have homeomorphisms of moduli spaces: 
$$
{\cal M}_S(E)\to {\cal M}_{\p\widehat{X}_S}(\widehat{E}), \ {\cal M}_S(E)\to {\cal M}_{\p\bar X^-}(E^-)\times{\cal M}_{\p\bar X^+}(E^+),
$$
$$
{\cal M}_S(P)\to  {\cal M}_{\p\widehat X_S}(\widehat P),\ {\cal M}_S(P)\to {\cal M}_{\p\bar X^-}(P^-)\times{\cal M}_{\p\bar X^+}(P^+).
$$
\end{re}
\begin{re}
In 	section \ref{abstract-interpr} we will identify the moduli spaces ${\cal M}_S(E)$, ${\cal M}_S(P)$, ${\cal M}_{\p\bar X}(E)$, ${\cal M}_{\p\bar X}(P)$ intervening in Theorems \ref{iso-moduli-vector-bundles-th}, \ref{iso-moduli-G-bundles-th} with moduli spaces of {\it framed bundles} defined (in an abstract way, see Definition \ref{abstract-framed}) as pairs consisting of a holomorphic bundle on $X$ ($\bar X$) and a framing of class ${\cal C}^{\kappa+1}$ on $S$ (respectively $\p\bar X$).
\end{re}
In section  \ref{ExamplesSection} we will consider explicit examples of (boundary) framed moduli spaces and give explicit formulae for the homeomorphisms given by Theorem \ref{iso-moduli-G-bundles-th} and their inverses in the special cases :
\begin{itemize}
\item 	$X=\P^1_\C$ and $S\subset\C$ is a closed curve.
\item $X$ is an elliptic curve and $S\subset X$ is a non-separating closed curve.
\end{itemize}

\section{Gluing holomorphic bundles  along a real hypersurface}

\subsection{The tangential Cauchy-Riemann operator}
\label{tangentialCR}
Let $U$ be a differentiable manifold, $S\subset U$  a closed real hypersurface and $\eta_S\subset T^{*\C}_{U|S}$ be the annihilator of $T_S$ (or, equivalently, of $T_S^\C$) in the restriction $T^{*\C}_{U|S}$ of the complex cotangent bundle $T_U^{*\C}$ of $U$ to $S$; $\eta_S$ can be identified with the complexification of the  conormal real line bundle $n^*_S$ of $S$ in $U$.

Suppose now that $U$ is an $n$-dimensional complex manifold.   The image $\eta^{0,1}_S$ of $\eta_S$ in $\extp^{0,1}_{\;U|S}$ is a line subbundle of $\extp^{0,1}_{\;U|S}$, which can be identified with the annihilator of the canonical distribution 
$$\Tg_S=T_S\cap J_U(T_S)\subset T_S$$
(or, equivalently, of the hyperplane $\Tg_S^{0,1}\subset T_{U|S}^{0,1}$) in $\extp^{0,1}_{\;U|S}$. Here $J_U\in \Gamma(U,\End(T_U))$ stands for the integrable almost complex structure on $U$ induced by its complex manifold structure. The projection $T^{*\C}_{U|S}\to \extp^{0,1}_{\;U|S}$ induces a line bundle isomorphism $\psi_S:\eta_S\to \eta^{0,1}_S$.

 Put  $\extp^{0,q}_{\;S}\edf \extp^{0,q}\Tg_S^{*\C}$. The fiber $\extp^{0,q}_{\;S,x}$ of this bundle over   $x\in S$ can be identified with the space of alternate $\R$-multilinear forms $\Tg_{S,x}^q\to\C$  which are anti-linear with respect to each argument. We have an obvious bundle epimorphism 
 $$r_{\Tg_S}:\extp^{0,q}_{\;U|S}\to \extp^{0,q}_{\;S}$$ induced by the inclusion $\Tg_S\subset T_{U|S}$.
We obtain a commutative diagram of bundle morphisms on $S$ with exact horizontal rows 
\begin{equation}\label{diagfordelta_S}
\begin{tikzcd}[column sep = large]
&\eta_S   \ar[d, " \underset{\simeq}{\psi_S}"']&\\
0\ar[r]&\eta^{0,1}_S\ar[r, hook ]\ar[d, equal]& \extp^{0,1}_{\;U|S}\ar[r,   "\sqcap_S	"]\ar[d, equal]&  \eta^{0,1*} \otimes\extp^{0,2}_{\; U|S}\\	
0\ar[r]&\eta^{0,1}_S\ar[r, hook ]& \extp^{0,1}_{\;U|S}\ar[r,   "r_{\Tg_S}"]&   \extp^{0,1}_{\;S}\ar[r]\ar[u, hook', "{[\sqcap_S	]}"]&0\,,
\end{tikzcd}
\end{equation}
 where, via the identification $\extp^{0,1}_{\;U|S}= \eta^{0,1*}\otimes(\eta^{0,1}\otimes \extp^{0,1}_{\;U|S})$, we have put   
$$\sqcap_S	\edf \id_{\eta^{0,1*}}\otimes \wedge\,,$$
and $[\sqcap_S	]$ is induced by $\sqcap_S	$. \begin{re}
By the definition of $\sqcap_S	$ we have the identity:  
\begin{equation}\label{id-otimes-wedge}
\forall x\in S,\ \forall (a^{0,1},b^{0,1})\in \eta^{0,1}_x\times\extp^{0,1}_{\;U,x},\ a^{0,1}\otimes  \sqcap_S	 (b^{0,1})=a^{0,1}\wedge b^{0,1}.
\end{equation} 
Taking $a^{0,1}=\psi_S(a)$ with $a\in \eta_x$, we obtain
$$
\forall x\in S,\ \forall (a,b^{0,1})\in \eta_x\times\extp^{0,1}_{\;U,x},\ \psi_S(a)\otimes \sqcap_S	(b^{0,1})=\psi_S(a)\wedge b^{0,1}.
$$
This formula can be written as
\begin{equation}\label{psi-otimes-wedge}
\big(\psi_S \otimes  \id_{\extp^{0,2}_{\,U|S}})\circ\big( \id_{\eta_S}\otimes  \sqcap_S	\big)= \wedge(\psi_S\otimes\id_{\extp^{0,1}_{\,U|S}}),	
\end{equation}
where, on the right, ${\wedge}$ stands for the bundle morphism $\eta^{0,1}_S\otimes \extp^{0,1}_{\;U|S}\to \extp^{0,2}_{\;U|S}$ induced by the wedge product.	
\end{re}

Similarly, for any $q\geq 1$ we obtain a commutative diagram of bundles on $S$
\begin{equation}\label{diagfordelta_S-q}
\begin{tikzcd}[]
0\ar[r]&\eta^{0,1}_S\wedge\extp^{0,q-1}_{\;U|S}\ar[r, hook ]\ar[d, equal]& \extp^{0,q}_{\;U|S}\ar[r,   "\sqcap_S	"]\ar[d, equal]&  \eta^{0,1*} \otimes\extp^{0,q+1}_{\; U|S}\\	
0\ar[r]&\eta^{0,1}_S\wedge\extp^{0,q-1}_{\;U|S}\ar[r, hook ]& \extp^{0,q}_{\;U|S}\ar[r,   "r_{\Tg_S}"]&   \extp^{0,q}_{\;S}\ar[r]\ar[u, hook', "{[\sqcap_S	]}"]&0
\end{tikzcd}
\end{equation}
 with exact rows. 
\vspace{2mm}

Let $E$ be a complex vector bundle of class ${\cal C}^{\infty}$   on $U$. For $ \gamma\in [0,+\infty]$ put
\begin{align}\label{defI0q}
{\cal I}^{\gamma}_S(U, \extp^{0,q}_{\; U}\otimes E)\edf & \ker\big[\Gamma^\gamma(U,\extp^{0,q}_{\; U} \otimes E)\to \Gamma^\gamma(S,\extp^{0,q}_{\; U|S}\otimes E)  \to  \Gamma^\gamma(S,\extp^{0,q}_{\;S}\otimes E)\big]	\nonumber\\
=&\left\{\begin{array}{ccc}\ker(\Gamma^\gamma(U,E)\to  \Gamma^\gamma(S,E))&\rm if&q=0,\\ \{\beta\in \Gamma^\gamma(U,\extp^{0,q}_{\; U}\otimes E)|\  (\sqcap_S	\otimes\id_E)\beta_S=0\} &\rm if&q>0.\end{array}\right. 
\end{align}
\begin{re}\label{Q0q}
The two restriction maps  in  the definition of ${\cal I}^{\gamma}_S(U, \extp^{0,q}_{\;U}\otimes E)$ are surjective, so their composition 	$\Gamma^\gamma(U,\extp^{0,q}_{\;U} \otimes E)\to \Gamma^\gamma(S,\extp^{0,q}_{\;S}\otimes E)$ induces an isomorphism
$$
 \qmod{\Gamma^\gamma(U,\extp^{0,q}_{\; U} \otimes E)}{{\cal I}^{\gamma}_S(U, \extp^{0,q}_{\; U}\otimes E)}\textmap{\simeq}\Gamma^\gamma(S,\extp^{0,q}_{\; S}\otimes E).
$$
\end{re}
\begin{proof}
  The surjectivity of $\Gamma^\gamma(U,\extp^{0,q}_{\; U} \otimes E)\to \Gamma^\gamma(S,\extp^{0,q}_{\; U|S}\otimes E)$ follows taking $m=0$ in	 Corollary \ref{extensionCoroCkappa} (1),  (2). The  map $\Gamma^\gamma(S,\extp^{0,q}_{\; U|S}\otimes E)  \to  \Gamma^\gamma(S,\extp^{0,q}_{\; S}\otimes E)$ is induced by an epimorphism of ${\cal C}^\infty$ bundles on $S$, so is surjective.   
\end{proof}

 Let ${\cal C}^\gamma_U$ (respectively ${\cal C}^\gamma(\extp^{0,q}_{\;U} \otimes E)$, ${\cal C}^{\gamma}(\extp^{0,q}_{\;U|S}\otimes E))$) be the sheaves of locally defined $\C$-valued functions (sections of the bundles $\extp^{0,q}_{\;U} \otimes E$, $\extp^{0,q}_{\;U|S}\otimes E)$ on $U$, respectively $S$, of class ${\cal C}^\gamma$ (see section \ref{Ckappa}).  The assignment 
$$
U\stackrel{\hb{\tiny open}}{\supset} V\mapsto {\cal I}^{\gamma}_{S\cap V}(V, \extp^{0,q}_{\;V}\otimes E)
$$
defines a sheaf on $U$ which will be denoted ${\cal I}^{\gamma}_{S}(\extp^{0,q}_{\;U}\otimes E)$; it is a   ${\cal C}^\gamma_U$-submodule  of ${\cal C}^\gamma(\extp^{0,q}_{\;U} \otimes E)$ which coincides with ${\cal C}^\gamma(\extp^{0,q}_{\;U} \otimes E)$ on the complement of $S$. Let $x\in X$,  $\rho\in{\cal C}^\infty(V,\R)$ be a local defining function for $S$ around $x$ (i.e. we have $x\in S\cap V=\rho^{-1}(0)$ and $d\rho$ is nowhere vanishing on $S\cap V$), and $\rho_x$ its germ at $x$.

\begin{re}\label{descriptionI}
The stalk ${\cal I}^{\gamma}_S(\extp^{0,q}_{\;U}\otimes E)_x$ of ${\cal I}^{\gamma}_S(\extp^{0,q}_{\;U}\otimes E)$ at $x\in S$ is given by: 
$${\cal I}^{\gamma}_S(\extp^{0,q}_{\;U}\otimes E)_x=\ker({\cal C}^{\gamma}(\extp^{0,q}_{\;U}\otimes E)_x\to {\cal C}^{\gamma}(\extp^{0,q}_{\;U|S}\otimes E)_x)+ \bp\rho_x\wedge {\cal C}^{\gamma}(\extp^{0,q-1}_{\;U}\otimes E)_x.
$$
 \end{re}
 \begin{proof}
Diagram (\ref{diagfordelta_S-q}) shows that a form $\beta\in \Gamma^\gamma (V,\extp^{0,q}_{\;V}\otimes E)$ belongs to  ${\cal I}^\gamma_S(V,\extp^{0,q}_{\;V}\otimes E)$ if and only if its restriction to $V\cap S$  is a section of the subbundle 
$$\eta^{0,1}_{V\cap S}\wedge\extp^{0,q-1}_{\;V|S}\otimes E_S\subset \extp^{0,q}_{\;V|S}\otimes E_S.$$
  It suffices to note that the restriction of the sub-bundle $\bp\rho\wedge\extp^{0,q-1}_{\;V}\subset \extp^{0,q}_{\;V}$ to $S\cap V$ coincides with $\eta^{0,1}_{V\cap S}\wedge\extp^{0,q-1}_{\;V|S}$.
 \end{proof}

Note that, for $\gamma=\infty$, Remark \ref{descriptionI} gives:
$${\cal I}^{\infty}_S(\extp^{0,q}_{\;U}\otimes E)_x=\rho_x\,{\cal C}^{\infty}(\extp^{0,q}_{\;U}\otimes E)_x+ \bp\rho_x\wedge {\cal C}^{\infty}(\extp^{0,q-1}_{\;U}\otimes E)_x.
$$
This description does not extend to  the case $\gamma<\infty$. For instance, for $q=0$,  an element of the stalk ${\cal C}^\gamma(E)_x$ which vanishes on $S$ is not necessarily divisible by $\rho_x$ in this ${\cal C}^\gamma_{U,x}$-module.
\begin{re}\label{deltaIsubsetI}
Let $\kappa\in [0,\infty]$,  $\delta$  a Dolbeault operator with coefficients in ${\cal C}^\kappa$ on $E$ and   $0\leq \gamma\leq \kappa$. Then $
\delta  {\cal I}^{\gamma+1}_S(U, \extp^{0,q}_{\;U}\otimes E)\subset {\cal I}^{\gamma}_S(U, \extp^{0,q+1}_{\;U}\otimes E)$.	
\end{re}
\begin{proof} Let $\beta\in {\cal I}^{\gamma+1}_S(U, \extp^{0,q}_{\;U}\otimes E)_x$. With respect to   a holomorphic chart of $U$ and a local trivialization  of $E$  around $x$, $\delta$ is given   by
$$
\delta \psi =\bp \psi+\alpha\psi
$$
for a germ $\alpha\in {\cal C}^{\kappa}(\extp^{0,1}_{\;U}\otimes\gl(r,\C))_x$.
By Remark	\ref{descriptionI} we have $\beta=\beta^0+\bp \rho_x \wedge \nu$ where 
$$\beta^0\in \ker({\cal C}^{\gamma+1}(\extp^{0,q}_{\;U}\otimes E)_x\to {\cal C}^{\gamma+1}(\extp^{0,q}_{\;U|S}\otimes E)_x),\ \nu\in {\cal C}^{\gamma+1}(\extp^{0,q-1}_{\;U}\otimes E)_x.$$
 Writing $\beta^0 = \sum_{|I|=q}  d\bar z^I\otimes \beta_{I}$, where all  the germs $\beta_{I}\in {\cal C}^{\gamma+1}_{Ux}$ vanish on $S$, we have
$$
\delta\beta^0=(-1)^q\sum_{|I|=q} d\bar z^I\wedge (\bp \beta_{I}+\alpha \beta_I).
$$
Since $\beta_I$ vanishes on $S$, it follows that  $d\beta_I$ vanishes on $T_S$  around $x$, so  $\bp \beta_I$ vanishes on $\Tg_S$ around $x$. This proves that the terms $d\bar z^I\wedge\bp \beta_I$ belong to ${\cal I}^{\gamma}_S(\extp^{0,q+1}_{\;U}\otimes E)_x$.  
The terms $d\bar z^I\wedge  \alpha \beta_I$, $\delta(\bp\rho\wedge\nu)=-\bp\rho\wedge\delta\nu$ obviously belong to ${\cal I}^{\gamma}_S(\extp^{0,q+1}_{\;U}\otimes E)_x$, which completes the proof.
\end{proof}

Using Remarks \ref{Q0q}, \ref{deltaIsubsetI} it follows.

\begin{co}
Let $\delta$ be a Dolbeault operator with coefficients in ${\cal C}^\kappa$ on $E$. The associated operator $\Gamma^{\kappa+1}(U,\extp^{0,q}_{\;U} \otimes E)\to \Gamma^{\kappa}(U,\extp^{0,q+1}_{\;U} \otimes E)$ induces  a first order differential operator $\Gamma^{\kappa+1}(S,\extp^{0,q}_{\;S} \otimes E)\to \Gamma^{\kappa}(S,\extp^{0,q+1}_{\;S} \otimes E)$   with   coefficients in ${\cal C}^\kappa$. 	
\end{co}
Taking $q=0$, we obtain a first order differential operator 
$$
\delta_S:\Gamma^{\kappa+1}(S,E)\to \Gamma^{\kappa}(S,\extp^{0,1}_{\;S} \otimes E),
$$
with coefficients in ${\cal C}^\kappa$, which is called {\it the tangential Cauchy-Riemann operator associated with} $\delta$.  
\begin{re}\label{variationdeltaS}
For a form $\alpha\in \Gamma^\kappa(U,\extp^{0,1}_{\;U}\otimes \End(E))$ we have  $(\delta+\alpha)_S=\delta_S+\alpha_S$
where $\alpha_S$ is the image of $\alpha$ under  $\Gamma^\kappa(U,\extp^{0,q}_{\;U} \otimes E)\to \Gamma^\kappa(S,\extp^{0,q}_{\;U|S}\otimes E)  \to  \Gamma^\kappa(S,\extp^{0,q}_{\:S}\otimes E)$.
\end{re}

In a similar way one obtains a tangential Cauchy-Riemann operator 
$$
\delta_{\p\bar U^+}:\Gamma^{\kappa+1}({\p\bar U^+},E)\to \Gamma^{\kappa}({\p\bar U^+},\extp^{0,1}_{\;\p\bar U^+} \otimes E),
$$
 associated to any Dolbeault operator with coefficients in ${\cal C}^\kappa$ on a ${\cal C}^\infty$ vector bundle $E$ on a submanifold with boundary $\bar U^+\subset U$, where $U^+\subset U$ is open. Note that the correspondence $\delta\mapsto \delta_{\p\bar U^+}$ plays an important role in  \cite[section 3.5]{Do}.  
 
\begin{re}\label{JS} 
The tangential Cauchy-Riemann operator has an analogue in the framework of principal bundles (see section \ref{ACSsection} in the appendix):	 Let $p:P\to U$ ($p^+:P^+\to \bar U^+$) be a principal $G$-bundle on a complex manifold (with boundary) $U$ ($\bar U^+\subset U$), and let $S\subset U$ be a closed, oriented real hypersurface in $U$ (respectively let $S\edf\p\bar U^+=\bar U^+\setminus U^+$). 
A bundle ACS $J$ of class ${\cal C}^\kappa$ on $P$   defines an ACS $J_S$ of the same class on the pull-back distribution $\Tg_{P_S}\edf p_{S*}^{-1}(\Tg_S)\subset T_{P_S}$; $J_S$ is $G$-invariant, makes the vector bundle epimorphism $\Tg_{P_S}\to p_S^*(\Tg_S)$ $\C$-linear, and the parametrization of the $G$-orbits  pseudo-holomorphic. \end{re}

\subsection{Gluing theorems} 

Let $\iota:G\to\Aut(G)$ be the group morphism which assigns to $g\in G$ the interior automorphism $\iota_g$, ${\cal C}^{\kappa+1}_\iota(P,G)$ the space of $\iota$-equivariant maps $P\to G$ of class ${\cal C}^{\kappa+1}$, and $A^{p,q}_{\Ad}(P,\g^{1,0})_\kappa$ the space of tensorial $\g^{1,0}$-valued forms of type $\Ad$, bidegree $(p,q)$ and class ${\cal C}^\kappa$ on $P$. In section \ref{ACSsection} we associated with a bundle ACS $J\in {\cal J}^\kappa_P$ the  maps
$$\bar\lg_J:{\cal C}^{\kappa+1}_\iota(P,G)\to A^{0,1}_{\Ad}(P,\g^{1,0})_\kappa=A^{0,1}(U,P\times_\Ad\g^{1,0})_\kappa\simeq A^{0,1}(U,\Ad(P))_\kappa,
$$ 
$$
\bar \kg_J:A^{0,1}_\Ad(P,\g^{1,0})_\kappa \to  A^{0,2}_\Ad(P,\g^{1,0})_{\kappa-1} \hb{ (for $\kappa\geq 1$)}.
$$
We identify $\g^{1,0}$ with $\g$ and $\theta^{1,0}$ with $\theta$ in the standard way (see section \ref{ACSsection}), so $\bar\kg_J$ becomes a map $A^{0,1}(U,\Ad(P))_\kappa\to A^{0,2}(U,\Ad(P))_{\kappa-1}$.  For $s\in \Gamma^{l+1}(U,\Ad(P))$ put $\bar\dg_J(s)\edf \bar\lg_J(\exp(s))$.
 
\begin{lm}\label{newlm}
Let $0\leq l\leq k$. Let $s\in \Gamma^{l+1}(U,\Ad(P))$  with $j^l_S s=0$, so that the intrinsic differential $D_S^{l+1}s\in \Gamma^0(S,\eta_S^{\otimes(l+1)}\otimes \Ad(P))$ is well defined (see section \ref{IntrinsicDiff}). Let $J$ be a bundle ACS of class ${\cal C}^k$ on $P$. 	Then
\begin{enumerate}
	\item $j_S^{l-1} (\bar\dg_J(s))=0$ (if $l\geq 1$).
	\item $D_S^l(\bar \dg_J(s))= (\id_{\eta_S}^{\otimes l}\otimes\psi_S\otimes \id_{\Ad(P)})(D_S^{l+1}s)$.
\end{enumerate}
\end{lm}
\begin{proof}
(1) The section $s$ can be identified with an element, denoted by the same symbol, of ${\cal C}^{l+1}_\Ad(P,\g)$. Using this interpretation of $s$, we obtain an element $\sigma=\exp(s)\in {\cal C}^{l+1}_\iota(P,G)$. Let  $\tau\in \Gamma^\infty(W,P)$ be a local section of $P$, and put 
$$s_\tau\edf s\circ\tau\in {\cal C}^{l+1}(W,\g),\ \sigma_\tau\edf \sigma\circ\tau=\exp(s_\tau)\in {\cal C}^{l+1}(W,G).$$
 Using formula  (\ref{bar-lg-sigma}) explained in section \ref{ACSsection},  we have: 
\begin{equation}\label{tau*bar-dfJ(s)}
\begin{split}
\tau^*(\bar\dg_J(s))&=\tau^*(\bar\lg_J(\sigma))=\bar \lg^\tau_J(\sigma)=\sigma_\tau^*(\theta)^{0,1}+(\Ad_{\sigma_\tau^{-1}}-\id)(\alpha_J^\tau)=\\ 
&=s_\tau^*(\exp^*(\theta))^{0,1}+(\Ad_{\exp(-s_\tau)}-\id)(\alpha_J^\tau).
\end{split}	
\end{equation}
We may suppose that $U$ is an open subset of $\C^n$.  Since $j^l_S s=0$, it follows by the composition Lemma \ref{ABLemma} (2)  that 
\begin{equation}\label{jlSjlS}
j^l_S((\Ad_{\exp(-s_\tau)}-\id)(\alpha_J^\tau))=0,
\end{equation}
and by Lemma \ref{Dls0-3} that $j^{l-1}_S(s_\tau^*(\exp^*(\theta))^{0,1})=0$.
\vspace{2mm}\\
(2) It suffices to prove that for any local section $\tau\in \Gamma^\infty(W,P)$ we have
$$
D_S^l(\bar \dg_J(s_\tau))= (\id_{\eta_S}^{\otimes l}\otimes\psi_S\otimes \id_{\Ad(P)})(D_S^{l+1}s_\tau).
$$
Taking into account (\ref{tau*bar-dfJ(s)}) and (\ref{jlSjlS}), it suffices to compute	$D^l_S\big(s^*_\tau(\exp^*(\theta))^{0,1}\big)$. We use formula (\ref{Dls3}) of Lemma \ref{Dls0-3}, taking $V=F=\g$, $f=s_\tau:U\to \g$,  and $\omega\edf \exp^*(\theta)$, which is a holomorphic $(1,0)$ form on $\g$, because $\exp$ is a holomorphic map.  We have to specify the map $\omega^f_S$ intervening  on the right in (\ref{Dls3}). Regarded as a map ${\cal C}^\infty(\g,\Hom(\g,\g))$, $\omega$ is given by 
$$
\omega(a)(v)=(l_{\exp(a)^{-1}})_{*}((\exp_{*,a}(v)),
$$
so, for $x\in U$, we have
$$
\omega^f(x)(v)=(l_{\exp(f(x))^{-1}})_{*}((\exp_{*,f(x)}(v)).
$$

Since we assumed $j^l_Ss=0$, we have $f(x)=s_\tau(x)=0$ for any $x\in S$, so $\omega^f_S(x)=\id_\g$ for any $x\in S$.  
\end{proof}
Let $\kappa\in (0,\infty]\setminus\N$ and $k\edf[\kappa]$.
\begin{lm}\label{sigma-beta'}
Let $J$ be a bundle ACS of class ${\cal C}^{\kappa}$ on $P$. Let $l\leq k$ be a non-negative integer, and let $\beta\in \Gamma^{\kappa}(U,\extp^{0,1}_{\; U}\otimes \Ad(P))$ be such that 
$$j_S^{l-1}\beta=0\ (\hb{required only if $l\geq 1$}),\ (\id_{\eta_S}^{\otimes l}\otimes \sqcap_S\otimes\id_{\Ad(P)_S})(D^l_S\beta)=0.$$
\begin{enumerate}
\item 	There exists $s\in \Gamma^{\kappa+1}(U,\Ad(P))$ such that, putting 
$$\beta'\edf  \beta-\bar\dg_J(s)\in \Gamma^{\kappa}(U,\extp^{0,1}_{\; U}\otimes \Ad(P)),$$
we have
$$
j_S^{l} s=0,\ j_S^l\beta'=0.
$$
\item  Suppose that the considered pair $(J,\beta)$ also satisfies:
\begin{enumerate}[(i)]
\item $j_S^{k-1}(\kg_J(\beta))=0$.
\item $j_S^{k-2-l}(\fg_J)=0$ (required only if $l\leq k-2$).
\end{enumerate}
Then, for any such $s$, putting $\tilde\beta\edf \Ad_{\exp(s)}(\beta')$, we have:
\begin{enumerate}
\item $j_S^l\tilde\beta=0$.
\item  $j_S^{k-1}(\kg_J(\tilde\beta))=0$.
\item If $l\leq k-1$, we also have $(\id_{\eta_S}^{\otimes (l+1)}\otimes\sqcap_S	\otimes\id_{E_S})(D_S^{l+1}\tilde\beta)=0$. 
\end{enumerate}
\item If $\kappa\ne\infty$,  $s$ can be chosen to depend continuously on $\beta$.
\end{enumerate}
\end{lm}
\begin{proof}
(1) Using the first horizontal exact sequence in (\ref{diagfordelta_S}), we see that the hypothesis implies 
$$
D^l_S\beta\in \Gamma^{\kappa-l}(S,\eta_S^{\otimes l}\otimes \eta_S^{0,1}\otimes \Ad(P)),
$$   
so, since $\psi_S$ is a line bundle isomorphism, there exists 
$b\in \Gamma^{\kappa-l}(S,\eta_S^{\otimes (l+1)}\otimes \Ad(P)_S)$ 
such that
\begin{equation}\label{DlSbeta}
D^l_S\beta=(\id_{\eta_S}^{\otimes l}\otimes \psi_S\otimes\id_E)(b).
\end{equation}

The key argument in the proof: by the  extension Corollary  \ref{extensionCoroCkappa}, there exists $s\in \Gamma^{\kappa+1}(U,\Ad(P))$ such that 
\begin{equation}\label{ConditionsForSigma}
j^l_Ss=0,\ D^{l+1}_Ss=b.	
\end{equation}
 It follows that $\beta'\edf \beta-\bar\dg_J(s)$ belongs to $\Gamma^{\kappa}(U,\extp^{0,1}_{\; U}\otimes \Ad(P))$. Moreover, by Lemma \ref{newlm} (1), we have $j_S^{l-1}(\bar\dg_J(s))=0$, so, since $j_S^{l-1}\beta=0$, we obtain   $j_S^{l-1}\beta'=0$ and 
 $$D^l_S\beta'=D^l_S\beta-D^l_S(\bar\dg_J(s)).$$
   Using (\ref{DlSbeta}), (\ref{ConditionsForSigma}) and Lemma \ref{newlm} (2),  we obtain $D^l_S\beta-D^l_S(\bar\dg_J(s))=0$, so $j_S^{l}\beta'=0$.
\vspace{2mm}\\ 
(2)(a)   follows from $j^l_S(\beta')=0$ using the composition Lemma   \ref{ABLemma} (2).
\vspace{2mm}\\
(2)(b) By Lemma \ref{kgJlgJ} (2) proved in section \ref{ACSsection}, we have:
$$
\bar\kg_J(\tilde\beta)=\bar\kg_J(\Ad_{\exp(s)}(\beta-\bar\dg_J(s))=\Ad_{\exp(s)}(\bar\kg_J(\beta))+(\Ad_{\exp(s)}-\id)(\fg_J).
$$
Since we assumed $j^{k-1}_S(\bar\kg_J(\beta))=0$, it follows again by Lemma   \ref{ABLemma} (2) that $j^{k-1}_S(\Ad_{\exp(s)}(\bar\kg_J(\beta)))=0$.

On the other hand, since $j_S^{ k-2-l}(\fg_J)=0$ by hypothesis and $j^l_Ss=0$ by (1), it follows   by Lemma \ref{ABLemma} (1), (2) that
$$j^{k-1}_S((\Ad_{\exp(s)}-\id)(\fg_J))=0.$$
Therefore $j_S^{k-1}(\bar\kg_J(\tilde\beta))=0$ as claimed.  
 \vspace{2mm}\\
 (2)(c) Suppose $l\leq k-1$. Formula (\ref{psi-otimes-wedge}) shows that
$$
\big(\psi_S \otimes  \id_{\extp^{0,2}_{U|S}\otimes E_S})\circ\big( \id_{\eta_S}\otimes  \sqcap_S	\otimes\id_{E_S}\big)= \wedge(\psi_S\otimes\id_{\extp^{0,1}_{\,U|S}\otimes E_S})
$$
on $\eta_S\otimes \extp^{0,1}_{\;U|S}\otimes E_S$. Therefore
\begin{equation}\label{TTTold}
\begin{split}
\big(\id_{\eta_S}^{\otimes l}\otimes \psi_S \otimes  \id_{\extp^{0,2}_{U|S}\otimes E_S}&\big)\big((\id_{\eta_S}^{\otimes (l+1)}\otimes \sqcap_S	\otimes\id_{E_S})(D^{l+1}_S\tilde\beta)\big)=\\
&=\big(\id_{\eta_S}^{\otimes l}\otimes  \wedge(\psi_S\otimes\id_{\extp^{0,1}_{\,U|S}\otimes E_S})\big)(D^{l+1}_S\tilde\beta)=D_S^l(\bp_J\tilde\beta),\end{split}
\end{equation}
where, for the last equality we used formula (\ref{Dl-delta-beta-eq}) of Lemma \ref{Dl-delta-beta}  for $q=1$.
Since $l\leq k-1$  and we have $j_S^{k-1}(\kg_j (\tilde\beta))=0$ by (2)(b), it follows that $j_S^l(\kg_j (\tilde\beta))=0$, in particular $D_S^l(\kg_j (\tilde\beta))=0$. 

But
$$
\kg_j (\tilde\beta)=\bp_J\tilde\beta+\frac{1}{2}[\tilde\beta\wedge\tilde\beta]
$$
with $j^l_S(\tilde\beta)=0$, which  implies $j_S^{2l+1}([\tilde\beta\wedge\tilde\beta])=0$ by Lemma \ref{ABLemma} (1). It  follows $D_S^l(\bp_J\tilde\beta)=0$.
Since the linear map $\id_{\eta_S}^{\otimes l}\otimes \psi_S \otimes  \id_{\extp^{0,2}_{U|S}\otimes E}$ is injective, formula (\ref{TTTold}) shows that
$$
(\id_{\eta_S}^{\otimes (l+1)}\otimes \sqcap_S	\otimes\id_{E_S})(D^{l+1}_S\tilde\beta)=0,
$$
as claimed.
\vspace{2mm}\\
(3) For $\kappa\ne \infty$,  the extension Corollary \ref{extensionCoroCkappa}  provides a solution $s\in  \Gamma^{\kappa+1}(U,\Ad(P))$ of the equations (\ref{ConditionsForSigma}) which depends continuously on $b$, so on $\beta$. 

\end{proof}

\begin{pr}
\label{prop-beta-v-new}   
Let $J$ be a bundle ACS of class ${\cal C}^{\kappa}$ on $P$ such that $j_S^{k-2}(\fg_J)=0$. Let $\beta\in {\cal I}_S^{\kappa}(U,\extp^{0,1}_{\; U}\otimes \Ad(P))$ be such that
$j_S^{k-1}(\bar\kg_J(\beta))=0$.
Then 
\begin{enumerate}
\item There exists $s\in {\cal I}_S^{\kappa+1}(U,\Ad(P))$ such that  $j_S^k (\beta-\bar\dg_J (s))=0$.
\item If $\kappa\ne \infty$, $s$ can be chosen to depend continuously on $(J,\beta)$.
\end{enumerate}
\end{pr}
\begin{proof}
(1) Suppose   that $\kappa\ne \infty$. Our assumption $\beta\in {\cal I}_S^{\kappa}(U,\extp^{0,1}_{\;U}\otimes E)$ means that $\beta$ satisfies the hypothesis of  Lemma 	\ref{sigma-beta'}   for $l=0$. Applying successively this Lemma 	to 
$$\beta_0\edf\beta,\   \beta_1\edf\tilde\beta_0,\ \dots,\ \beta_{k}\edf\tilde\beta_{k-1},$$
  we obtain sequences $(\beta'_i)_{0\leq i\leq k}, \ (s_i)_{0\leq i\leq k}$
in  $\Gamma^{\kappa}(U,\extp^{0,1}_{\; U}\otimes \Ad(P))$,  $\Gamma^{\kappa+1}(U,\Ad(P))$ respectively such that
\begin{equation}\label{jls,jlbeta'}
j^l_S(s_l)=0,\  j^l_S(\beta'_l)=0 \hb{ for } 0\leq l\leq k,	
\end{equation}
and, putting $\sigma_i\edf\exp(s_i)$, one has:
\begin{alignat}{3}
 \beta_l&=\Ad_{\sigma_{l-1}}(\beta_{l-1}') & \hb{ for } 1\leq l\leq k, 
 \label{beta-l-beta'-(l-1)} \\ 
 \beta_l&=\bar\lg_J(\sigma_l)+\beta'_l &\hb{ for } 0\leq l\leq k. 
 \label{beta-l-sigma-l}
\end{alignat}
Combining (\ref{beta-l-beta'-(l-1)}) and (\ref{beta-l-sigma-l}) we obtain:
\begin{equation}
 \begin{array}{ccl}
\beta&=&\bar\lg_J(\sigma_0) 	+\beta_0'\\
\beta'_0&=&\Ad_{\sigma_0^{-1}}(\bar\lg_J(\sigma_1)+\beta_1')\\
\beta'_1&=&\Ad_{\sigma_1^{-1}}(\bar\lg_J(\sigma_2)+\beta_2')\\
\vdots &\vdots& \vdots\\
\beta'_{k-1}&=&\Ad_{\sigma_{k-1}^{-1}}(\bar\lg_J(\sigma_k)+\beta_k').
 \end{array}
\end{equation}
This implies:
\begin{equation}
 \begin{array}{ccl}
\beta&=&\bar\lg_J(\sigma_0) 	+\beta_0'\\
\beta'_0&=&\Ad_{\sigma_0^{-1}}(\bar\lg_J(\sigma_1)+\beta_1')\\
\Ad_{\sigma_0^{-1}}(\beta'_1)&=&\Ad_{\sigma_0^{-1}}\Ad_{\sigma_1^{-1}}(\bar\lg_J(\sigma_2)+\beta_2')\\
\vdots &\vdots& \vdots\\
\Ad_{\sigma_0^{-1}}\Ad_{\sigma_1^{-1}}\dots\Ad_{\sigma_{k-2}^{-1}} (\beta'_{k-1})&=&\Ad_{\sigma_0^{-1}}\Ad_{\sigma_1^{-1}}\dots\Ad_{\sigma_{k-1}^{-1}}(\bar\lg_J(\sigma_k)+\beta_k').
 \end{array}
\end{equation}
For $0\leq l\leq k$ put $\sg_l\edf\sigma_l\dots\sigma_0$. We obtain
\begin{equation}
 \begin{array}{ccl}
\beta&=&\bar\lg_J(\sigma_0) 	+\beta_0'\\
\beta'_0&=&\Ad_{\sg_0^{-1}}(\bar\lg_J(\sigma_1)+\beta_1')\\
\Ad_{\sg_0^{-1}}(\beta'_1)&=&\Ad_{\sg_1^{-1}}(\bar\lg_J(\sigma_2)+\beta_2')\\
\vdots &\vdots& \vdots\\
\Ad_{\sg_{k-2}^{-1}} (\beta'_{k-1})&=&\Ad_{\sg_{k-1}^{-1}}(\bar\lg_J(\sigma_k)+\beta_k').
 \end{array}
\end{equation}
Therefore
\begin{equation}\label{formula-for-beta}
\beta=\bar\lg_J(\sigma_0)+\sum_{l=1}	^k\Ad_{\sg_{l-1}^{-1}}(\bar\lg_J(\sigma_l))+  \Ad_{\sg_{k-1}^{-1}}(\beta'_k).
\end{equation}
But $\sg_l=\sigma_l\sg_{l-1}$ so, by Lemma \ref{kgJlgJ} (1), we have
$$
\bar\lg_J(\sg_l)=\bar\lg_J(\sg_{l-1})+\Ad_{\sg_{l-1}^{-1}}(\bar\lg_J(\sigma_l)) \hbox{ for }1\leq l\leq k,
$$
so
$$
\bar\lg_J(\sg_k)=\bar\lg_J(\sigma_0)+\sum_{l=1}	^k\Ad_{\sg_{l-1}^{-1}}(\bar\lg_J(\sigma_l)).
$$
Formula (\ref{formula-for-beta}) becomes
\begin{equation}\label{new-formula-for-beta}
\beta=\bar\lg_J(\sg_k) +\Ad_{\sg_{k-1}^{-1}}(\beta'_k).
\end{equation}

 Making use of Proposition \ref{fiberwise-exp}, let  $\Ad(P)_0$, $\iota(P)_0$ be  neighborhoods of the zero section (identity section) in the two bundles such that $\exp$ induces a diffeomorphism $\Ad(P)_0\to \iota(P)_0$.  

We can assume that $\sg_l$ takes values in $\iota(P)_0$ for $0\leq l\leq k$, so we can write $\sg_l=\exp(s^l)$ for a section $s^l\in \Gamma^{\kappa+1}(U,\Ad(P)_0)\subset\Gamma^{\kappa+1}(U,\Ad(P))$. It suffices to put $s\edf s^k$ and to take into account that $j_S^k\big(\Ad_{\sg_{k-1}^{-1}}(\beta'_k)\big)=0$ because $j_S^k(\beta'_k)=0$.
\vspace{2mm}\\
 Suppose $\kappa=\infty$. In this case Lemma 	\ref{sigma-beta'} yields  infinite sequences $(\beta'_i)_{i\geq 0}$, $(s_i)_{i\geq0}$ in  $\Gamma^{\infty}(U,\extp^{0,1}_{\; U}\otimes \Ad(P))$,  $\Gamma^{\infty}(U,\Ad(P))$ satisfying (\ref{jls,jlbeta'}) and (\ref{beta-l-sigma-l}) for $l\geq 0$ and (\ref{beta-l-beta'-(l-1)}) for $l\geq 1$. We define in the same way $\sigma_l$, $\sg_l\in \Gamma^\infty_\iota(U,\iota(P)_0)$, $s^l\in \Gamma^\infty(U,\Ad(P)_0)$.

For $l\geq 1$ put $s'_l\edf s^l-s^{l-1}$. Denoting by $\log: \iota(P)_0\textmap{\simeq} \Ad(P)_0$ the inverse of the fiber bundle isomorphism $\exp|_{\Ad(P)_0}:\iota(P)_0\textmap{\simeq} \Ad(P)_0$, we have
$$
s'_l=\log(\sg_l)-s^{l-1}=\log(\sigma_l\sg_{l-1})-s^{l-1}=\log(\exp(s_l)\exp(s^{l-1}))-s^{l-1}.
$$
Since the map $y\mapsto \log(\exp(y)\exp(s^{l-1}))-s^{l-1}$ vanishes at $y=0$ and $j_S^l(s_l)=0$, it follows by Lemma \ref{ABLemma} that $j_S^l(s'_l)=0$. Recalling that $j^0_S(s)=0$ and using Corollary \ref{extensionCoroCkappa} (3), it follows that there exists $s\in \Gamma^\infty(U,\Ad(P))$ such that for any $k\in\N$, $j_S^{k+1}((s_0+\sum_{l=1}^k s'_l)-s)=0$, i.e.
$$
j_S^{k+1}(s^k-s)=0.
$$
Using formula (\ref{tau*bar-dfJ(s)}) for $s^k$ and $s$ together with the second formula in (\ref{Omegaf-0-1}), we see that this implies 
\begin{equation}\label{jSk(barlg(exp(sk)-barlg(s))}
j_S^{k}(\bar\lg_J(\exp(s^k))-\bar\lg_J(\exp(s)))=0
\end{equation}
i.e. $j_S^{k}(\bar\lg_J(\sg_k)-\bar\lg_J(\exp(s)))=0$. By (\ref{new-formula-for-beta}) we infer
$$
j_S^{k}\big(\beta-\bar\lg_J(\exp(s))-\Ad_{\sg_{k-1}^{-1}}(\beta'_k))\big)=0,
$$
so $j_S^{k}\big(\beta-\bar\lg_J(\exp(s))=0$ because $j_S^{k}(\beta'_k)=0$.
\vspace{2mm}\\
 (2) For $l\geq 1$ the form 
$$\beta_l=\tilde\beta_{l-1}=\Ad_{\exp(s_{l-1})}(\beta'_{l-1})=
\Ad_{\exp(s_{l-1})}(\beta_{l-1}-\bar\dg_J(s_{l-1}))$$
 depends continuously on $\beta_{l-1}$,  $s_{l-1}$ and $J$.  On the other hand, by   \ref{sigma-beta'} (3), $s_{l-1}$ can be chosen to depend continuously   on $\beta_{l-1}$. By induction it follows that all $s_l$ (hence also $\sg_l$, $s^l$) can be chosen to depend continuously on the initial data $(J,\beta)$.

  \end{proof}

Suppose  that $S$ separates $U$, and let $U=\bar U^-\cup \bar U^+$ be the corresponding decomposition of $U$ as union of manifolds with boundary.
\begin{thr}\label{mainGnew}
Let $G$ be a complex  Lie  group   and $p:P\to U$  a principal $G$-bundle on $U$. Let $\kappa\in(0,+\infty]\setminus\N$ and $J^\pm$  formally integrable bundle ACS of class ${\cal C}^\kappa$ on $P_{\bar U^\pm}$ such that $J^+_S=J^-_S$. 
\begin{enumerate}
\item 	There exists 
\begin{enumerate}
\item  $\sigma_+\in\Gamma^{\kappa+1}(\bar U^+,\iota(P))$ with $\sigma_+|_S=e$, the unit element of $G$.
\item   an integrable bundle ACS $J$ of class ${\cal C}^\kappa$ on $P$, 
\end{enumerate}
such that $J|_{P^+}= J^+\cdot\sigma_+$ and $J|_{P^-}= J^-$.	
\item For any open neighborhood $V$ of $S$ in $U$, there exists  a pair $(\sigma_+,J)$ as above such that, moreover, $\sigma_+$ is constantly $e$ on $U^+\setminus V$.
\item If $\kappa\ne\infty$ the pair $(\sigma_+,J)$ can be chosen to depend continuously on $(J^-,J^+)$.
\end{enumerate}

\end{thr}

\begin{proof}
(1) 
Let $J_\pm$ a (not necessarily integrable) extension of $J^\pm$ to a bundle ACS of class ${\cal C}^\kappa$ on $P$.  The existence of such an extension is obtained using the affine  structure with model space $\Gamma^\kappa(\bar U^\pm,\extp^{0,1}_{\;\bar U^\pm}\otimes\Ad(P))$ of the space ${\cal J}^\kappa_{P_{\bar U^\pm}}$ and the extension principle given by  Corollary \ref{glueCkappa} (1).
Using the affine space structure of the space ${\cal J}^\kappa_P$ put $\beta\edf J_--J_+$ and note that the assumption $J^+_S=J^-_S$ implies $\beta\in  {\cal I}_S^{\kappa}(U,\extp^{0,1}_{\; U}\otimes \Ad(P))$.
\vspace{2mm}\\ 
Case (i): $\kappa>1$:  By formula   (\ref{fg-(J+b)}) of  section \ref{ACSsection} we have
$$
\kg_{J_+}(\beta)=\fg_{J_-}-\fg_{J_+}.
$$
We have $\fg_{J_\pm}|_{\bar U^\pm}=\fg_{J^\pm}=0$, because $J^\pm$ are assumed formally integrable. Since $\fg_{J_\pm}$ are of class ${\cal C}^{\kappa-1}$, this implies that  $j_x^{k-1}(\fg_{J_\pm})=0$ for any $x\in \bar U^\pm$, in particular for any $x\in S$. This proves that $j_S^{k-1}(\bar\kg_{J_+}(\beta))=0$ and $j_S^{k-1}(\fg_{J_+})=0$, in particular $j_S^{k-2}(\fg_{J_+})=0$.  Therefore Proposition \ref{prop-beta-v-new}   applies to the pair $(J_+,\beta)$ and gives a section $s\in\Gamma^{\kappa+1}(U,\Ad(P))$ with $s|_S=0$  such that, putting  $\sigma=\exp(s)$, we have $j^k_S(\beta-\bar\lg_{J_+}(\sigma))=0$. On the other hand we have:
$$
j^k_S(J_--J_+\cdot \sigma)=j^k_S(J_--J_++J_+-J_+\cdot \sigma)=j^k_S(\beta-\lg_{J_+}(\sigma)),
$$
where, for the last equality we have used formula   (\ref{Jsigma-J=lgsigma}) of section \ref{ACSsection}. Therefore $j^k_S(J_--J_+\cdot \sigma)=0$. By Corollary   \ref{glueCkappa} (2), there exists $J\in {\cal J}^\kappa_P$ which coincides with $J_-$ ((hence with $J^-$) on $\bar U^-$ and with $J_+\cdot \sigma$ (hence with $J^+\cdot \sigma|_{\bar U^+}$) on $\bar U^+$. $J$ is integrable, because $\fg_{J}$ coincides with $\fg_{J^-}$ on $\bar U^-$ and, by formula (\ref{fg-Jsigma}),  with  $\Ad_{\sigma^{-1}}(\fg_{J^+})$ on $\bar U^+$. It suffices to put $\sigma_+\edf\sigma|_{\bar U^+}$.
\vspace{2mm}\\ 
Case (ii): $\kappa\in (0,1)$. In this case the assumption ``$J^\pm$ is a formally integrable bundle ACS on $P_{\bar U^\pm}$" means that $\fg_{J^\pm}$ vanishes as distribution supported by $\bar U^\pm$, see section \ref{ACSsection}.
We apply Lemma \ref{sigma-beta'} (1) for $l=0$ recalling that the conditions imposed on $\beta$ in the hypothesis of this lemma reduce to $\beta\in  {\cal I}_S^{\kappa}(U,\extp^{0,1}_{\; U}\otimes \Ad(P))$. We obtain  as above $s\in\Gamma^{\kappa+1}(U,\Ad(P))$ with $s|_S=0$  such that $j^0_S(\beta-\bar\lg_{J_+}(\sigma))=0$ with   $\sigma=\exp(s)$. We conclude as in the Case (i), but making use of 
\begin{itemize}
\item[-] Remark \ref{equiv-formula-for-C0} to show that  $J^+\cdot \sigma|_{\bar U^+}$ is formally integrable, 
\item[-]   Proposition \ref{(J-,J+)->Jintegrable} to infer that $J$ is  integrable. 	
\end{itemize}
(2) Let $V'$ be an open neighborhoods of $S$ in $U$ such that $\bar V'\subset V$. By  the smooth version of Urysohn's lemma \cite[Lemma 1.3.2]{Pe}, it follows that there exists a ${\cal C}^\infty$ function $\lambda: U\to [0,1]$ such that $\lambda|_{\bar V'}\equiv 1$ and $\lambda|_{U\setminus V}\equiv 0$. It suffices to replace in the proof of (1) $s$ by $\lambda s$. 
\vspace{2mm}\\
(3) Using Proposition \ref{prop-beta-v-new} (2), Lemma \ref{sigma-beta'} (3) and the continuity properties of the extension operators in Corollary \ref{glueCkappa} it follows that the objects $J_\pm$, $s$, $J$ introduced in the proof of (1)  can be chosen to depend continuously on the input data $(J^-,J^+)$.	
\end{proof}
%
%
%
%
%
 We can now prove  Theorem \ref{new-main-sep-G} stated in the introduction:
 \begin{proof}
(1) Making use of Proposition \ref{CinftyEGamma}, let $\Sg\in \mathscr{S}_a$ be an admissible ${\cal C}^\infty$-structure on $P^\upsilon$ and $P^\upsilon_\Sg$ the corresponding ${\cal C}^\infty$ bundle. The obvious isomorphisms $o^\pm:P^\pm\to P^\upsilon_{\bar U^\pm}$ become isomorphisms $P^\pm\to P^\upsilon_{\Sg\bar U^\pm}$ of class ${\cal C}^{\kappa+1}$ between ${\cal C}^\infty$ bundles on $\bar U^\pm$, so the given formally integrable bundle ACS  $J^\pm$ of class ${\cal C}^\kappa$ on $P^\pm$ induce via $o^\pm$ formally integrable bundle ACS $J'^\pm$ of class ${\cal C}^\kappa$ on $P^\upsilon_{\Sg\bar U^\pm}$. The hypothesis ``$J^\pm_S$ agree via   $\upsilon$" in Theorem \ref{new-main-sep-G} is equivalent to the condition  $J'^-_S=J'^+_S$. 

Theorem \ref{mainGnew}  applies and gives    $\sigma_+\in \Gamma^{\kappa+1}(\bar U^+,P^\upsilon_{\Sg\,\bar U^+})$  with $\sigma_+|_S=e$, and an integrable bundle ACS $J$ of class ${\cal C}^\kappa$ on $P^\upsilon_\Sg$  which coincides with $J'^-$  on $P^\upsilon_{\Sg\bar U^-}$ and with $J'^+\cdot\sigma_+=(\tilde\sigma_+)^{-1}(J'^+)$ (see section \ref{ACSsection}) on $P^\upsilon_{\Sg\bar U^+}$.

By Theorem \ref{NNG},  $J$ defines a  holomorphic structure $\hg_J$ on  the underlying ${\cal C}^{\kappa+1}$ bundle   of $P^\upsilon_\Sg$; a local section is holomorphic with respect to $\hg_J$ if and only if it is $J$-pseudo-holomorphic (see also   \cite[Corollary 1.4]{Te2}).

The pair $(\id_{P^-},\tilde\sigma_+)$ defines an element $f\in\Aut^0(P^\upsilon)_a$, so $\Sg'\edf f(\Sg)$ also belongs to $\mathscr{S}_a$ by Lemma \ref{CinftyEGamma}. Therefore $f$ becomes a ${\cal C}^\infty$ bundle isomorphism   $P^\upsilon_{\Sg}\to P^\upsilon_{\Sg'}$.  The  direct image $\hg^\upsilon\edf f(\hg_J)$ will be a holomorphic structure on the  underlying ${\cal C}^{\kappa+1}$ bundle  of $P^\upsilon_{\Sg'}$ which   coincides with $\hg_\pm$ on $U^\pm$ via $o^\pm$, because $f(J)$ coincides with $J'^\pm$ on $U^\pm$. 

We now prove the unicity property claimed  in (1): Let $\hg'$, $\hg''$ be    holomorphic structures (see \cite[Definition 1.3]{Te2}) on  the topological bundle $P^\upsilon$  which extend $\hg^\pm$. Let $\tau': V'\to P^\upsilon_{V'}$, $\tau'':  V'' \to P^\upsilon_{V''}$ be local sections which are holomorphic with respect to $\hg'$, respectively $\hg''$.  The restrictions
$$
\tau': V'\cap  U^\pm\to   P^\pm_{V'\cap  U^\pm}, \ \tau'':V''\cap  U^\pm\to P^\pm_{V''\cap  U^\pm}	$$
are  holomorphic sections  of the holomorphic bundle $(P^\pm_{U^\pm},\hg^\pm)$. The corresponding comparison map
$$
g_{\tau'\tau''}:V'\cap V''\to G 
$$
is continuous on the whole $V'\cap V''$ and holomorphic on both $V'\cap V''\cap U^{\pm}$, i.e. on $(V'\cap V'')\setminus S$. By the extension Theorem \ref{acrossS} it follows that $g_{\tau'\tau''}$ is holomorphic on the whole $V'\cap V''$, so $\tau'$, $\tau''$ are holomorphically compatible, so they belong to the same holomorphic structure on $P^\upsilon$. 	
\vspace{2mm}

(2) Any $\hg^\upsilon$-holomorphic local section $\tau:V\to P^\upsilon$ is a section of class ${\cal C}^{\kappa+1}$ of $P^{\Sg'}$, so, since $\Sg'\in \mathscr{S}_a$, its restrictions to $V\cap\bar U^\pm$ will be of class ${\cal C}^{\kappa+1}$.
 \end{proof}

 Theorem \ref{new-main-non-sep-G} follows easily from Theorem \ref{new-main-sep-G} taking into account that: 
 \begin{enumerate}
 \item 	Any oriented smooth hypersurface $S\subset U$ separates a sufficiently small open neighborhood $U_0\supset S$ of $S$ in $U$.
 \item The problem has a local character with respect to $S$.
 \end{enumerate}
 Theorems \ref{new-main-non-sep}, \ref{new-main}  follow from Theorems \ref{new-main-non-sep-G}, respectively \ref{new-main-sep-G} taking $G=\GL(r,\C)$. Using Theorem \ref{new-main} we can prove now Corollary \ref{loc-free-sheaves}.
\begin{proof}(of Corollary \ref{loc-free-sheaves}):

By the extension Theorem \ref{acrossS}, the locally free sheaf ${\cal E}^\upsilon$ of locally defined holomorphic sections of $(E^\upsilon,\hg^\upsilon)$ coincides with the sheaf
$$
U\stackrel{\hb{\tiny open}}{\supset} W\mapsto \{f\in \Gamma^0(W,E^\upsilon)|\ \sigma|_{W\setminus S} \hb{ is }\hg^\upsilon-\hb{holomorphic}\}, 
$$
which (taking into account that $\hg^\upsilon$ extends $\hg^\pm$ and the definition of $E^\upsilon$) coincides with the sheaf defined by formula (\ref{first-sheaf}).  By Theorem \ref{new-main} (2), the restrictions of any local $\hg$-holomorphic section $f\in \Gamma^0(W,E^\upsilon)$ to $W\cap\bar U^\pm$ are of class ${\cal C}^{\kappa+1}$ up to the boundary, so formulae  (\ref{first-sheaf}), (\ref{second-sheaf}) define the same sheaf, as claimed.
\end{proof}

Corollary \ref{loc-free-sheaves-non-sep} follows from Corollary \ref{loc-free-sheaves} taking into account again that any oriented smooth hypersurface $S\subset U$ is locally (with respect to $S$) separating.

\section{Isomorphism theorems. Interpretation in terms  of framed bundles}\label{iso-moduli-section}

In this section we come back to the objects considered in section \ref{Iso-Moduli-G-section}: let $X$ be a {\it closed} complex manifold, $S\subset X$   a closed, smooth, oriented real hypersurface, $P$ a principal $G$ bundle on $X$ and $\widehat{P}$ its pull back to $\widehat{X}_S$.

\subsection{The proofs of the isomorphism theorems}
\label{proofs-section}

We begin with the following remark which will be used in  the proof of Theorem \ref{iso-moduli-G-bundles-th}:

\begin{re}\label{cont-desc}
Any gauge transformation $\fg\in   {\cal G}_{\p\widehat{X}}(\widehat{P})$ descends to a continuous gauge transformation $\check{\fg}$ on $P$ which  is of class ${\cal C}^{\kappa+1}$ on $X\setminus S$ and is identity on $S$.	
\end{re}

\begin{proof} (of  Theorem \ref{iso-moduli-G-bundles-th}). The second claim of the theorem is a special case of the first, so we will prove only the first.\\
{\it Injectivity:} Let $J_1$, $J_2\in {\cal I}^\kappa_P$ and $\fg\in {\cal G}_{\p\widehat{X}}^{\widehat{P}}=\Gamma^{\kappa+1}(\widehat{X},\iota(\widehat{P}))$ be such that $\widehat J_2= \widehat J_1\cdot\fg$, where $\widehat J_i$ is the pull back of $J_i$ to $\widehat{P}$.  It follows that $J_2= J_1\cdot \check{\fg}$ on $X\setminus S$. 

Let $G\times G$ act on $G$ from the left by
$$
\mu((a,b),g)=agb^{-1}
$$
and  note that $\iota(P)\edf P\times_\iota G$ can be identified with the associated bundle 
$$
\mu(P\times_X P)\edf (P\times_X P)\times_\mu G.
$$
The pair $(J_1,J_2)$ defines an integrable bundle ACS of class ${\cal C}^\kappa$ on $P\times_X P$, so a  holomorphic structure $\hg_{(J_1,J_2)}$  on the principal $G\times G$-bundle $P\times_X P$. The known   
\begin{prop}  $J_2= J_1\cdot \check{\fg}$ on $X\setminus S$.	
\end{prop}
\noindent is equivalent to:  
\begin{prop} $\check{\fg}$, regarded as a section in the  bundle $\mu(P\times_X P)$, is holomorphic   with respect to the holomorphic structure (induced by) $\hg_{(J_1,J_2)}$, on $X\setminus S$.
	
\end{prop}

By Remark \ref{cont-desc}, $\check{\fg}$ is continuous on $X$, and by Corollary \ref{ContSect}  it follows that $\check{\fg}$ is in fact holomorphic with respect to  $\hg_{(J_1,J_2)}$ on whole $X$. Using Corollary \ref{kappa-regularity} we infer that $\check{\fg}$ is of class ${\cal C}^{\kappa+1}$ on $X$, and the relation $J_2=J_1\cdot \check{\fg}$ extends to $X$.
\vspace{2mm}\\
{\it Surjectivity:} Let $\Jg\in {\cal I}^\kappa_{\widehat{P}\downarrow}$ be a descendable formally integrable bundle ACS on $\widehat{P}$.  We have to prove the existence of a pair $(J,\sg)\in {\cal I}^\kappa_P\times  {\cal G}_{\p\widehat{X}}^{\widehat{P}}$  such that $\Jg\cdot \sg=\widehat J$.

 Let 
$$
S\times\R\textmap{\nu\simeq }U\hookrightarrow X
$$
be a tubular neighborhood of $S$ which is compatible with the orientation of its normal bundle associated with the complex orientation of $U$ and the fixed orientation of  $S$. Put 
$$U^\pm=\nu(S\times\R^*_\pm),\ \bar U^\pm=\nu(S\times\R_\pm).$$
 The disjoint union $\bar U^-\coprod\bar U^+=\widehat U_S$ is a neighborhood of $\widehat S=\p\widehat X_S$ in $\widehat X_S$, so the restriction of $\Jg$ to this neighborhood gives   formally integrable bundle ACS $J^\pm$ of class ${\cal C}^\kappa$ on $P_{\bar U^\pm}$. The       assumption ``$\Jg$ is descendable" is equivalent to the condition $J^-_S=J^+_S$. 
 
 By Theorem \ref{mainGnew} there exists  $\sigma_+\in \Gamma^{\kappa+1}(\bar U^+, \iota(P_{\bar U^+}))$ which is constantly $e$ on $S\cup  \nu([1,+\infty))$ and an integrable bundle ACS $J_0$ of class ${\cal C}^\kappa$ on $P_U$ which coincides with $J^-$ on $E_{\bar U^-}$ and with $J^+\cdot \sigma_+$ on $P_{\bar U^+}$.
 
 We define  $\sigma_0\in\Gamma^{\kappa+1}(\widehat U_S,\iota(\widehat{P}))$ using the constant section $e$ on $\bar U^-$ and $\sigma_+$ on $\bar U^+$. Since $\sigma_0$ is  constantly $e$ above $\widehat U_S\setminus \widehat{\nu(-1,1)}_S$, it extends to $\widehat{X}$ giving  a global   section $\sg\in \Gamma^{\kappa+1}(\widehat{X},\iota(\widehat{P}))$ which is constantly $e$ on $S$ and satisfying
 $$
 \Jg\cdot \sg =\widehat J
 $$
 for an integrable bundle ACS $J$ of class ${\cal C}^\kappa$ on $P$ which coincides with $J_0$ on $P_U$.   
 
The pull-back map $J\mapsto \widehat J$ is obviously continuous. Using Theorem \ref{mainGnew} we also infer that, for $\kappa\ne\infty$, $\sigma_+$ (so also $\sigma_0$  and $\sg$) can be chosen to depend continuously on $(J^-,J^+)$, so, with this choice, $J$ will depend continuously on $\Jg$. This proves the continuity of the inverse of the pull-back map.
 \end{proof}

 Remark \ref{kappa=infty} concerning the case $\kappa=\infty$ follows from the following simple
 \begin{lm}\label{XYXiYi}
Let $X$, $Y$ be sets and, for any $i\in I$, let $X_i$, $Y_i$ be topological spaces, and $f_i:X_i\to Y_i$,  $a_i:X\to X_i$, $b_i:Y\to Y_i$ be maps such that the diagrams 
 $$\begin{tikzcd}
 X_i \ar[r, "f_i"]& Y_i\\
 X\ar[r,"f"]\ar[u, "a_i"] &Y	\ar[u,"b_i"']
 \end{tikzcd}$$
are commutative. Endow $X$ ($Y$) with the coarsest topology which makes all maps $a_i$ (respectively $b_i$) continuous. Then
\begin{enumerate}
\item If all $f_i$ are continuous, then $f$ is continuous.
\item If all $f_i$ are homeomorphisms and $f$ is bijective, then $f$ is a homeomorphism.	
\end{enumerate}	
 \end{lm}
 \begin{proof}
 (1) follows from the universal property of the initial topology on $Y$ defined by the family of maps $(b_i)_{i\in I}$.	For (2) put $g\edf f^{-1}$, $g_i\edf f_i^{-1}$ and note that for any $i\in I$, 
 $$f_i\circ a_i\circ g=b_i\circ f\circ g=b_i=f_i\circ g_i\circ b_i,$$
 so, since $f_i$ is injective,  we have $a_i\circ g=g_i\circ b_i$. Therefore, since all $g_i$ are continuous, it follows by (1) that $g$ is continous. 
 \end{proof}

 Theorem \ref{iso-moduli-vector-bundles-th}  is a special case of Theorem \ref{iso-moduli-G-bundles-th}.

\subsection{Interpretation in terms of framed bundles} 
\label{abstract-interpr}

The moduli spaces ${\cal M}_S(P)$, ${\cal M}_{\p\bar X}(P)$  intervening    in Theorem \ref{iso-moduli-G-bundles-th} have ``abstract" interpretations in terms of isomorphism classes of framed (formally) holomorphic bundles: 
 \begin{dt}\label{abstract-framed}
 Let $X$ be a closed complex manifold,  $S\subset X$ a closed real hypersurface and  $\Phi$  a fixed ${\cal C}^\infty$  $G$-bundle on $S$ (a framing bundle).

 An $S$-framed $G$-bundle of type $(\Phi, \kappa+1)$ on $X$  is a pair $({\cal P},\theta)$, where ${\cal P}$ is a holomorphic  $G$-bundle on $X$ and $\theta:\Phi\to {\cal P}_S$ is a  bundle isomorphism of class ${\cal C}^{\kappa+1}$ on $S$. 	
 
 An    isomorphism $({\cal P},\theta)\to ({\cal P}',\theta')$ of $S$-framed holomorphic bundles   of type $(\Phi, \kappa+1)$ is a  holomorphic   isomorphism $f:{\cal P}\to  {\cal P}'$ such that $f_S\circ\theta=\theta'$. 
 
Let $\Phi$ be ${\cal C}^\infty$ $G$-bundle on  the boundary $\p\bar X
$ of a compact complex manifold with boundary $\bar X$. 

A boundary framed  formally holomorphic   bundle  of type $(\Phi, \kappa+1)$ on $\bar X$ is a triple $(P,J,\theta)$, where $P$ is a ${\cal C}^\infty$ $G$-bundle on $\bar X$, $J$ is a formally integrable  bundle ACS of class ${\cal C}^\kappa$ on $P$, and $\theta:\Phi\to P_{\p\bar X}$ is a  bundle isomorphism of class ${\cal C}^{\kappa+1}$ on $\p\bar X$. 

An  isomorphism $(P,J,\theta)\to (P',J',\theta')$ of  boundary framed formally holomorphic bundles   of type $(\Phi, \kappa+1)$ is a  pseudo-holomorphic   isomorphism $f:(P,J)\to (P',J')$ of class ${\cal C}^{\kappa+1}$ on $\bar X$ such that $f_{\p\bar X}\circ\theta=\theta'$. 
 \end{dt}

 In the special case when $G=\GL(r,\C)$ and  $\Phi=S\times\C^r$, one recovers the notions of an $S$-framed, respectively boundary framed bundle as used in \cite[Theorem 1']{Do} and explained in the introduction of this article.
 
Comparing the two definitions note that, whereas a holomorphic $G$-bundle on a closed complex manifold has a canonical ${\cal C}^\infty$-structure and any holomorphic isomorphism of holomorphic bundles is ${\cal C}^\infty$, this is no longer true for formally holomorphic bundles  and formally holomorphic isomorphisms on manifolds with boundary.
 
Let $P$ be a ${\cal C}^\infty$ $G$-bundle on $X$,  $\Phi$ a ${\cal C}^\infty$ $G$-bundle on $S$ which is isomorphic to $P_S$, and $\theta_0:\Phi\to P_S$ a fixed bundle isomorphism of class ${\cal C}^\infty$.

By Theorem \ref{NNG} (see also   \cite{Te2}), a bundle ACS $J$ of class ${\cal C}^\kappa$ on $P$ defines a holomorphic reduction $\hg_J$ of the underlying  ${\cal C}^{\kappa+1}$-bundle of $P$.  We obtain a holomorphic bundle ${\cal P}_J=(P,\hg_J)$ and the identity isomorphism $\id_P:P\to {\cal P}_J$ is an isomorphism of class ${\cal C}^{\kappa+1}$ between ${\cal C}^\infty$-bundles, so $\theta_0:\Phi\to P_S$ becomes a bundle isomorphism of class ${\cal C}^{\kappa+1}$ if $P_S$ is endowed with the ${\cal C}^\infty$ structure induced by the holomorphic structure of ${\cal P}_J$.
 The pair $({\cal P}_J,\theta_0)$ is an $S$-framed holomorphic bundle  of type $(\Phi,{\kappa+1})$ on $X$.

 Similarly, let  $P$ be a ${\cal C}^\infty$ $G$-bundle on $\bar X$, $\Phi$ be a ${\cal C}^\infty$ $G$-bundle on $\p\bar X$ which is isomorphic to $P_{\p\bar X}$ and $\theta_0:\Phi\to P_{\p\bar X}$ a fixed bundle isomorphism of class ${\cal C}^\infty$.
  \begin{pr}\label{abstr-to-concrete}
   With the notations and definitions above
  \begin{enumerate}
  \item Let $P$ be a ${\cal C}^\infty$ $G$-bundle on $X$. The assignment
  $$
  J\cdot {\cal G}_S^P\mapsto \hb{ the isomorphism class of } ({\cal P}_{J},\theta_0)
  $$
  gives a bijection  between   the moduli space ${\cal M}_S(P)$   and the set ${\cal M}_S(P,\theta_0)$ of isomorphism  classes of $S$-framed holomorphic bundles of type $(\Phi,{\kappa+1})$  on $X$ which are topologically isomorphic to $(P,\theta_0)$.
  \item Let $P$ be a ${\cal C}^\infty$ $G$-bundle on $\bar X$. The assignment
  $$
J\cdot   {\cal G}_{\p\bar X}^P \mapsto \hb{ the isomorphism class of } (P,J,\theta_0)
  $$
  gives a bijection  between   the moduli space ${\cal M}_{\p\bar X}(P)$ and the set ${\cal M}_{\p\bar X}(P,\theta_0)$ of isomorphism  classes of boundary framed holomorphic bundles of type $(\Phi,{\kappa+1})$  on $\bar X$ which are topologically isomorphic to $(P,\theta_0)$.
	
  \end{enumerate}

 \end{pr}
 
 \begin{proof}

 (1) Injectivity: Let $J$, $J'\in {\cal J}^\kappa_P$.  An isomorphism $f:({\cal P}_J,\theta_0)\to ({\cal P}_{J'},\theta_0)$ in the sense of Definition \ref{abstract-framed}  is an holomorphic isomorphism $f:{\cal P}_J\to {\cal P}_{J'}$ such that $f_S\circ\theta_0=\theta_0$, i.e. such that $f_S=\id_{P_S}$. On the other hand, using Corollary \ref{kappa-regularity} as in the proof of  Theorem \ref{iso-moduli-G-bundles-th}, we see that, since $J$, $J'$ are of of class ${\cal C}^\kappa$,  $f$ is of class ${\cal C}^{\kappa+1}$. Therefore $f\in {\cal G}_S^P$. On the other hand, the holomorphy of $f:{\cal P}_J\to {\cal P}_{J'}$ means $J=J'\cdot f$, so $J\cdot {\cal G}_S^P= J'\cdot {\cal G}_S^P$.
  
 For the surjectivity, let $({\cal P},\theta)$ be an $S$-framed holomorphic $G$-bundle of type $(\Phi, \kappa+1)$ on $X$ which is topologically isomorphic to   $(P,\theta_0)$. Therefore  there exists a topological bundle isomorphism $g: P\to {\cal P}$ such that $g_S\circ\theta_0=\theta$. 

Recall that the differentiable and topological classifications of principal bundles on differentiable manifolds coincide, so $P$, ${\cal P}$ are also isomorphic as differentiable bundles. Let $g_0:P\to {\cal P}$ be a ${\cal C}^\infty$ isomorphism which is sufficiently close to $g$ in the ${\cal C}^0$-topology such that $(g_0^{-1}\circ g)_S$ takes values in the disk bundle $\iota(P_S)_0$ obtained  by applying Proposition \ref{fiberwise-exp} to the bundle $P_S$. Since $(g_0^{-1}\circ g)_S=g_{0S}^{-1}\circ \theta\circ\theta_0^{-1}$ is of class  ${\cal C}^{\kappa+1}$, it follows by this proposition that $(g_0^{-1}\circ g)$ can be written as $\exp(\lambda)$ for a section $\lambda\in \Gamma^{\kappa+1}(S,\Ad(P))$.

 By Corollary \ref{extensionCoroCkappa} (for $m=0$) there exists an extension $\mu\in \Gamma^{\kappa+1}(X,\Ad(P))$ of $\lambda$. The bundle isomorphism $f=g_0\exp(\mu): P\to {\cal P}$ is of class ${\cal C}^{\kappa+1}$ and extends $g_S=\theta\circ\theta_0^{-1}$. The pull back $J\edf f^{-1}(J_{\cal P})$ of the canonical bundle ACS $J_{\cal P}$ of ${\cal P}$ is an integrable bundle ACS of class ${\cal C}^\kappa$ on $P$ and $f$ gives and isomorphism $(P_J,\theta_0)\to ({\cal P},\theta)$ of $S$-framed holomorphic $G$-bundles of type $(\Phi,\kappa+1)$ on $X$.
 \vspace{2mm}\\
 (2) The injectivity is clear. For the surjectivity let $(Q,I,\theta)$ be a  boundary framed formally holomorphic bundle of type $(\Phi, \kappa+1)$ on $\bar X$ which is topologically isomorphic to   $(P,\theta_0)$. Therefore there exists a topological  bundle isomorphism $g:P\to Q$ such that $g_{\p\bar X}\circ\theta_0=\theta$; in other words $g$ is a continuous extension of $\theta\circ \theta_0^{-1}$. We use the same method as above to replace $g$ by an 	extension  $f:P\to Q$ of $\theta\circ \theta_0^{-1}$ which is of class ${\cal C}^{\kappa+1}$. Putting $J\edf f^{-1}(I)$ we see that $f$ is an isomorphism $(P,J,\theta_0)\to (Q,I,\theta)$ of boundary framed formally holomorphic bundles of type $(\Phi,\kappa+1)$.
 \end{proof}
 \begin{re} In terms of  abstract boundary framed formally holomorphic bundles,  the descendibility condition if Definition \ref{descendable} becomes:   Let $\Phi$ be a ${\cal C}^\infty$ bundle on $S$ and $\Phi_{\widehat S}=\Phi_{S^-}\cup \Phi_{S^+}$ its bull-back to $\p\widehat{X}_S=\widehat S=S^-\cup S^+$. A boundary framed formally holomorphic bundle  $(\Qg,\Ig,\theta)$ of type $(\Phi_{\widehat S},\kappa+1)$ on $\widehat{X}_S$ is  descendable  if and only if the tangential almost complex structures $\Ig_{S^\pm}$  induced by $\Ig$ on $\Phi_{S^\pm}$  via $\theta$ agree via the obvious bundle  isomorphism $\Phi_{S^-}\to b^*(\Phi_{S^+})$.
 
Similarly, if $S$ separates $X$, a pair $((Q_-,I_-,\theta_-),(Q_+,I_+,\theta_+))$ of boundary framed formally holomorphic bundles of type $(\Phi,\kappa+1)$ on $\bar X^\pm$ corresponds to a point in the  fiber product ${\cal M}_{\p\bar X^-}(P^-)\times_{\cal I}{\cal M}_{\p\bar X^+}(P^+)$ intervening in Theorem \ref{iso-moduli-G-bundles-th}, if and only if $I_\pm$ induce the same  tangential almost complex structures on $\Phi$ via $\theta_\pm$ and the identifications $\p\bar X^\pm=S$.
 \end{re}

\section{Examples}
\label{ExamplesSection}

Throughout this section we fix $\kappa\in (0,+\infty]\setminus\N$ and   a connected complex Lie group $G$. Let   $X$ a Riemann surface, and $Y\subset X$ a connected open subset whose closure $\bar Y$ is a compact surface with smooth, non-empty boundary $\p\bar Y=\bar Y\setminus Y$.  

\begin{pr}	
\label{BundleACSOnDisk}
For any ${\cal C}^\infty$ principal $G$-bundle $P$ on $\bar Y$ and bundle ACS  $J$ of class ${\cal C}^\kappa$  on $P$, there exists a $J$-pseudo-holomorphic section $\tau_0\in \Gamma(\bar Y,P)_{\kappa+1}$. In other words, for any such pair $(P,J)$ there exists a pseudo-holomorphic bundle isomorphism $(\bar Y\times G, J_0)\to (P,J)$ of class ${\cal C}^{\kappa+1}$, where $J_0$ is the standard  bundle ACS on the trivial bundle  $\bar Y\times G$.
\end{pr} 
\begin{proof} Since $\p\bar Y\ne\emptyset$, $\bar Y$ has the homotopy type of a bouquet of circles. Taking into account that $G$ is connected, it follows that any topological (differentiable) $G$ bundle on $\bar Y$ is trivial, so we may suppose that $P=\bar Y\times G$. Let $N$ be a tubular neighborhood of $\p\bar Y$ in $X$ and $\tilde Y\edf \bar Y\cup N$. Therefore $\tilde Y$ is an open neighborhood of $\bar Y$ in $X$ which is homotopically equivalent to $\bar Y$.

 The bundle ACS $J$ is defined by a form $\alpha_J\in \Gamma^\kappa(\bar Y,\extp^{0,1}_{\;\tilde Y}\otimes\g)$ (see section \ref{ACSsection}).  
By the extension Corollary \ref{glueCkappa} there exists an extension $\tilde \alpha\in \Gamma^\kappa(\tilde Y,\extp^{0,1}_{\;\tilde Y}\otimes\g)$ of $\alpha_J$. The form $\tilde \alpha$ corresponds to a bundle ACS $\tilde J$ of class ${\cal C}^\kappa$ on $\tilde Y\times G$ which extends $J$.  By the Newlander-Nirenberg Theorem \ref{NNG}, $\tilde J$ defines a holomorphic structure on the underlying ${\cal C}^{\kappa+1}$ bundle of $\tilde Y\times G$. This structure is trivial by Grauert's classification theorem of holomorphic bundles on Stein manifolds \cite{Gra}, so it admits a global holomorphic section $\tilde\tau_0$. It suffices to put $\tau_0\edf \tilde \tau_0|_{\bar Y}$.
\end{proof}

Note that any topological $G$-bundle on $\p\bar Y$ is also trivial so, with the notations of section \ref{abstract-interpr}, it's natural to take $\Phi=\p\bar Y\times G$ as framing bundle on $\p\bar Y$. In other words, in this section, by a boundary framing of a $G$ bundle on $\bar Y$ we will always mean a trivialization, or, equivalently, a section of its restriction to  $\p\bar Y$. 

Consider now the special case when  $\bar Y$ is a disk $\bar D\subset X$. Isomorphism classes of  boundary framed topological $G$-bundles on $\bar D$ correspond bijectively to  homotopy classes  $\chi\in [\p\bar D,G]$  of maps $\theta:\p\bar D\to G$. Since $\pi_1(G,e)$ is Abelian, the obvious map $\pi_1(G,e)\to [\p\bar D,G]$ is injective, so $[\p\bar D,G]$ has a natural Abelian group structure.  Endowing $\p\bar D$ with its boundary orientation (induced by the complex orientation of $\bar D$), this set can be further identified with $H_1(G,\Z)$ via the map
$$
[\theta]\mapsto \deg(\theta)\edf H_1(\theta)([\p\bar D]).
$$

For a class   $h\in H_1(G,\Z)$ we will denote by   $h_{\bar D}$ the corresponding isomorphism class of  boundary framed topological $G$-bundles on $\bar D$ and by ${\cal M}_{\p\bar D}^{\bar D}(h)$ the moduli space  of boundary framed formally holomorphic $G$ bundles of class ${\cal C}^\kappa$ in this class.  By Proposition \ref{BundleACSOnDisk} we obtain:
\begin{co}\label{M{pbar D}{bar D}}
Let $h\in H_1(G,\Z)$. We have a natural identification $$
{\cal M}_{\p\bar D}^{\bar D}(h)\simeq \qmod{{\cal C}^{\kappa+1}_h(\p\bar D,G)}{H^{\kappa+1}(\bar D,G)},
$$ 
where ${\cal C}^{\kappa+1}_h(\p\bar D,G)$ is the space of ${\cal C}^{\kappa+1}$ maps $\p\bar D\to G$  of degree $h$, and $H^\kappa(\bar D,G)$ is the group of ${\cal C}^{\kappa+1}$ maps $\bar D\to G$ which are holomorphic on $D$.
\end{co}

\begin{re}\label{LoopsRem}
Suppose that $G$ is reductive, and let $K\subset G$ be a maximal compact subgroup of $G$. In this case the canonical map
$$\qmod{{\cal C}^{\kappa+1}_h(\p\bar D,K)}{K}\to \qmod{{\cal C}^{\kappa+1}_h(\p\bar D,G)}{H^{\kappa+1}(\bar D,G)}$$
is an isomorphism. For the standard disk this is a well known factorization theorem in loop group theory \cite[chapter 8]{PS}, whereas the general case follows using \cite[Theorem 1']{Do}.
\end{re}

\subsection{Holomorphic bundles framed along a circle in \texorpdfstring{$\P^1_\C$}{P1C}} 
\label{framed-along-S-in-P1}

Let  now $S\subset \C$ be a smooth closed curve and $\P^1_\C=\bar U^-\cup\bar U^+$ be the corresponding decomposition of $\P^1_\C$ as union of closed disks, where  $\bar U^-\cap\bar U^+=S$, $0\in U^-$, $\infty\in U^+$. Note that the identifications $S=\p\bar U^\pm$ induce on $S$ opposite orientations.

Let $(P,\theta)$ be a topological $S$-framed principal $G$-bundle on $\P^1_\C$. 
Choose sections $\tau^\pm$ of the restrictions  $P^\pm\edf P_{\bar U^\pm}$ and let $g:\p\bar U^-\to G$ be the comparison map defined by $\tau^+_S=\tau^-_S g$.  

The homotopy degree $e(P)\edf\deg(g)\in H_1(G,\Z)$   is independent of the pair $(\tau^-,\tau^+)$; it is a topological invariant of $P$;   isomorphism classes of topological $G$-bundles over $\P^1_\C$ correspond bijectively to elements $e\in  H_1(G,\Z)$ via this invariant. For a section $\theta\in \Gamma(S,P)$ we define the maps $f^\pm_\theta:S\to G$  by   $\theta= \tau^\pm_S f^\pm_\theta$; these maps satisfy the identity $f^+_\theta=g^{-1}f^-_\theta$ and $h^\pm_\theta\edf\deg(f^\pm_\theta)\in H_1(G,\Z)$  are topological invariants of the  framed bundle $(P,\theta)$.
\begin{re}
Isomorphism classes of $S$-framed topological $G$-bundles on $\P^1_\C$ correspond bijectively to pairs $(e,h)\in H_1(G,\Z)\times H_1(G,\Z)$ via the map
$$
(P,\theta)\mapsto (e(P),h^-_\theta).
$$
\end{re}
For a pair   $(e,h)\in H^1(G,\Z)\times H^1(G,\Z)$ we denote by $(e,h)_{\P^1_\C}$ the corresponding isomorphism class of $S$-framed topological bundles on $\P^1_\C$, and by  ${\cal M}^{\P^1_\C}_S(e,h)$  the moduli space of $S$-framed holomorphic bundles of class ${\cal C}^\kappa$ on $\P^1_\C$ belonging to this isomorphism class.
By  Theorem \ref{iso-moduli-G-bundles-th} and Proposition \ref {abstr-to-concrete} we obtain the decomposition:
\begin{equation}\label{M(e,h)}
{\cal M}^{\P^1_\C}_S(e,h)={\cal M}_{\p\bar U^-}^{\bar U^-}(h)\times {\cal M}^{\bar U^+}_{\p\bar U^+}(e-h).
\end{equation}
Consider now the case $G=\C^*$ and identify $H_1(\C^*,\Z)$ with $\Z$ in the standard way.   ${\cal M}^{\P^1_\C}_S(e,h)$ is just the moduli space of pairs $(L,\theta)$ consisting of  a holomorphic line bundle $L$ of degree $e$ on $\P^1_\C$ and a nowhere vanishing section $\theta$ of class ${\cal C}^{\kappa+1}$ and degree $h$ (with respect to a trivialization on $\bar U^-$) of $L_S$. 

Any holomorphic line bundle  of degree $e$ on $\P^1_\C$ is isomorphic to $|{\cal O}_{\P^1_\C}(e)|$. We trivialize  over $\P^1_\C\setminus\{\infty\}$  (respectively $\P^1_\C\setminus\{0\}$) the line bundle $|{\cal O}_{\P^1_\C}(e)|$ using $\varphi_0^{\otimes e}$ (respectively $\varphi_1^{\otimes e}$), where $\varphi_i:\C^2\to \C$ is the linear form  defined by $\varphi_i(Z_0,Z_1)=Z_i$. Since $\Aut(|{\cal O}_{\P^1_\C}(e)|)=\C^*\id$, we obtain an obvious identification
\begin{equation}
{\cal M}^{\P^1_\C}_S(e,h)\simeq \qmod{{\cal C}^{\kappa+1}_h(\p\bar U^-,\C^*)}{\C^*}.	
\end{equation}
This isomorphism combined with the decomposition (\ref{M(e,h)}) and Corollary \ref{M{pbar D}{bar D}} gives the homeomorphism  
\begin{equation}\label{iso-quotients}
\begin{tikzcd} 
\qmod{{\cal C}^{\kappa+1}_h(\p\bar U^-,\C^*)}{\C^*}\ar[r, "\Psi_{e,h}", "\simeq"'] & \begin{array}{c}	  {\cal C}^{\kappa+1}_h(\p\bar U^-,\C^*)/ H^{\kappa+1}(\bar U^-,\C^*)\\ \bigtimes \\   {\cal C}^{\kappa+1}_{e-h}(\p\bar U^+,\C^*) / H^{\kappa+1}(\bar U^+,\C^*) 	
 \end{array}
 \end{tikzcd}
\end{equation}
given explicitly by
$$
 [f]_{\C^*}\mapsto \left( \big[f\big]_{H^{\kappa+1}(\bar U^-,\C^*)}\,,\,   \big[(\varphi_1\varphi_0^{-1})^{-e} f\big]_{H^{\kappa+1}(\bar U^+,\C^*)}\right).
$$
We are interested in an explicit formula for the inverse of this map.  Note that the map $f\mapsto (\varphi_1\varphi_0^{-1})^{-e} f$ induces an isomorphism 
$${\cal C}^{\kappa+1}_{e-h}(\p\bar U^+,\C^*) / H^{\kappa+1}(\bar U^+,\C^*)\to {\cal C}^{\kappa+1}_{-h}(\p\bar U^+,\C^*) / H^{\kappa+1}(\bar U^+,\C^*),$$
so it suffices to consider the case $e=0$. We will see that the inverse of $\Psi_{0,g}$ can be written down explicitly using the Cauchy transform  and classical results in holomorphic function theory. Recall first that the Cauchy transform 
$$u\mapsto C^S(u), \ C^S(u)(z)=\frac{1}{2\pi i}\int_S \frac{u(\zeta)}{\zeta-z}d\zeta$$
 associated with the oriented closed curve $S=\p\bar U^-$ defines continuous operators 
 $$C^S_\pm: {\cal C}^{\kappa+1}(S,\C)\to {\cal C}^{\kappa+1}(\bar U^\pm,\C)$$
 (see \cite[Section 2.22]{Mu} and \cite[Theorem 1.10 p. 22, formula (3.3) p. 23]{Ve}) satisfying the Plemelj-Privalov formula
 \begin{equation}\label{PlemPriv}
 C^S_-(u)|_S-C^S_+(u)|_S=u	
 \end{equation}
(see \cite[formula (17.3) p. 43]{Mu}).

Let $f^-$, $f^+\in {\cal C}^{\kappa+1}(S,\C^*)$ be  maps  of   degree $h$ (respectively $-h$) with respect to 0 (respectively $\infty$). Therefore   $\deg(f^\pm)=h$ with respect to 0. Let $\varphi\in   {\cal C}^{\kappa+1}(S,\C)$ be such that $\exp(\varphi)=f^+/f^-$. By (\ref{PlemPriv}) we obtain $e^{C^S_-(\varphi)|_S- C^S_+(\varphi)|_S}=f^+/f^-$, so
$$
e^{C^S_-(\varphi)}|_Sf^-= e^{C^S_+(\varphi)}|_Sf^+.
$$
Noting that $e^{C^S_-(\varphi)}\in H^{\kappa+1}(\bar U^\pm,\C^*)$ and putting   $f\edf e^{C^S_-(\varphi)}|_S f^-=e^{C^S_+(\varphi)}|_S f^+$ we see that $\C^* f$ is the pre-image of the pair $\big([f^-]_{H^{\kappa+1}(\bar U^-,\C^*)},  [f^+]_{H^{\kappa+1}(\bar U^+,\C^*)}\big)$ via $\Psi_{0,h}$. Therefore $\Psi_{0,h}^{-1}$ is given by the explicit formula:
\begin{equation}\label{InverseOfPsi}
\Psi_{0,h}^{-1}\big(\big[f^-\big]_{H^{\kappa+1}(\bar U^-,\C^*)}\,,\,  \big[f^+\big]_{H^{\kappa+1}(\bar U^+,\C^*)}\big)=\big[e^{C^S_\pm(\log(f^+/f_-)}|_S f^\pm\big]_{\C^*}.	
\end{equation}

\begin{re}
Combining the isomorphism (\ref{iso-quotients}) with Remark \ref{LoopsRem}, 	we obtain an isomorphism
$$
\qmod{{\cal C}^{\kappa+1}_h(\p\bar U^-,\C^*)}{\C^*}\textmap{\simeq}\qmod{{\cal C}^{\kappa+1}_h(\p\bar U^-,S^1)}{S^1} \times \qmod{{\cal C}^{\kappa+1}_{e-h}(\p\bar U^+,S^1)}{S^1}.
$$
\end{re}
This is a typical example of identification obtained by combining the isomorphism Theorem \ref{iso-moduli-G-bundles-th} with Donaldson's Theorem 1'.

\subsection{Holomorphic bundles framed along a circle in an elliptic curve}
\label{Framed-Circle-Elliptic}

As in Example \ref{RHOnElliptic} of section \ref{RHonRiemannSurf-section}, let $\alpha\in\C^*$ with $|\alpha|<1$  and $X=\C^*/\langle \alpha\rangle$ be the associated elliptic curve; let $\bar D\subset \C$ be a smooth compact disk such that $\alpha\bar D\subset D$, $\Omega\edf D\setminus \alpha\bar D$, $\bar\Omega\edf \bar D\setminus \alpha  D$. Endow the curves $S^+ \edf\p\bar D$, $S^-\edf\alpha S^+ =\p  (\alpha \bar D)$ with their boundary orientations. 

As noticed above, since we assumed $G$ connected, any differentiable $G$ bundle on $\bar\Omega$ ($\p\bar\Omega$) is trivial.  Taking as in the previous section  $\Phi=\p\bar\Omega\times G$ as framing bundle, we see that the data of a topological boundary framing of the trivial bundle $\bar\Omega\times G$ is equivalent to the data of a pair $(\tau^+,\tau^-)$ of continuous maps $\tau^\pm: S^\pm \to G$. %
\begin{re}
The formula $[(\bar\Omega\times G,\tau^+,\tau^-)]\mapsto \deg(\tau_+)-\deg(\tau_-)$ gives a bijection between isomorphism classes of boundary framed topological $G$-bundles on $\bar\Omega$ and $H_1(G,\Z)$. 	
\end{re}

For a class   $h\in H_1(G,\Z)$ we  denote by   $h_{\bar \Omega}$ the corresponding isomorphism class of  boundary framed topological $G$-bundles on $\bar \Omega$ and by ${\cal M}_{\p\bar \Omega}^{\bar \Omega}(h)$ the moduli space  of boundary framed formally holomorphic $G$ bundles of class ${\cal C}^\kappa$ in this class.

 Let $H^\kappa(\bar\Omega,G)$   be the the group of ${\cal C}^{\kappa+1}$ maps $\bar \Omega\to G$ which are holomorphic on $\Omega$ and $H^\kappa_m(\bar\Omega,G)\edf\{f\in H^\kappa(\bar\Omega,G)|\ \deg(f)=m\}$. Using  Proposition \ref{BundleACSOnDisk} again we obtain:
\begin{co}\label{M{pbar Omega}{bar Omega}}
Let $h\in H_1(G,\Z)$ and $n\in\Z$. We have  natural identifications 
\begin{equation}
\begin{split}
{\cal M}_{\p\bar \Omega}^{\bar \Omega}(h)&\simeq
 \big({\coprod_{m\in\Z}{\cal C}^{\kappa+1}_m(S^+,G)\times {\cal C}^{\kappa+1}_{m-h}(S^-,G)}\big)/{H^{\kappa+1}(\bar \Omega,G)}\\
& =\big({{\cal C}^{\kappa+1}_n(S^+,G)\times   {\cal C}^{\kappa+1}_{n-h}(S^-,G)}\big)/{H^{\kappa+1}_0(\bar \Omega,G)}.	
\end{split}	
\end{equation}
\end{co}

Suppose now that $G$ is reductive,  let $K\subset G$ be a maximal compact subgroup of $G$ and let $M_{\p\bar \Omega}^{\bar \Omega}(h)$ be the moduli space of boundary framed flat $K$-connections of topological type $h_{\bar\Omega}$ and class ${\cal C}^\kappa$ modulo the gauge group ${\cal C}^{\kappa+1}(\bar\Omega,K)$.
Using \cite[Theorem 1']{Do} it follows that the canonical map $M_{\p\bar \Omega}^{\bar \Omega}(h)\to {\cal M}_{\p\bar \Omega}^{\bar \Omega}(h)$ is a homeomorphism. The moduli space $M_{\p\bar \Omega}^{\bar \Omega}(h)$ can be easily described as follows: 

Identify $\Omega$ with the quotient $\bar\Og/H$, where $c:\bar\Og\to \bar\Omega$ is a universal cover of $\bar\Omega$ and $H\edf\Aut_{\bar\Omega}(\bar\Og)$. Let $h_0\in H$ be the generator of $H$ which corresponds to the generator of positive degree of the fundamental group of $\bar\Omega$. For any $a\in K$ let $\tilde a:H\to K$ be the group morphism  which maps $h_0$ to $a$. Put $\Sg^\pm\edf c^{-1}(S^\pm)$ and note that the space
$$
\Ag\edf  \{(a, \tg^+,\tg^-)|\ \tg^\pm\in {\cal C}^{\kg+1}_{\tilde a}(\Sg^\pm, K)\}
%
$$
%
%
comes with a natural free $K$-action given by:
$$
b\cdot (a, \tg^+,\tg^-)=(bab^{-1}, b\tg^+,b\tg^-).
$$
Let $\fg:\bar\Sg^+\to\bar\Sg^-$ be a $H$-equivariant lift of the diffeomorphism $S^+\ni z\mapsto\alpha z\in  S^-$ and note that, for $\tg^\pm\in {\cal C}^{\kg+1}_{\tilde a}(\Sg^\pm, K)$, the product  $(\tg^-\circ \fg)^{-1}\tg^+:\Sg^+\to K $ is $H$-invariant, so it descends to a map $[(\tg^-\circ \fg)^{-1}\tg^+]:S^+\to K$ whose degree $\deg([(\tg^-\circ \fg)^{-1}\tg^+])\in H_1(K,\Z)=H_1(G,\Z)$ is independent of the choice of $\fg$. The subspace
$$
\Ag(h)\edf \big\{(a, \tg^+,\tg^-)|\ \tg^\pm\in {\cal C}^{\kg+1}_{\tilde a}(\Sg^\pm, K),\ \deg([(\tg^-\circ \fg)^{-1}\tg^+])=h\big\}\subset \Ag
$$
is  $K$-invariant.
Let  $(a, \tg^+,\tg^-)\in \Ag$. The principal $K$-bundle $P_a\edf \bar\Og\times_{\tilde a} K$ comes with a canonical flat connection $A_a$ and the maps $\tg^\pm$ can be interpreted as  sections of class ${\cal C}^{\kappa+1}$ of $P_a|_{S^\pm}$. 
\begin{re} The map %
$${\Ag(h)}/{K}\to M^{\bar\Omega}_{\p\bar\Omega}(h),\ K\cdot(a,\tg^+,\tg^-)\mapsto [A_a,\tg^+,\tg^-]$$
 is a homeomorphism. 
 \end{re}
 This remark gives a simple description of the Donaldson moduli space $M^{\bar\Omega}_{\p\bar\Omega}(h)$ of boundary framed flat $K$-connections on an annulus.

On the other hand, note that the $G$-bundle $P_a^\C\edf \bar\Og\times_{\tilde a} G$ comes also with a flat connection, so with a canonical bundle ACS $J_a$.  Making use of Proposition  \ref{BundleACSOnDisk}, let $\tg\in \Gamma(\bar\Omega, P_a^\C)_{\kappa+1}$ be a $J_a$-pseudo-holomorphic section and let $\tau^\pm:\bar\Omega\to G$ be the maps defined by the formulae $\tg^\pm =\tg \tau^\pm$.  The pair $(\tau^+,\tau^-)$ is independent of the choice of $\tau$ up to the $H^{\kappa+1}(\bar \Omega,G)$ action. In conclusion,  combining Corollary \ref{M{pbar Omega}{bar Omega}} with \cite[Theorem 1']{Do} we obtain

\begin{re}
We have a natural homeomorphism
$$
 {\Ag(h)}/{K}\textmap{\simeq}  \bigg(\coprod_{m\in\Z}{\cal C}^{\kappa+1}_m(S^+,G)\times {\cal C}^{\kappa+1}_{m-h}(S^-,G)\bigg)/{H^{\kappa+1}(\bar \Omega,G)}
$$ 
given explicitly by $K\cdot (A_a,\tg^+,\tg^-)\mapsto H^{\kappa+1}(\bar \Omega,G)(\tau^+,\tau^-)$.
\end{re}

Our next goal is to make explicit the isomorphism given by   Theorem \ref{iso-moduli-G-bundles-th} and its inverse in the special case when $X$ is the elliptic curve $\C^*/\langle \alpha\rangle$, $S$ is the image of $S^\pm$ in $X$ and $G=\C^*$. Note that $\widehat X_S$ can be identified with $\bar\Omega$. Isomorphism classes of $S$-framed $\C^*$-bundles over $X$ correspond bijectively to isomorphism classes of $\C^*$-bundles on $X$.  This follows taking into account that  the restriction map ${\cal C}(X,\C^*)\to {\cal C}(S,\C^*)$ is surjective, so the automorphism group of a topological $\C^*$-bundle $P$ on $X$ acts transitively on the space of continuous sections of $P_S$. For $e\in\Z$ let $e_{\P^1_\C}$ be the isomorphism class of $S$-framed topological $\C^*$-bundles $(P,s)$ with $\deg(P)=e$, and let ${\cal M}_S^{X}(e)$ be the corresponding moduli space. 

Putting $\Pic^e(X)\edf\{[L]\in\Pic(X)|\ \deg(L)=e\}$,  we have:
\begin{re}
The natural map ${\cal M}_S^{X}(e)\to\Pic^e(X)$, $[(P,s)]\mapsto [P\times_{\C^*}\C]$	is a principal bundle with structure group ${\cal C}^{\kappa+1}(S,\C^*)/\C^*$.
\end{re}
Taking into account Theorem \ref{iso-moduli-G-bundles-th}, Proposition \ref {abstr-to-concrete} and Corollary \ref{M{pbar Omega}{bar Omega}}, we obtain:
\begin{co}
We have a natural homeomorphism 	
$$\Psi_e:{\cal M}_S^{X}(e)\textmap{\simeq} \big({\coprod_{m\in\Z}{\cal C}^{\kappa+1}_m(S^+,\C^*)\times {\cal C}^{\kappa+1}_{m-e}(S^-,\C^*)}\big)/{H^{\kappa+1}(\bar \Omega,\C^*)}$$
defined as follows: for an $S$-framed holomorphic $\C^*$-bundle $(P,s)$ on $X$, let $(\hat P,\hat s)$ be the pull-back boundary framed  formally holomorphic bundle on $\bar\Omega$, $\hat s^\pm\edf \hat s|_{S^\pm}$ and let $\tau$ be a pseudo-holomorphic section of $\hat P$. Then $\Psi_e([P,s])= (f_+,f_-)$, where $f^\pm:S^\pm\to \C^*$ are defined by the formulae $\hat s^\pm=\tau f^\pm$. \end{co}
Choosing $\tau$ such that $\deg(f^+)=n$ in the definition of $\rho$, we obtain a homeomorphism
\begin{equation}
\Psi_{e,n}:  {\cal M}_S^{X}(e)\textmap{\simeq}  \big({{\cal C}^{\kappa+1}_n(S^+,\C^*)\times {\cal C}^{\kappa+1}_{n-e}(S^-,\C^*)}\big)/{H^{\kappa+1}_0(\bar \Omega,\C^*)},
\end{equation}
which is an analogue for elliptic curves of the homeomorphism (\ref{iso-quotients}) obtained by applying Theorem \ref{iso-moduli-G-bundles-th} to $\P^1_\C$. We are interested in an explicit formula for the inverse 
$$\Psi_{0,n}^{-1}:\big({{\cal C}^{\kappa+1}_n(S^+,\C^*)\times {\cal C}^{\kappa+1}_{n}(S^-,\C^*)}\big)/{H^{\kappa+1}_0(\bar \Omega,\C^*)}\to  {\cal M}_S^{X}(0)$$
corresponding to the special case $e=0$. Let $f^\pm\in {\cal C}^{\kappa+1}_n(S^\pm,\C^*)$ and let $\varphi\in {\cal C}^{\kappa+1}(S^+,\C)$ be such that for any  $z\in S^+$, we have $e^\varphi(z)=f^+(z)f^-(\alpha z)^{-1}$. With the notations introduced in the previous section, put 
$$
\Cg^{S^+}_-(\varphi)\edf C^{S^+}_-(\varphi)- C^{S^+}_-(\varphi)(0),$$
and  define $\psi:\bar\Omega\to\C^*$ by the formula:\footnote{The idea to define $\psi$ in this way and formula (\ref{psi(z)-psi(alpha z)}) are due to Alexander Borichev \cite{Bor}.}
$$ 
\psi(z)=\sum_{k=0}^\infty \Cg^{S^+}_-(\varphi)(\alpha^k z)+\sum_{k=1}^\infty   C^{S^+}_+(\varphi)(\alpha^{-k} z).
$$
Noting that $\Cg^{S^+}_-(\varphi)(0)=0$ by the definition of $\Cg^{S^+}_-(\varphi)$, and $C^{S^+}_+(\infty)=0$ by \cite[p. 23]{Mu}, it follows using Lemma \ref{MaxPrinc} proved below that both series in the definition of $\psi$ are normally convergent on $\bar\Omega$. Moreover, writing 
$$
\psi(z)=\Cg^{S^+}_-(\varphi)(z)+C^{S^+}_+(\varphi)(\alpha^{-1} z)+\sum_{k=1}^\infty \Cg^{S^+}_-(\varphi)(\alpha^k z)+\sum_{k=2}^\infty   C^{S^+}_+(\varphi)(\alpha^{-k} z),
$$
using the properties of the Cauchy transforms $C^{S^+}_\pm$ mentioned in the previous section, and noting that the two sums on the right extend holomorphically to a neighborhood of $\bar\Omega$, it follows that $\psi\in H^{\kappa+1}(\bar\Omega,\C)$. 

For any $z\in\bar\Omega$ we have
\begin{equation}\label{psi(z)-psi(alpha z)}
\psi(z)-\psi(\alpha z)
%
%
%
=\Cg^{S^+}_-(\varphi)(z)-C^{S^+}_+(\varphi)(z)=\varphi(z)-C^{S^+}_-(\varphi)(0),
\end{equation}
where, for the last equality, we used the Plemelj-Privalov formula (\ref{PlemPriv}). Putting $\lambda\edf e^{C^{S^+}_-(\varphi)(0)}\in\C^*$, $g^\pm\edf e^{-\psi}|_{S^\pm} f^\pm$, this implies
\begin{equation}\label{exp(psi)}
\forall z\in S^+,\ g^-(\alpha z)=\lambda g^+(z).
\end{equation}
Let $P_\lambda$ be the flat holomorphic $\C^*$-bundle over $X=\C^*/\langle\alpha\rangle$ defined by
$$
P_\lambda\edf \C^*\times\C^*/\langle(\alpha,\lambda)\rangle=\bar\Omega\times\C^*/\stackrel{\lambda}{\sim}, 
$$
where $\stackrel{\lambda}{\sim}$ is the equivalence relation generated by the set of pairs
$$\{((z^+,\zeta),(\alpha z^+,\lambda\zeta))|\ z^+\in S^+,\ \zeta\in\C^*\}.$$
 Formula (\ref{exp(psi)}) shows that $(g^+,g^-)$ defines a section $g\in \Gamma(S,P_\lambda)^{\kappa+1}$, and that
$$
\Psi_{0,n}([P_\lambda,g])= [g^+,g^-]_{H^{\kappa+1}_0(\bar\Omega,{\C}^*)}.
$$  
On the other hand, the definition of $g^\pm$  gives $(g^+,g^-)=e^{-\psi}(f^+,f^-)$, where $e^{-\psi}\in H^{\kappa+1}_0(\bar\Omega,\C^*)$, so
$$
[f^+,f^-]_{H^{\kappa+1}_0(\bar\Omega,\C^*)}=[g^+,g^-]_{H^{\kappa+1}_0(\bar\Omega,{\C}^*)}.
$$
Therefore
\begin{equation}\label{PSI{0n}-1}
\Psi_{0,n}^{-1}\big([f^+,f^-]_{H^{\kappa+1}_0(\bar\Omega,\C^*)}\big)=[P_\lambda,g].	
\end{equation}

\begin{lm} \label{MaxPrinc} Let $r>0$ and  $u$ be a holomorphic function defined on a neighborhood of the standard compact disk $\bar D_r\subset \C$ such that $u(0)=0$. 
\begin{enumerate}
	\item For any $z\in \bar D_r$ we have
$$|u(z)|\leq r^{-1}\sup_{\zeta\in   S_r}|u(\zeta)| |z|.$$
\item Let $\alpha\in\C^*$ with $|\alpha|<1$. For any $(z,k)\in\C\times\N$ such that $\alpha^kz\in \bar D_r$ we have 
$$|u(\alpha^kz)|\leq r^{-1} |\alpha|^k \sup_{\zeta\in   S_r}|u(\zeta)||z|.$$
\end{enumerate} 
\end{lm}
\begin{proof}
For (1) apply the Maximum Principle	to the holomorphic extension of the function $z\mapsto z^{-1}u(z)$ on $\bar D_r$. (2) follows directly from (1).
\end{proof}

\subsection{\texorpdfstring{$S$}{S}-framed holomorphic \texorpdfstring{$\SL(2,\C)$}{SL}-bundles on \texorpdfstring{$\P^1_\C$}{P1}}

\label{framed-P1-SL(2,C)}

We come back to the decomposition $\P^1_\C=\bar U^-\cup \bar U^+$ associated with a closed curve $S\subset\C$ as considered in section \ref{framed-along-S-in-P1}. We are interested in the moduli space of $S$-framed $\SL(2,\C)$-bundles on $\P^1_\C$. We will use the vector bundle formalism, so in this section by $\SL(2,\C)$-bundle we mean a holomorphic vector bundle of rank 2 endowed with a trivialization of its determinant line bundle. By Grothendieck's classification theorem \cite{Gro}  the map
$$
\N\ni n\mapsto |{\cal O}(n)|\oplus |{\cal O}(-n)|
$$
is a bijection onto the set of isomorphism classes of $\SL(2,\C)$-bundles on $\P^1_\C$. In the above formula we used the notation $|{\cal L}|$ for the line bundle associated with an invertible sheaf ${\cal L}$.  Denoting by $\C[Z_0,Z_1]_{d}$ the space of homogeneous polynomials of degree $d$ in $Z_0$, $Z_1$, note that  

$$
\Aut(|{\cal O}(n)|\oplus |{\cal O}(-n)|)=\left\{ \begin{array}{ccc}
\SL(2,\C) &\rm if &n=0,\\
\left\{\begin{pmatrix} a & P\\ 0 & a^{-1}\end{pmatrix}\vline\ a\in\C^*,\ P\in \C[Z_0,Z_1]_{2n}\right\}&\rm if&n>0.
\end{array}\right.
$$
On the affine line $\C\subset\P^1_\C$ we trivialize  the line bundles $|{\cal O}(1)|$, $|{\cal O}(-1)|$ using respectively the linear form $\varphi_0$ defined in  section \ref{framed-along-S-in-P1} and the meromorphic section $\xi_0$ of the tautological line bundle $|{\cal O}(-1)|$ given by 
$$\P^1\ni\xi=[Z_0,Z_1]\mapsto \big(1,\frac{Z_1}{Z_0}\big)\in|{\cal O}(-1)|_\xi.$$
The matrix of 
$A=\begin{pmatrix} a & P\\ 0 & a^{-1}\end{pmatrix}$
with respect to the basis $(\varphi_0^{\otimes n},\xi_0^{\otimes n})$  is
$\Ag=\begin{pmatrix} a & p|_S\\ 0 & a^{-1}\end{pmatrix}$
where $p\in \C[z]_{\leq 2n}$, $p(z)=P(1,z)$ is the dehomogenization  of $P$ with respect to $Z_0$.  We obtain:
\begin{pr}
The moduli space ${\cal M}_S$ of $S$-framed $\SL(2,\C)$-bundles on $\P^1_\C$ admits a natural stratification ${\cal M}_S=\coprod_{n\in\N} {\cal M}_S^n$, where
$$
{\cal M}_S^0= {{\cal C}^{\kappa+1}(S,\SL(2,\C))}/{\SL(2,\C)},$$
$$ {\cal M}_S^n=  {{\cal C}^{\kappa+1}(S,\SL(2,\C))}\bigg/{\left\{\begin{pmatrix} a & p\\ 0 & a^{-1}\end{pmatrix}\vline\ a\in\C^*,\ p\in \C[z]_{\leq 2n}\right\}} \hb{ for }n\geq 1.
$$
For any $n\in\N$, ${\cal M}_S^n$ is open in $\overline{{\cal M}_S^n}=\union_{m\geq n}{\cal M}_S^m$.
\end{pr}
Therefore any stratum ${\cal M}_S^n$ is an infinite dimensional homogeneous Banach manifold obtained by factorizing a Banach Lie group by a finite dimensional  affine algebraic subgroup. 
Theorem \ref{iso-moduli-G-bundles-th} gives a homeomorphism
$$
{\cal M}_S\textmap{\simeq} \begin{array}{c} {{\cal C}^{\kappa+1}(\p\bar U^-,\SL(2,\C))}/{H^{\kappa+1}(\bar U^-,\SL(2,\C))}\\ \bigtimes\\ {{\cal C}^{\kappa+1}(\p\bar U^+,\SL(2,\C))}/{H^{\kappa+1}(\bar U^+,\SL(2,\C))} 	
 \end{array}
$$
induced by the obvious restriction map. Combining this result with \cite[Theorem 1']{Do}  applied to the two factors on the right, we obtain:
\begin{co}
The product 	
$$\qmod{{\cal C}^{\kappa+1}(\p\bar U^-,\SU(2))}{\SU(2)}\times \qmod{{\cal C}^{\kappa+1}(\p\bar U^+,\SU(2))}{\SU(2)}$$
can be identified with the moduli space ${\cal M}_S=\coprod_{n\in\N} {\cal M}_S^n$ of $S$-framed holomorphic $\SL(2,\C)$-bundles on $\P^1_\C$.\end{co}

\section{Appendix}

\subsection{Lipschitz spaces, spaces of maps and sections of class \texorpdfstring{${\cal C}^\kappa$}{kappa} }\label{Ckappa}

In this section we will introduce the spaces:
 $\mathrm{Lip}^\kappa(\R^n,T)$, $\mathrm{Lip}^\kappa_{\R^n}(F,T)$, ${\cal C}^\kappa(U,T)$, ${\cal C}^\kappa(\bar U,T)$, $\Gamma^\kappa(U,E)$, $\Gamma^\kappa(\bar U,E)$. 
 
 Let $T$ be a finite dimensional normed space, $k\in\N$ and $f\in {\cal C}^k(\R^n,T)$. The order $k$ remainder of $f$ is the map $\R^n\times\R^n\to T$  defined by 
 $$
  R_f^k(x,y)\edf  f(x)-\sum_{0\leq|l|\leq k} \frac{1}{l!}\p^lf(y) (x-y)^{l}.
 $$
 Using the integral formula  for the order $k-1$ Taylor remainder, we obtain
\begin{equation}
\begin{split}
R_f^k(x,y)&=R_f^{k-1}(x,y)-\sum_{|l|=k} \frac{1}{l!}\p^lf(y)(x-y)^{l}\\
&=k\int_0^1 (1-t)^{k-1}\sum_{|l|=k}\frac{1}{l!}\big[ \p^lf( y+t(x-y))- \p^lf(y)\big](x-y)^l dt,	
\end{split}	
\end{equation}
which gives the estimate
\begin{equation}\label{RemainderEstimate}
\big\|R_f^k(x,y)\big\|\leq c(n,k)\| \sup_{\substack{t\in[0,1]\\ |l|=k}}\|\p^lf(y+t(x-y))-\p^lf(y)\|\|x-y\|^k.	
\end{equation}
Applied to the ${\cal C}^{k-|j|}$ map $\partial^j f$ for $|j|\leq k$, formula (\ref{RemainderEstimate}) gives 
\begin{equation}\label{RemainderEstimate-j}
\big\|R_{\partial^j f}^{k-|j|}(x,y)\big\|\leq c(n,k-|j|)\| \sup_{\substack{t\in[0,1]\\ |l|=k-|j|}}\|\p^{j+l}f(y+t(x-y))-\p^{j+l}f(y)\|\|x-y\|^{k-|j|}	
\end{equation}

 Let now $\kappa\in (0,+\infty)\setminus\N$.   We  denote by $\mathrm{Lip}^\kappa(\R^n,T)$ the order $\kappa$ Lipschitz space in supremum norm, as defined in  \cite[p. 2]{JW},  \cite[p. 176]{St}:
 \begin{equation}\label{defLipkappa}
 \mathrm{Lip}^\kappa(\R^n,T) \edf \{f\in {\cal C}^{[\kappa]}(\R^n,T)|\ \| f\|_{\mathrm{Lip}^\kappa}<\infty\},
 \end{equation}
 where
\begin{equation}
\begin{split}\label{DefLipkappa-norm}
\| f\|_{\mathrm{Lip}^\kappa} \edf    \inf & \big\{  M\in\R_+|\ \sup_{\R^n}\|\p^j f\|\leq M, \hb{ for }|j|\leq [\kappa], \hb{ and}\\
 &\|\p^jf(x)-\p^jf(y)\|\leq M\|x-y\|^{\kappa-[\kappa]} \hb{ for } |j|=[\kappa],\  x,\ y\in\R^n\big\}.
\end{split}	
\end{equation}
Using formulae (\ref{RemainderEstimate}), (\ref{RemainderEstimate-j}) for $k=[\kappa]$, it follows that 
\begin{re}\label{EstimatesForLipkappa}
For any $f\in \mathrm{Lip}^\kappa(\R^n,T)$ and any $j\in\N^n$ with $|j|\leq[\kappa]$ we have an estimate of the form:
$$
\big\|R_{\partial^j f}^{[\kappa]-|j|}(x,y)\big\|\leq M_j \|x-y\|^{\kappa-|j|}.
$$	
\end{re}

This justifies the following definition (see \cite[p. 176]{St}, \cite[p. 22]{JW} for $\R$-valued functions, and \cite[Definition B1]{BBHM} for maps with values in a Banach space):
\begin{dt}
Let $\kappa\in (0,+\infty)\setminus\N$ and $F\subset\R^n$ be a closed subset.	A $T$-valued Whitney jet of order $\kappa$ on $F$ is a family $\fg=(f^{(j)})_{0\leq|j|\leq[\kappa]}$ of bounded continuous maps $f^{(j)}:F\to T$ such that, putting
$$
R_{j,\fg}^{[\kappa]}(x,y)\edf f^{(j)}(x)-\sum_{0\leq |l|\leq [\kappa]-|j|} \frac{1}{l!} f^{(j+l)}(y)(x-y)^l,
$$
we have estimates of the form $\|R_{j,\fg}^{[\kappa]}(x,y)\|\leq M_j \| x-y\|^{\kappa-|j|}$ on $F\times F$.  
\end{dt}

Endowed with the norm
$$
\|\fg\|_{\mathrm{Lip}^\kappa}\edf \inf\left \{M\in \R_+\vline\ \begin{array}{l}\|f^{(j)}(x)\|\leq M,\\ \|R_{j,\fg}^{[\kappa]}(x,y)\|\leq M \| x-y\|^{\kappa-|j|}\end{array} \begin{array}{c}\hb{for any }(x,y)\in F\times F, \\ j\in\N^n \hb{ with }|j|\leq[\kappa]\end{array}  \right\},
$$
the space  $\mathrm{Lip}^\kappa_{\R^n}(F,T)$ of  $T$-valued Whitney jets of order $\kappa$ on $F$ becomes a Banach space. The role of the subscript $_{\R^n}$ in our notation is to avoid confusion with the space $\mathrm{Lip}^\kappa(F,T)$ in the sense of (\ref{defLipkappa}) in the special case when $F$ is an affine subspace of $\R^n$ (in which case $F$ can be identified with a space $\R^m$ with $m\leq n$).

We refer to \cite[Theorem 4, p. 177]{St}, \cite[Theorem, p. 23]{JW} \cite[Theorem B.2]{BBHM} for the following fundamental:
\begin{thr}\label{WhitneyExt} (Whitney extension theorem for jets  of order $\kappa$)
  Let $\kappa\in (0,+\infty)\setminus\N$ and $F\subset\R^n$  a closed set. There exists a continuous  extension operator 
$${\cal E}_\kappa: \mathrm{Lip}^\kappa_{\R^n}(F,T)\to \mathrm{Lip}^\kappa(\R^n,T)$$
such that, putting $f\edf {\cal E}_\kappa(\fg)$, we have $\p^j f|_F=f^{(j)}$ for any $j\in\N^n$ with $|j|\leq [\kappa]$ and $f|_{\R^n\setminus F}\in{\cal C}^\infty(\R^n\setminus F,T)$.
\end{thr}
For the ${\cal C}^\infty$ property of $f$ on  $\R^n\setminus F$, see the comments of \cite[p. 173, 179]{St}.

 Let $H\subset\R^n$ be an open half-space bounded by an affine hyperplane $S\subset \R^n$. We endow the space
 $$\mathrm{Lip}^\kappa(\bar H,T)\edf \{f\in {\cal C}^0(\bar H,T)|\ \exists \tilde f\in \mathrm{Lip}^\kappa(\R^n ,T) \hb{ such that } \tilde f|_{\bar H}=f\}$$
 with the quotient norm induced by the obvious linear isomorphism
 $$
 \qmod{\mathrm{Lip}^\kappa(\R^n ,T)}{\{\varphi\in \mathrm{Lip}^\kappa(\R^n ,T)|\ \varphi|_{\bar H}=0\}}\textmap{\simeq}\mathrm{Lip}^\kappa(\bar H,T).
 $$
 
 For  $f\in \mathrm{Lip}^\kappa(\bar H,T)$ and $j\in\N^n$ with $|j|\leq [\kappa]$ we put $\partial^jf \edf \p^j\tilde f|_{\bar H}$, where $\tilde f\in  \mathrm{Lip}^\kappa(\R^n ,T)$ is an extension of $f$ to $\R^n$.  Note that, by Whitney extension Theorem \ref{WhitneyExt} and \cite[Corollary 1 p. 42]{JW}), the space $\mathrm{Lip}^\kappa(\bar H,T)$ can also be identified with the space $\mathrm{Lip}^\kappa_{\R^n}(\bar H,T)$ of  Whitney jets of order $\kappa$ on $\bar H$ via the map 
 $$
 \mathrm{Lip}^\kappa_{\R^n}(\bar H,T)\ni \fg=(f^{(j)})_{|j|\leq [\kappa]}\mapsto f^{(0)}=f\in \mathrm{Lip}^\kappa(\R^n ,T). $$
Via this identification we have $\p^j f=f^{(j)}$. Similarly, for $m\in\N$, the Fréchet space ${\cal C}^m(\bar H,T)$ can be identified with the Fréchet space of  $T$-valued Whitney jets of class ${\cal C}^m$ on $\bar H$ (see \cite[section 1.1]{FJW} for the Fréchet structure on the space of Whitney jets of class ${\cal C}^m$). By Whitney extension Theorem for Lipschitz   spaces,  the original Whitney extension for ${\cal C}^k$-spaces (\cite{Wh}, \cite{FJW}) and Seeley's extension theorem \cite{See} for ${\cal C}^\infty$-spaces,  we have
 \begin{pr}\label{ExtFromHalfSpace}
 For $\kappa\in (0,+\infty)\setminus\N$ there exists a continuous extension operator $ \mathrm{Lip}^\kappa(\bar H,T)\to \mathrm{Lip}^\kappa(\R^n,T)$. For $m\in\N\cup\{\infty\}$ there exists a continuous extension operator $ {\cal C}^m(\bar H,T)\to {\cal C}^m(\R^n,T)$.
 \end{pr}
Put $\R^n_\pm\edf \R^{n-1}\times \R_\pm$, where $\R_\pm\edf \pm[0,+\infty)$.
\begin{lm}\label{glueLip-alpha}
Let $\alpha\in (0,1)$, $M_\pm\in\R_+$ and let $F_\pm:\R^{n}_\pm\to T$ be such that $F_+|_{\R^{n-1}\times\{0\}}=F_-|_{\R^{n-1}\times\{0\}}$ and 
$$
\|F_\pm(x)-F_\pm(y)\|\leq M_\pm \|x-y\|^\alpha \ \forall (x,y)\in (\R^n_\pm)^2.
$$
 Let $F$  be the mutual extension of $F_\pm$ to $\R^n$. 
  Then    %
  $$\|F(x)-F(y)\| \leq 2^{1-\alpha}\max(M_-,M_+)\|x-y\|^\alpha\ \forall (x,y)\in (\R^n)^2.
  $$
\end{lm}
\begin{proof} Note first that
$$
\sup_{\substack{(x,y)\in (\R^n_\pm)^2\\ x\ne y}} \frac{1}{\|x-y\|^\alpha} \| F(x)-F(y)\|\leq \max(M_-,M_+).
$$

It remains to estimate $\| F(x)-F(y)\|$ in terms of $\|x-y\|^\alpha$ when $x=x_+\in \R^n_+$ and $y=x_-\in \R^n_-$.
Let $x_0\in [x_-,x_+]\cap(\R^{n-1}\times\{0\})$. We have
$$
\|F(x_+)-F(x_-)\|\leq \|F(x_+)-F(x_0)\|+\|F(x_0)-F(x_-)\|\leq 
$$
$$
\leq M_+\|x_+-x_0\|^\alpha+M_-\|x_0-x_-\|^\alpha\leq \max(M_-,M_+)(\|x_+-x_0\|^\alpha+\|x_0-x_-\|^\alpha).
$$
Using standard estimates between the norms $\|\cdot\|_p$ ($1\leq p\leq+\infty$) on $\R^n$ we obtain:
$$
\|x_+-x_0\|^\alpha+\|x_0-x_-\|^\alpha=\|(\|x_+-x_0\|^\alpha,\|x_0-x_-\|^\alpha)\|_1\leq
$$
$$ 2^{1-\alpha}\|(\|x_+-x_0\|^\alpha,\|x_0-x_-\|^\alpha)\|_{\frac{1}{\alpha}}= 2^{1-\alpha}(\|x_+-x_0\|+\|x_0-x_-\|)^\alpha=2^{1-\alpha}\|x_+-x_-\|^\alpha.
$$
\end{proof}
\begin{pr}\label{glueLip-kappa}
Let $\kappa\in (0,\infty)\setminus\N$. Let $F_\pm\in \mathrm{Lip}^{\kappa}(\R^n_\pm,T)$  be such that 
$$\partial^j F_-|_{\R^{n-1}\times\{0\}}=	\partial^j F_+|_{\R^{n-1}\times\{0\}} \hb{ for }|j|\leq [\kappa],$$
and let $F$ be the mutual extension  of  $F_\pm$ to $\R^{n}$.  Then  $F\in \mathrm{Lip}^{\kappa}(\R^n)$, and 
$$\| F\|_{\mathrm{Lip}^{\kappa}}\leq  2^{1-\alpha} \max(\| F_-\|_{\mathrm{Lip}^{\kappa} },\| F_+\|_{\mathrm{Lip}^{\kappa}}).$$
\end{pr}

\begin{proof}
For   $j\in \N^n$ with $|j|\leq [\kappa]$ let   $F^j$ be the mutual extension of $\p^jF_\pm$ to $\R^n$. We prove  that
\begin{cl}
For any $j\in\N^n$ with $|j|\leq [\kappa]-1$,  $F^j$ is differentiable and $\p_i F^j=F^{j+e_i}$ for  $1\leq i\leq n$.	
\end{cl}
The claim is clear on $\R^n\setminus(\R^{n-1}\times \{0\})$, so let $y\in \R^{n-1}\times \{0\}$.  We know that
$$
\|F_\pm ^j(x_\pm)-\sum_{|j+l|\leq [\kappa]}\frac{1}{l!}F^{j+l}_\pm (y)(x_\pm-y)^l\|\leq M_\pm \|x_\pm- y\|^{\kappa-|j|}
$$
for $x_\pm\in\R^n_\pm$, where $M_\pm \edf\| F_\pm \|_{\mathrm{Lip}^\kappa}$.  Since $\kappa-|j|>1$, this implies
$$
\lim_{\substack{x_\pm\to y\\x_\pm \in \R^n_\pm}}\frac{1}{\|x_\pm- y\|}\|F_\pm ^j(x_\pm)-\sum_{i=1}^n F^{j+e_i}(y)(x_\pm-y)^i\|=0.
$$
Therefore 
$$
\lim_{x\to y }\frac{1}{\|x - y\|}\|F^j(x)-\sum_{i=1}^n F^{j+e_i}(y)(x-y)^i\|=0,
$$
which proves the claim.

Using the claim it follows by induction that $F$ is $[ \kappa]$ times differentiable, and $\p^j F=F^j$ for $0\leq |j|\leq [\kappa]$, in particular 
$$\sup_{\R^n}\|\p^j F\|\leq \max(\sup_{\R^n_-}\|\p^j F_-\|,\sup_{\R^n_+}\|\p^j F_+\|)\leq \max(M_+,M_-).$$
To complete the proof it suffices to apply Lemma
\ref{glueLip-alpha}  to the maps $\p^jF_\pm$ for $|j|=[\kappa]$.
\end{proof}

\begin{pr}\label{extensionProp-Lip}
Let $\kappa\in (0,+\infty)\setminus\N$. There exists   continuous operators
$$
E_\kappa: \bigoplus_{0\leq s\leq [\kappa]} \mathrm{Lip}^{\kappa-s}(\R^{n-1},T)\to \mathrm{Lip}^{\kappa}(\R^n,T)
$$
with the following property:  putting $A\edf E_\kappa((A_s)_{0\leq s\leq [\kappa]})$ we have
\begin{equation}\label{nabla-s-A}
\forall x'\in\R^{n-1},\ \p^s_{n} A(x',0)=A_s(x')
\end{equation}
for $0\leq s\leq [\kappa]$. Similarly, for any $k\in\N$ there exists a continuous operator
$$
F_k: \bigoplus_{0\leq s\leq k} {\cal C}^{k-s}(\R^{n-1},T)\to {\cal C}^{k}(\R^n,T)
$$
such that (\ref{nabla-s-A}) holds for $0\leq s\leq k$. In both cases   $A$ is $ {\cal C}^\infty$ on $\R^n\setminus(\R^{n-1}\times\{0\})$.

\end{pr}

\begin{proof}
Let $(A_s)_{0\leq s\leq [\kappa]}\in \bigoplus_{0\leq s\leq [\kappa]} \mathrm{Lip}^{\kappa-s}(\R^{n-1},T)$. For any $j=(j_1,\dots,j_n)=(j',j_n)\in \N^n$ with $|j|=|j'|+j_n\leq [\kappa]$ let $a^{(j)}\in {\cal C}^{0}(\R^{n-1}\times\{0\},T)$ be given by  
\begin{equation}\label{a(j)}
a^{(j)}(x',0)\edf  \p^{j'} A_{j_n} (x').	
\end{equation}
 We  prove first that 
 \begin{cl}
 The system $\ag=(a^{(j)})_{0\leq |j|\leq [\kappa]}$ belongs to $\mathrm{Lip}^\kappa_{\R^n}(\R^{n-1}\times\{0\},T)$   and   
 $$(A_s)_{0\leq s\leq [\kappa]}\mapsto \ag\edf(a^{(j)})_{0\leq |j|\leq [\kappa]}$$
  defines a continuous operator
 $$
 L_\kappa:\bigoplus_{0\leq s\leq [\kappa]} \mathrm{Lip}^{\kappa-s}(\R^{n-1},T)\to \mathrm{Lip}^\kappa_{\R^n}(\R^{n-1}\times\{0\},T).
 $$	
 \end{cl}
Indeed, since $A_{j_n}\in \mathrm{Lip}^{\kappa-j_n}(\R^{n-1})$ by assumption, we have  the estimates 
\begin{equation}\label{sup|a(j)|}
 \sup_{\R^{n-1}\times\{0\}}\|a^{(j)}\| \leq \sup_{\R^{n-1}} \| \p^{j'} A_{j_n} \|\leq M_{j_n}\edf \| A_{j_n}\|_{\mathrm{Lip}^{\kappa-j_n}} \hb{ for }|j|\leq [\kappa].	
\end{equation}

On the other hand for any $x'$, $y'\in \R^{n-1}$ and $j=(j',j_n)$ with $|j'|+j_n\leq [\kappa]$ we have
 \begin{equation}\label{R[kappa]jag}
 \begin{split}
 R^{[\kappa]}_{j,\ag}((x',0),(y',0))&=a^{(j)}(x',0)-\sum_{|j+l|\leq [\kappa]}\frac{a^{(j+l)}(y')}{l!}(x'-y',0)^l\\
 &=\p^{j'} A_{j_n}(x')-\sum_{|j'|+|l'|\leq [\kappa]-j_n}
 \frac{1}{l'!} \p^{j'+l'}A_{j_n} (y')(x'-y')^{l'}\\
 &=R^{[\kappa]-j_n-|j'|}_{\p^{j'}A_{j_n}}(x',y').
 \end{split}	
 \end{equation}
For the second equality we took into account that     $(x'-y',0)^l=0$ for all $l=(l',l_n)$ with $l_n>0$. Since $A_{j_n}\in \mathrm{Lip}^{\kappa-j_n}(\R^{n-1},T)$, Remark \ref{EstimatesForLipkappa}  gives estimates of the form
 \begin{equation}
 \|R^{[\kappa]-j_n-|j'|}_{\p^{j'}A_{j_n}}(x',y')\|\leq M^{j_n}_{j'}\|x-y\|^{\kappa-j_n -|j'|}\,, 	
 \end{equation}
 which gives $\||R^{[\kappa]}_{j,\ag}((x',0),(y',0))\|\leq M^{j_n}_{j'}\|(x',0)-(y',0)\|^{\kappa -|j|}$. Therefore
 $$
 \|\ag\|_{\mathrm{Lip}^\kappa}\leq \max\{M^s_{j'}|\ 0\leq |j'|+s\leq [\kappa]\},
 $$
which proves the claim.

For the first statement it suffices to put $E_\kappa={\cal E}_\kappa\circ L_\kappa$, where 
$${\cal E}_\kappa:
\mathrm{Lip}^\kappa_{\R^n}(\R^{n-1}\times\{0\},T)\to \mathrm{Lip}^\kappa(\R^{n},T)$$
is Whitney's extension operator given by Theorem \ref{WhitneyExt}.

For the second statement we prove that formula (\ref{a(j)}) for $|j|\leq m$ defines a continuous operator from $\bigoplus_{0\leq s\leq m} {\cal C}^{m-s}(\R^{n-1},T)$ to the Fréchet space of Whitney jets of class ${\cal C}^m$ on $\R^{n-1}\times\{0\}$ (see \cite[section 1.1]{FJW}), and we use Whitney's original extension theorem for ${\cal C}^m$ jets.   Replacing $[\kappa]$ by $m$ in (\ref{sup|a(j)|}), (\ref{R[kappa]jag}), and using  (\ref{RemainderEstimate-j}), we obtain for any compact $K\subset\R^{n-1}$ estimates of the form:
$$
 \sup_{K\times\{0\}}\|a^{(j)}\|\leq \sup_{K} \|\p^{j'} A_{j_n}\| \hb{ for any } j=(j',j_n)\in\N^n \hb{ with }|j|\leq m,
$$
\begin{equation*}
\begin{split}
q_{m,t,K}(\ag)&\edf\sup \bigg\{ \frac{\|R^{m}_{j,\ag}((x',0),(y',0))\|}{\|x'-y'\|^{m-j}}\,\vline\ \begin{array}{l}
 x',\, y'\in K,\,0<\|x'-y'\|\leq  t,\\
 |j|\leq m	
 \end{array}
\bigg\}\\ 
&\leq c \sup\bigg\{\| \p^{j'}A_{s}(x')- \p^{j'}A_{s}(y') \|\,\vline\ \begin{array}{l} x',\, y'\in K,\,\|x'-y'\|\leq  t,\\0\leq s\leq m,\, |j'|=m-s\end{array} \bigg\}.
\end{split}	
\end{equation*}
This shows that $\lim_{t\to 0} q_{m,t,K}(\ag)=0$ and gives estimates for $ \sup_{K\times\{0\}}\|a^{(j)}\|$,  $\sup_{t>0}q_{m,t,K}$ in terms of $\sup_{K} \|\p^{j'} A_{s}\|$, $|j'|\leq m-s$.
 
 \end{proof}
 
Using Whitney extension theorem for ${\cal C}^\infty$ maps \cite{Wh}, we obtain in a similar way:  
 \begin{pr}\label{extensionProp-infty}
 For any $(A_s)_{s\in\N}\in {\cal C}^\infty(\R^{n-1},T)^{\N}$  there exists $A\in {\cal C}^\infty(\R^n,T)$ such that	
 \begin{equation}\label{nabla-infty-A}
 \p^s_{n} A(x',0)=A_s(x') \hb{ for  } x'\in\R^{n-1},\   s\in\N.
\end{equation}
 \end{pr}
 
\begin{re} Proposition \ref{extensionProp-infty} gives a map 
 ${\cal C}^\infty(\R^{n-1},T)^{\N}\ni (A_s)_{s\in \N}\mapsto A\in {\cal C}^\infty(\R^n,T)$
  satisfying (\ref{nabla-s-A}), but such a map can no longer be given by a continuous operator \cite{FJW}.   
	
\end{re}

\begin{co}\label{CoroSums}
Let $(a_l)_{l\in\N}$ be a sequence of ${\cal C}^\infty(\R^n, T)$ such that for any $l\in \N_{\geq 1}$  and any $s\in \N$ with $s\leq l-1$ we have $\p_n^sa_l|_{\R^{n-1}\times\{0\}}=0$. 
There exists $a\in {\cal C}^\infty(\R^n, T)$ such that for any $m\in\N$ and any $s\in \N$ with $s\leq m$ we have $\p_n^s(a-\sum_{l=0}^m a_l)|_{\R^{n-1}\times\{0\}}=0$.
\end{co}
\begin{proof}
Apply Proposition \ref{extensionProp-infty}	to the sequence $(A_s)_{s\in\N}\in {\cal C}^\infty(\R^{n-1},T)^{\N}$, where
$$
A_s(x')\edf \sum_{l\geq 0} \p^s_n a_l(x',0)=\sum_{l= 0}^{s} \p^s_n a_l(x',0).
$$
\end{proof}
Propositions \ref{glueLip-kappa}, \ref{extensionProp-Lip}, \ref{extensionProp-infty} can be generalized for sections in vector bundles on manifolds as follows. Let $U$ be an $n$-dimensional differentiable  manifold and $E$  a ${\cal C}^\infty$ $\K$-vector bundle of rank $r$ on $U$, where $\K\in\{\R,\C\}$. Let ${\cal A}_U$ be the set of all charts (the maximal atlas) of $U$ and ${\cal T}_E$ the set of local trivializations of $E$. For $\theta: E_V\to V\times  \K^r\in {\cal T}_E$ we put   $\theta'\edf \mathrm{p}_{\K^r}\circ\theta:V\to \K^r$.

\begin{dt}\label{CkappaDef}
Let $\kappa\in (0,+\infty)\setminus\N$. We define
 \begin{alignat}{4}
 {\cal C}^\kappa(U,T)\edf & \big\{f\in {\cal C}^0(U,T)\big|&&\ (\chi f|_V)\circ h^{-1}\in \mathrm{Lip}^\kappa(\R^n,T)\hb{ for any}\nonumber\\
  &&&  V\stackrel{h}{\to} W\in {\cal A}_U,\      \chi\in {\cal C}^\infty_c(V,\R) \big\}. \nonumber\\
 \Gamma^\kappa(U,E)\edf &\big\{\sigma\in \Gamma^0(U,E)\big|&&\ (\chi \theta'\circ \sigma|_V)\circ h^{-1}\in \mathrm{Lip}^\kappa(\R^n,\K^r)\hb{ for any}\nonumber\\
  &&&  V\stackrel{h}{\to} W\in {\cal A}_U,\   E_V\stackrel{\theta}{\to}V\times\K^r\in {\cal T}_E,\    \chi\in {\cal C}^\infty_c(V,\R) \big\}.\nonumber 
   \end{alignat}	

 Similarly, for a manifold with boundary $\bar U$ and a ${\cal C}^\infty$ vector bundle  $E$  on $\bar U$, the spaces  ${\cal C}^\kappa(\bar U,T)$, $\Gamma^\kappa(\bar U,E)$ are defined by the same formulae, but using charts with values in open sets $W\subset \R^n_+$ and the Lipschitz spaces $\mathrm{Lip}^\kappa(\R^n_+,T)$ defined above.
\end{dt}
  
 Note that ${\cal C}^\kappa(U,T)$, ${\cal C}^\kappa(\bar U,T)$,  $\Gamma^\kappa(U,E)$, $\Gamma^\kappa(\bar U,E)$,  are naturally   Fréchet spaces; they become Banach spaces (in the sense that their topology can be defined by a single norm) when $U$, respectively $\bar U$ is compact. Definition \ref{CkappaDef} is in accordance with Palais' formalism \cite[section 7]{Pa} and with the definition of the spaces $\Lambda_\alpha$ for manifolds with boundary \cite[section 14.a]{GS}. In particular 
 
 \begin{re}\label{restr-to-compact-of-Ckappa}
 A section $\sigma\in \Gamma^0(U,E)$ ($\sigma\in \Gamma^0(\bar U,E)$) belongs to $\Gamma^\kappa(U,E)$ ($\Gamma^\kappa(\bar U,E)$) if and only if for every $x\in U$ ($x\in \bar U$) there exists a compact $n$-dimensional submanifold with boundary $\bar W\subset U$ ($\bar W\subset \bar U$) which is a neighborhood of $x$ in $U$ (in $\bar U$) such that $\sigma|_{\bar W}\in \Gamma^\kappa(\bar W,E)$.   
 \end{re}

 Let $S\subset U$ be a smooth, closed hypersurface and let $n^*_S\subset T^*_{U|S}$ be the conormal line bundle of $S$ in $U$; this line bundle coincides with the annihilator of $T_S$ in the restriction $T^{*}_{U|S}$ of the   contangent bundle $T_U^{*}$ of $U$ to $S$.

 Let $l$, $m\in \N$ with $l\leq m$. Let $\sigma\in \Gamma^m(U,E)$.  We'll say that order $l$ jet of $s$ along $S$ vanishes, and we write $j^l_S \sigma=0$, if the order $l$ jet $j^l_x\sigma$ of $\sigma$ at $x$ vanishes  for any $x\in S$. If this is the case and $l<m$, the intrinsic derivative $D_S^{l+1}\sigma\in \Gamma^{m-l-1}(S,n_S^{*\otimes (l+1)}\otimes E_S)$ or order $l+1$ is defined, and $D_S^{l+1}\sigma=0$ if and only if $j^{l+1}_S\sigma=0$ (section \ref {IntrinsicDiff} for details). 
\begin{co}\label{extensionCoroCkappa}
\begin{enumerate}
\item 	Let $\kappa\in [0,+\infty)$ and $m\in\N$ with $ m\leq [\kappa]$. There exists a continuous operator 	
$$
E^\kappa_{S,m}:\Gamma^{\kappa-m}(S,n_S^{*\otimes m}\otimes E_S)\to \Gamma^{\kappa}(U,E)
$$
such that, for any $b\in  \Gamma^{\kappa-m}(S,n_S^{*\otimes m}\otimes E_S)$, putting $\sigma\edf E^\kappa_{S,m}(b)$, we have
\begin{equation}\label{j(m-1)Dm}
j^{m-1}_S\sigma=0\ (\hb{if }m\geq 1),\ D^{m}_S\sigma=b,	
\end{equation}
and $\sigma|_{U\setminus S}\in \Gamma^\infty(U\setminus S,E)$.
\item Let  $m\in\N$. For any  $b\in \Gamma^\infty (S,n_S^{*\otimes m}\otimes E_S)$ there exists $\sigma\in \Gamma^\infty(U,E)$ such that (\ref{j(m-1)Dm}) holds.
\item Let $(a_l)_{l\in\N}$ be a sequence of $\Gamma^\infty(U,E)$ such that $j^{l-1}_S a_l=0$ for any $l\geq 1$. There exists $a\in  \Gamma^\infty(U,E)$ such that $j^m_S(a-\sum_{l=0}^m a_l)=0$ for any $m\in\N$.
\end{enumerate}
\end{co}
\begin{proof}
(1) Put $E^\kappa_m\edf E^\kappa\circ e^\kappa_m$, $F^k_m\edf F^k\circ f^k_m$, where 
\begin{equation}
\begin{split}
e^\kappa_m:\mathrm{Lip}^{\kappa-m}(\R^{n-1},T)&\to \hspace{-1mm}	\bigoplus_{0\leq s\leq [\kappa]} \mathrm{Lip}^{\kappa-s}(\R^{n-1},T), \\ 
f^k_m:{\cal C}^{k-m}(\R^{n-1},T)&\to \hspace{-1mm}	\bigoplus_{0\leq s\leq k} {\cal C}^{k-s}(\R^{n-1},T)  	
\end{split}	
\end{equation}
are the obvious embeddings. 
Let $(V_i\stackrel{h_i}{\to}\R^n)_{i\in I}$ be a system of charts of $U$ and $(E_{V_i}\stackrel{\theta_i}{\to}V_i\times\K^n)_{i\in I}$ a system of trivializations of $E$ such that
\begin{enumerate}
\item The family of open sets $(V_i)_{i\in I}$ is locally finite and $\union_{i\in I} V_i\supset S$.
	\item   $\bar V_i$  is compact and $h(V_i\cap S)=\R^{n-1}\times\{0\}$ for any $i\in I$. 	
\end{enumerate}
Via the identifications provided by $h_i$ and $\theta_i$, the operators $E^\kappa_m$,  $F^k_m$ give operators

$$
E^\kappa_{S,m,i}:\Gamma^{\kappa-m}_c(S\cap V_i,n_S^{*\otimes m}\otimes E_S)\to \Gamma^{\kappa}(V_i,E)
$$
satisfying (\ref{j(m-1)Dm}). The point here is that the intrinsic derivative $D^m_S$, on sections  whose $m-1$ jet along $S$ vanishes, is compatible with vector bundle isomorphisms and diffeomorphic base changes. Let $(\varphi_i:S\to[0,1])_{i\in I}$ be a smooth partition of unity on $S$ which is subordinate  to the cover $(S\cap V_i)_{i\in I}$ and let, for any $i\in I$, $\chi_i:U\to [0,1]$ be a smooth function on $U$ such that $\sup(\chi_i)\subset V_i$ and  $\chi\equiv 1$ on a neighborhood of $\sup(\varphi_i)$ (which is compact) in $V_i$. It suffices to put
$$
E^\kappa_{S,m}(b)\edf \sum_{i\in I} \chi_i E^\kappa_{S,m,i}(\varphi_i b).
$$

 For (2) and (3)  we use  Proposition \ref{extensionProp-infty} respectively Corollary \ref{CoroSums} and a similar argument. 

\end{proof}

\begin{co}\label{glueCkappa}
Let $E$ be a ${\cal C}^\infty$ vector bundle on $U$, $S\subset U$  a separating closed real smooth hypersurface, and $U\setminus S= U^-\cup U^+$ a decomposition of $U\setminus S$ as    union of disjoint open subsets such that $\bar U^\pm=U^\pm\cup S$.	Put $E^\pm\edf E_{\bar U^\pm}$ and let $\kappa\in [0,+\infty]$.
\begin{enumerate}
\item  There exists a	continuous extension operator  $\Gamma^\kappa(\bar U^+, E^+)\to \Gamma^\kappa(U,E)$.

\item There exists a continuous operator
$$
{\cal E}_S:\{(\sigma_-,\sigma_+)\in \Gamma^\kappa(U ,E )\times\Gamma^\kappa(U ,E )|\ j_S^{[\kappa]}(\sigma_+-\sigma_-)=0\}\to \Gamma^\kappa(U ,E )
$$
with the property that, putting $\sigma={\cal E}_S(\sigma_-,\sigma_+)$, we have $\sigma|_{\bar U\pm}=\sigma_\pm|_{\bar U\pm}$.
\end{enumerate}
\end{co}

\begin{proof}
(1) follows from Proposition \ref{ExtFromHalfSpace} using a partition of unity. (2) follows from Proposition \ref{glueLip-kappa}	 for $\kappa\in  (0,+\infty)\setminus\N$ and from a similar gluing principle for ${\cal C}^m$ maps if $\kappa=m\in\N\cup\{\infty\}$.
\end{proof}

\subsection{The fiberwise exponential map}
\label{fiberwise-exp-section}
Let $M$ be a differentiable manifold, $G$ a Lie group and $p:P\to M$  a  ${\cal C}^\infty$  principal $G$-bundle on $M$. Let $\iota$ ($\Ad$) be the interior (adjoint) action of $G$ on itself (on its Lie algebra $\g$). Put $\iota(P)\edf P\times_\iota G$, $\Ad(P)\edf P\times_\Ad \g$.   Using Palais's formalism for spaces of sections in locally trivial fiber bundles \cite[p. 38]{Pa}, we have: 
\begin{pr}\label{fiberwise-exp}
	Let $\gamma\in[0,\infty]$. 
\begin{enumerate}
\item The  fiberwise exponential map $\exp:\Ad(P)\to \iota(P)$ maps $\Gamma^\gamma(M,\Ad(P))$ into $\Gamma^\gamma(M,\iota(P))$.
\item 	There exists an Euclidean  structure $h$ on $\Ad(P)$ such that $\exp$
 maps diffeomorphically the unit disk bundle $\Ad(P)_0\edf \{\xi\in \Ad(P)|\ \|\xi\|_h<1\}$ with respect to $h$ onto an  open neighborhood $\iota(P)_0$ of   the identity section $\id_P$ in $\iota(P)$. For any such $h$, the map $\exp$ induces a bijection  $\Gamma^\gamma(M,\Ad(P)_0)\to \Gamma^\gamma(M,\iota(P)_0)$.
\end{enumerate}

\end{pr}
\begin{proof} (1) The map $\exp:\Ad(P)\to \iota(P)$ is    fiber bundle morphism between locally trivial fiber bundles in the sense of \cite[section 10]{Pa}.  The claim follows from \cite[Theorem 13.4]{Pa} taking as base manifold   closures $\bar M'\subset M$ of relatively compact open submanifolds $M'\subset M$ with smooth boundary.\vspace{2mm}\\ 
(2) 
The map $\exp$ maps diffeomorphically the zero section $0_{\Ad(P)}\subset \Ad(P)$ onto $\id_P\subset \iota(P)$ and is fiberwise locally invertible at the points of $0_{\Ad(P)}$. By the relative Inverse Function Theorem \cite[Exercice 14, section 1.§8]{GP}	we obtain an open neighborhood  $U$ of  $0_{\Ad(P)}$ in $\Ad(P)$ such that $\exp(U)$ is open in $\iota(P)$ and the induced map $U\to \exp(U)$ is a diffeomorphism. It suffices to choose an Euclidian structure $h$ on $\Ad(P)$ such that the unit disk bundle with respect to $h$ is contained in $U$. 

For the second claim of (2), note that $\exp:\Ad(P)_0\to \iota(P)_0$ becomes an isomorphism of ${\cal C}^\infty$ fiber bundles in the sense of \cite[section 10]{Pa}, so the claim follows again by \cite[Theorem 13.4]{Pa}.
\end{proof}
 \begin{co}\label{ExtensionOfGamma}
 Let  $S$ be a differentiable manifold and $P^\pm$   ${\cal C}^\infty$ principal $G$-bundles on $S\times\R$. Identify $S$ with $S\times\{0\}$  and let $\upsilon:P^-_{S}\to P^+_{S}$ be a bundle isomorphism of class ${\cal C}^\gamma$. There exists   a  bundle isomorphism extension $\tilde\upsilon:P^-\to P^+$  of class ${\cal C}^\gamma$ of $\upsilon$ which is ${\cal C}^\infty$ on $S\times\R^*$.	
 \end{co}
 
 \begin{proof}  Let  $A^\pm$ be a connection of class ${\cal C}^\infty$ on $P^\pm$. Parallel transport with respect to $A^\pm$ alongs  paths of the form $t\mapsto (u, t)$, $u\in S$, gives ${\cal C}^\infty$ bundle isomorphisms $f^\pm: P^\pm\textmap{\simeq}P^\pm_{S}\times\R$. 
 
 The bundles $P^-_{S}$, $P^+_{S}$ on $S$ are topologically isomorphic, so they are also ${\cal C}^ \infty$ isomorphic.  Therefore we may suppose $P^-=P^+=P_S\times\R\eqcolon P$ (regarded as bundle on $S\times\R$), where $P_S$ is a ${\cal C}^\infty$ principal $G$-bundle on $S$. 
 
 The bundle isomorphism $\upsilon$ can then be regarded as an element of $\Gamma^\gamma(S,\iota(P_S))$.  Let $\iota(P_S)_0$ be an an open neighborhood of $\id$ in $\iota(P_S)$ as in Proposition \ref{fiberwise-exp}. There exists a smooth section $\sigma\in \Gamma^\infty(S\times\R,\iota(P))$  such that $\sigma|_S$ takes values in the disk bundle neighborhood $\iota(P_S)_0\upsilon$ of $\upsilon$. This follows using the density of ${\cal C}^\infty$ with respect to the strong ${\cal C}^0$-topology (see \cite[section 2.1]{Hir}, \cite[Theorem 2.6]{Hir}, \cite[Exercice 3 p. 56]{Hir}).

 Therefore we have
$\upsilon =\phi^{-1}\,\sigma|_S$, where $\phi\in \Gamma^\gamma(S,\iota(P_S)_0)$, because $\upsilon$ is of class ${\cal C}^\gamma$ and $\sigma|_S$ of class ${\cal C}^\infty$.
Making use of Proposition \ref{fiberwise-exp}, let  $\psi\in {\cal  C}^\gamma(S,\Ad(P_S)_0)$ be such that $\phi=\exp(\psi)$. By Corollary  \ref{extensionCoroCkappa} there exists an extension $\tilde\psi\in \Gamma^\gamma(S\times\R,\Ad(P))$ of $\psi$ which is ${\cal C}^\infty$ on $S\times\R^*$.  It suffices to put $\tilde\upsilon=\exp(-\tilde\psi)\sigma$.
 \end{proof}

\subsection{Gluing  bundles along a hypersurface}

Let $U$ be a differentiable manifold, $S\subset U$  a separating closed real smooth hypersurface, and $U\setminus S= U^-\cup U^+$ a decomposition of $U\setminus S$ as    union of disjoint open subsets such that $\bar U^\pm=U^\pm\cup S$. Let $P^\pm$ be a ${\cal C}^\infty$ principal $G$ bundle on $\bar U^\pm$,  $\gamma\in[0,\infty]$,  $\upsilon:P^-_S\to P^+_S$ an isomorphism of class ${\cal C}^{\gamma}$, and let $P^\upsilon\edf P^-\coprod_\upsilon P^+$ be the topological bundle obtained by gluing $P^\pm$ along $S$ via $\upsilon$. $P^\upsilon$ comes with obvious identifications  $P^\pm\to P^\upsilon|_{\bar U^\pm}$.
\begin{dt}\label{DefAdmiss} A ${\cal C}^\infty$ structure $\Sg$ on 
$P^\upsilon\edf P^-\coprod_\upsilon P^+$
 will be called {\it admissible} if, denoting by $P^\upsilon_\Sg$ the corresponding ${\cal C}^\infty$ principal $G$-bundle, the obvious identifications $P^\pm\to P^\upsilon_\Sg|_{\bar U^\pm}$ become bundle isomorphisms of class $ {\cal C}^{\gamma}$ on   $\bar U^\pm$. 
 \end{dt}

Let $\Aut^0(P^\upsilon)\simeq \Gamma^0(U,P^\upsilon\times_\iota G)$ be the gauge group  of the topological bundle $P^\upsilon$ and $\Aut^0(P^\upsilon)_a$ be the subgroup of $\Aut^0(P^\upsilon)$ whose elements are the bundle automorphisms $F\in\Aut^0(P^\upsilon)$ which induce automorphisms of class ${\cal C}^{\gamma}$ on $P^\pm$.

 \begin{pr}\label{CinftyEGamma}
The set $\mathscr{S}_a$ of admissible ${\cal C}^\infty$ structures on $P^\upsilon$ is non-empty. The group $\Aut^0(P^\upsilon)_a$ acts transitively on $\mathscr{S}_a$.  The stabilizer of an element $\Sg\in \mathscr{S}_a$ coincide with the gauge group $\Aut^\infty(P^\upsilon_\Sg)\simeq\Gamma^\infty(U,P^\upsilon_\Sg\times_\iota G)$ of the ${\cal C}^\infty$ bundle $P^\upsilon_\Sg$.
 \end{pr}
\begin{proof}
Let 
$$
S\times\R\textmap{\nu\simeq }N\hookrightarrow U
$$
be a ${\cal C}^\infty$ tubular neighborhood of $S$ in $U$ such that $\nu(S\times \R_{\pm})=N \cap \bar U^\pm$. Let $q:N\to S$, $q_\pm: N \cap \bar U^\pm\to S$ be the projections induced by the obvious projections $S\times\R\to S$, $S\times\R_\pm\to S$.

Put $\tilde U^\pm\edf U^\pm\cup N$, and let $\tilde P^\pm$ be a ${\cal C}^\infty$ bundle on $\tilde U^\pm$ which extends $P^\pm$. One obtains easily such an extension by choosing a  connection $A^\pm$ of class ${\cal C}^\infty$ on $P^\pm$ and noting that parallel transport alongs  paths of the form 
$$\R_\pm\ni t\mapsto \nu(u, t),\ u\in S$$
 gives   ${\cal C}^\infty$ bundle  isomorphisms $\eta^\pm:q_\pm^*(P^\pm_S) \textmap{\simeq}P^\pm_{N\cap \bar U^{\pm}}$. Therefore it suffices to put $\tilde P^\pm\edf P^\pm \coprod_{\eta^\pm} q^*(P^\pm_S)$.
 
By Corollary \ref{ExtensionOfGamma} there exists an   extension $\tilde\upsilon:\tilde P^-_{N}\to \tilde P^+_{N}$ of class ${\cal C}^\gamma$ of $\upsilon$ which is ${\cal C}^\infty$ on $N\setminus S$. Put   $\tilde\upsilon^\pm\edf \tilde\upsilon|_{N\cap U^\pm}$. We obtain obvious bundle isomorphisms
$$
\begin{tikzcd}
\tilde P^-_{\tilde U^-}\coprod_{\tilde\upsilon^+} P^+_{U^+}\ar[r, "b", "\simeq"']&\tilde P^-_{\tilde U^-}\coprod_{\tilde \upsilon}\tilde P^+_{\tilde U^+}\ar[r, "a", "\simeq"']& P^-\coprod_\upsilon P^+=P^\upsilon
\end{tikzcd}
$$
over $U$, where $\tilde P^-_{\tilde U^-}\coprod_{\tilde\upsilon^+} P^+_{U^+}$ is naturally a ${\cal C}^\infty$ bundle, $\tilde P^-_{\tilde U^-}\coprod_{\tilde \upsilon}\tilde P^+_{\tilde U^+}$ is naturally a ${\cal C}^\gamma$ bundle,  $b$ is a bundle isomorphism of class ${\cal C}^\gamma$ and $a$ is a topological bundle isomorphism. The ${\cal C}^\infty$ structure on $P^\upsilon$ induced via $a\circ b$ is obviously admissible.

The other statements follow taking into account that $\Aut^0(P^\upsilon)$ acts transitively on the set of ${\cal C}^\infty$ structures on $E^\upsilon$.
\end{proof}

\subsection{An extension theorem}

The following extension result  plays a fundamental role in this article. Since I could not find it in standard complex analysis textbooks or articles, I give below a detailed proof based on the regularity of the $\bp$ operator. My colleagues  Alexandre Boritchev and Karl Oeljeklaus suggested different proofs, which use   Morera Theorem (for $\dim(U)=1$) combined with the well known theorem on separately holomorphic functions (for $\dim(U)>1$). Another argument, suggested by Christine Laurent-Thiébaut, uses  the  Hartogs-Bochner extension theorem.

\begin{thr}
	\label{acrossS}
Let $U$, $F$ be complex manifolds and  $S\subset U$  a  closed, smooth  real hypersurface.  Let $f:U\to F$ be a continuous map whose restriction  $f|_{U\setminus S}$ is holomorphic. Then $f$ is holomorphic.	
\end{thr}
\begin{proof}
It suffices to prove that statement when $F=\C$ and $U$ is open in $\C^n$, so suppose we are in this case. 
We will show that $\bp f=0$ in the weak sense  around any  point $x\in S$; the claim will follow by the regularity property of the $\bp$ operator.

 Let $B_R\subset\R^{2n}$  be the radius $R$ ball with center $0_{\R^{2n}}$, and 
 $$\bar B_R^\pm\edf \{x\in  \bar B_R|\ \pm x_{2n}\geq 0\}.$$ 
 For  $t\in (-R,\,R)$, $\varepsilon>0$ put:
$$
\bar B_R^t\edf \{x\in\bar B_R|\ x_{2n}=t\},\ \bar B_{R,\varepsilon}^\pm\edf \{x\in  \bar B_R^\pm|\ |x_{2n}|\geq \varepsilon\},\ \bar B_{R,\varepsilon}\edf \{x\in  \bar B_R^\pm|\ |x_{2n}|\leq \varepsilon\}.$$

Let $r>0$ be sufficiently small such that $B(x,r)\subset U$ and there exists a diffeomorphism $\Psi: B(x,r)\to \R^{2n}$ with $\Psi(x)=0$ and $\Psi(S\cap B(x,r))=\R^{2n-1}\times\{0\}$.
 Let $\phi\in A^{n,n-1}(B(x,r))$	 be a type $(n,n-1)$-form with compact support $K\subset B(x,r)$, and let $R>0$ be sufficiently large such that $\Psi(K)\subset B_R$. Then
 \begin{equation}\label{int-on-ball}
 \int_{B(x,r)} f\bar\partial \phi = \int_{\Psi^{-1}(\bar B_R)}f\bar\partial \phi=\lim_{\varepsilon\searrow 0} \int_{\Psi^{-1}(\bar B_{R, \varepsilon}^+)}f\bar\partial \phi+\lim_{\varepsilon\searrow 0} \int_{\Psi^{-1}(\bar B_{R, \varepsilon}^-)}f\bar\partial \phi.
 \end{equation}
\begin{figure}[h]
\includegraphics[scale=0.6]{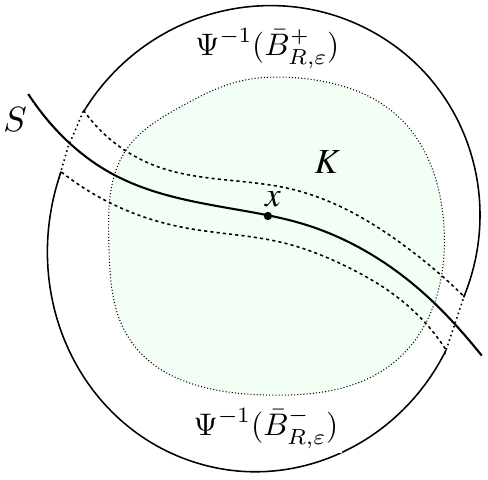}
\caption{$\Psi^{-1}(\bar B_R)$.}
\label{BR}
\end{figure}

We have used: the measure of $\Psi^{-1}(\bar B_{R,\varepsilon})$ (with respect to any Riemannian metric on $U$) tends to 0 as $\varepsilon\to 0$. 
Applying Stokes Theorem to the form $f\phi$ on the manifold with corners $\Psi^{-1}(\bar B_{R, \varepsilon}^\pm)$ (on which $f$ is smooth), we obtain
 $$
 \int_{\Psi^{-1}(\bar B_{R, \varepsilon}^\pm)}f\bar\partial \phi= -\int_{\Psi^{-1}(\bar B_{R, \varepsilon}^\pm)}\bp f \wedge   \phi+\int_{\partial  \Psi^{-1}(\bar B_{R, \varepsilon}^\pm)} f\phi = \int_{  \Psi^{-1}(\p \bar B_{R, \varepsilon}^\pm)} f\phi\,,
 $$
 because $\bp f=0$ on $ \Psi^{-1}(\bar B_{R, \varepsilon}^\pm)\subset U\setminus S$. Endowing $B_{R}^{t}$ with the orientation induced from $\R^{2n-1}\times\{0\}$ regarded as boundary of $\R^{2n-1}\times[0,+\infty)$, and noting that  $\phi$ vanishes on $\Psi^{-1}(\p \bar B_R)$, we obtain (see Fig. \ref{BR}):
 
 $$\lim_{\varepsilon\searrow 0}\int_{\Psi^{-1}(\bar B_{R, \varepsilon}^\pm)}f\bar\partial \phi=\pm \lim_{\varepsilon\searrow 0}\int_{  \Psi^{-1}(\bar B_{R}^{\pm \varepsilon})} f\phi=\pm \int_{  \Psi^{-1}(\bar B_{R}^0)} f\phi\,,
 $$
 so, by (\ref{int-on-ball}), we get $\int_{B(x,r)} f\bar\partial \phi=0$. Therefore $\bp f=0$ around $x$ in distribution sense.
\end{proof}

\begin{co}\label{ContSect}
Let $U$, $F$ be   complex manifolds and $p:\Fg\to U$	 a holomorphic locally trivial fiber bundle with standard fiber $F$. Let $S\subset U$ be a closed, smooth real hypersurface, and $\fg:U\to \Fg$ a continuous section which is holomorphic on $U\setminus S$. Then $\fg$ is holomorphic.
\end{co}
\begin{proof}
This follows from	Theorem \ref{acrossS} using local trivializations around the points of $S$.
\end{proof}

\subsection{Dolbeault operators and bundle almost complex structures}
\label{DolbeaultBundleACS}

We begin by recalling the well known formalism of Dolbeault operators (semi-connections) on complex vector bundles.

\subsubsection{Dolbeault operators on complex vector bundles}\label{DolbeaultSect}
Let $U$ be a complex manifold and $E$  a differentiable complex vector bundle of rank $r$ on $U$. A  Dolbeault operator (semi-connection)  on $E$ is a first order differential operator
$$\delta: A^0(U,E)\to A^{0,1}(U,E)$$
satisfying the Leibniz rule $\delta(f \sigma)=\bp f \sigma+f\delta\sigma$ (see for instance \cite[section 2.2.2]{DK}, \cite{LO}, \cite{LT}, \cite{Te2}).  Such an operator has  natural extensions $A^{0,q}(U,E)\to A^{0,q+1}(U,E)$; the square $\delta^2:A^0(U,E)\to A^{0,2}(U,E)$ is an order 0 operator, so it corresponds to an endomorphism valued form $F_\delta\in A^{0,2}(U,\End(E))$.  
By the bundle version of the Nirenberg-Newlander theorem (see Griffiths \cite[Proposition p. 419]{Gri} (see also \cite[Theorem 5.1]{AHS}, \cite[Proposition I.3.7]{Ko}, \cite[Theorem 2.1.53]{DK}) the $\End(E)$-valued (0,2)-form $F_\delta$ is the obstruction to the integrability of $\delta$. More precisely $F_\delta=0$ if and only if around any point $x\in U$ there exists a local frame $(\theta_1,\dots,\theta_r)$  with $\delta \theta_i=0$. If this is the case, $\delta$ defines a a holomorphic structure $\hg_\delta$ on $E$ characterized by the condition:  a local section $s$ of $E$ is $\hg_\delta$-holomorphic if and only if  $\delta\sigma=0$.

Let now $U^+\subset U$ be an open set whose closure $\bar U^+$ is a smooth submanifold  with boundary, i.e. $\bar U^+=U^+\cup S$, where $S$ is an oriented real hypersurface of $U$ and $\p \bar U^+=S$. Put $E^+\edf E|_{\bar U^+}$, $E_S\edf E|_S$.  

A Dolbeault operator $\delta: A^0(\bar U^+,E^+)\to A^{0,1}(\bar U^+,E^+)$ on $E^+$ and its associated form $F_\delta\in A^{0,2}(\bar U^+,\End(E))$ are defined in the same way as for bundles on manifolds without boundary, but,  in general, the analogue of the Newlander-Nirenberg Theorem   {\it does not} hold at boundary points  \cite[Proposition 1.5, Corollary 2.3]{Te}.   For this reason a Dolbeault operator $\delta$ on $E^+$ satisfying the condition $F_\delta=0$ will be called {\it formally} integrable (not integrable). Similarly, for a formally integrable Dolbeault operator $\delta$ on $E^+$ and an open set $V\subset \bar U^+$, a section $\sigma\in \Gamma (V,E^+)$ will be called   {\it formally}  $\delta$-holomorphic if $\delta\sigma=0$. This condition implies holomorphy  at interior points, but, in general, not at boundary points (not even at boundary points around which a formally $\delta$-holomorphic frame exists).

In this article we make use of a refinement of the above Newlander-Nirenberg for Dolbeault operators with coefficients in ${\cal C}^\kappa$ for $\kappa\in (0,+\infty]\setminus\N$.  This result is a special case of the Newlander-Nirenberg theorem for bundle ACS of class  ${\cal C}^\kappa$ on principal bundles \cite{Te2} which will be recalled briefly in the next section.

\subsubsection{Bundle almost complex structures on principal bundles}\label{ACSsection}

Let $G$ be a complex Lie group, $\g$ its Lie algebra and $\theta\in A^1(G,\g)$ the canonical left invariant $\g$-valued 1-form on $G$ \cite[p. 41]{KN}.
Let $p:P\to U$ be a  principal $G$-bundle of class ${\cal C}^\infty$ on $U$. Let $\kappa\in[0,+\infty]$.
%
%
%
\begin{dt}\label{BdACS}
A bundle almost complex structure  (bundle ACS) of class ${\cal C}^\kappa$ on $P$ is an almost complex structure  $J$ of class ${\cal C}^\kappa$ on $P$ which makes the $G$-action $P\times G\to P$  and the map $p:P\to U$ pseudo-holomorphic. 
 \end{dt}
 
 Let ${\cal J}^\kappa_P$ be the space of bundle ACS of class ${\cal C}^\kappa$ on $P$ and let ${\cal A}^\kappa_P$ be the space of sections $\alpha\in \Gamma^\kappa(P,p_*^{-1}(T^{0,1}_U)^* \otimes\g^{1,0})$ satisfying the conditions:
 
 \begin{enumerate}[(Pa)]
 	\item $\alpha$ is invariant with respect to the $G$ action $g\to \trp{R}_{g*}\otimes \Ad_{g}$ on   $p_*^{-1}(T^{0,1}_U)^* \otimes\g^{1,0}$. 
 	\item  $\alpha(a^\#_y)=a^{1,0}$  for any $y\in P$ and $a\in \g^\C\edf\g\otimes_\R\C=\g^{1,0}\oplus\g^{0,1}$.
 \end{enumerate}       
Here we used the notation $a^\#$ for the vertical vector field associated with $a$.  For any $J\in {\cal J}^\kappa_P$ there exists a unique $\alpha_J\in {\cal A}^\kappa_P$ such that $T^{0,1}_{P,J}=\ker(\alpha_J)$ and the map 
$${\cal J}^\kappa_P\ni J\mapsto \alpha_J\in {\cal A}^\kappa_P$$
is a bijection \cite{Te2}. Via this bijection ${\cal J}^\kappa_P$ gets the natural structure of an affine space with model space $A^{0,1}_{\Ad}(P,\g^{1,0})_\kappa$, the space of $\g^{1,0}$-valued tensorial forms of type $\Ad$ \cite[section II.5]{KN}, class ${\cal C}^\kappa$ and bidegree (0,1) on $P$. This space can be further identified \cite[p. 76]{KN} with the space  $A^{0,1}(U,P\times_\Ad\g^{1,0})_\kappa$ of  forms of class ${\cal C}^\kappa$ and bidegree $(0,1)$ with values in the associated vector bundle $P\times_\Ad\g^{1,0}$. Identifying $\g^{1,0}$ with $\g$ in the standard way, we conclude that ${\cal J}^\kappa_P$ is naturally an affine space with model space $A^{0,1}(U,P\times_\Ad\g)_\kappa=A^{0,1}(U,\Ad(P))_\kappa$.

Let $J\in {\cal J}^\kappa_P$ with  $\kappa\geq 1$. The map
$$
\Gamma(P,T^{0,1}_{P,J})^2\ni (A,B)\stackrel{\fg_J}{\mapsto} -\alpha_J([A,B]) 
$$
defines a $\g^{1,0}$-valued tensorial  form of type $(0,2)$ and class ${\cal C}^{\kappa-1}$ on $P$ hence an element  $\fg_J\in A^{0,2}_{\Ad}(P,  \g^{1,0})_{\kappa-1}=A^{0,2}_{\Ad}(P,  \g)_{\kappa-1}=A^{0,2} (U,\Ad(P))_{\kappa-1}$.

  The behavior of the map $J\mapsto\fg_J$ with respect to translations in the affine space ${\cal J}^\kappa_P$ is given by the formula
 \begin{equation}\label{fg-(J+b)}
 \fg_{J+b}=\fg_J+\bar\kg_J(b),	
 \end{equation}
where $\bar \kg_J:A^{0,1}_\Ad(P,\g^{1,0})_\kappa \to  A^{0,2}_\Ad(P,\g^{1,0})_{\kappa-1}$  is defined by
$$
\bar\kg_J(b)=\bp_J b+\frac{1}{2}[b\wedge b]$$
(see \cite[Proposition 2.9]{Te2}). Here $\bp_J$ stands for the Dolbeault operator on the vector bundle $P\times_\Ad\g^{1,0}\simeq\Ad(P)$ associated with $J$.

Let $W\subset U$ be an open subset, and $\tau\in \Gamma(W,P)$ be a local section of class ${\cal C}^\infty$ of $P$. Putting 
$$\alpha_J^\tau\edf \tau^*(\alpha)\in A^{0,1}(W,\g)_\kappa$$
 and, denoting by $\fg_J^\tau\in A^{0,2}(W,\g)_{\kappa-1}$ the form associated with $\fg_J$ with respect to $\tau$, we have (see \cite{Te2}):
$$
\fg_J^\tau=\bp\alpha_J^\tau+\frac{1}{2}[\alpha_J^\tau\wedge\alpha_J^\tau].
$$
 This formula shows that $\fg_J$ can be also  defined for $\kappa\in[0,1)$ as an $\Ad(P)$-valued form of type $(0,2)$ on $U$ with distribution coefficients.

 We refer to \cite{Te2} for the following principal bundle version of the Nirenberg-Newlander theorem:
 \begin{thr}[The Nirenberg-Newlander theorem for principal bundles]  \label{NNG}
 Let $G$ be a complex Lie group  and $p:P\to U$ a differentiable principal $G$-bundle on $U$. Let $J$ be a bundle ACS of class ${\cal C}^\kappa$ on $P$ with $\kappa\in (0,+\infty]\setminus\N$. The following conditions are equivalent:
 \begin{enumerate}
 	\item $\fg_J=0$.
 	\item $J$ is integrable in the following sense: for any point $x\in U$ there exists an open neighborhood $W$ of $x$ and a $J$-pseudo-holomorphic section $s\in \Gamma^{\kappa+1}(W,P)$.
 \end{enumerate}
 If this is the case, $J$ defines a bundle holomorphic reduction $\hg_J$ of the underlying ${\cal C}^{\kappa+1}$-bundle of $P$. For an open set $W\subset U$, a section $s\in \Gamma^{1}(W,P)$ is holomorphic with respect to $\hg_J$ if and only if it is $J$-pseudo-holomorphic; if this is the case then $s\in \Gamma^{\kappa+1}(W,P)$. 
\end{thr}
For $\kappa\in (0,1)$ the condition $\fg_J=0$ is meant in distribution sense.     We also refer to \cite{Te2} for the following regularity result:

\begin{co}
 \label{kappa-regularity}
 Let $U$ be a complex manifold, $G$ a complex Lie group, and $P$ a principal bundle of class ${\cal C}^\infty$ on $P$. Let $J$ be an integrable bundle	 ACS of class ${\cal C}^\kappa$  on $P$ with $\kappa\in (0,+\infty]\setminus\N$, and let $G\times F\to F$ be a holomorphic action of $G$ on a complex manifold $F$. The sheaf of $J$-holomorphic sections of the associated bundle $P\times_G F$ is contained in the sheaf of  sections of class ${\cal C}^{\kappa+1}$.
\end{co}

Let $\iota:G\to\Aut(G)$ be the morphism which maps any $g\in G$ to the interior automorphism $\iota_g$. An equivariant map $\sigma\in {\cal C}^{\kappa+1}_\iota(P,G)$ defines a gauge transformation $\tilde\sigma:P\to P$ of class ${\cal C}^{\kappa+1}$ of $P$ and the map 
$${\cal C}^{\kappa+1}_\iota(P,G)\ni \sigma\mapsto\tilde\sigma\in  {\cal G}^{\kappa+1}_P$$
 is an isomorphism onto the gauge group ${\cal G}^{\kappa+1}_P$ of $P$. The group ${\cal C}^{\kappa+1}_\iota(P,G)$ acts on the space ${\cal J}^\kappa_P$ from the right by the formula
 $$
 J\cdot \sigma\edf \tilde\sigma^{-1}_*\circ J\circ \tilde\sigma_*
 $$
and the corresponding action on ${\cal A}^\kappa_P$ is 
 $$
 \alpha\cdot \sigma=\alpha\circ \sigma_*,
 $$
from which we infer the behavior of the integrability obstruction $\fg_J$ with respect to the gauge symmetry of the space ${\cal J}^\kappa_P$:
 \begin{equation}\label{fg-Jsigma}
 \fg_{J\cdot\sigma}=\Ad_{\sigma^{-1}}(\fg_J).	
 \end{equation}
 We have the following formula (see \cite[Proposition 2.10]{Te2}) relating the affine space structure of ${\cal J}^\kappa_P$ to its gauge symmetry:
 \begin{equation}\label{Jsigma-J=lgsigma}
 J\cdot\sigma=J+\bar\lg_J(\sigma),	
 \end{equation}
 where the map
 $$\bar\lg_J:{\cal C}^{\kappa+1}_\iota(P,G)\to A^{0,1}_{\Ad}(P,\g^{1,0})_\kappa=A^{0,1}(U,P\times_\Ad\g^{1,0})_\kappa\simeq A^{0,1}(U,\Ad(P))_\kappa $$
  is defined by
 $$
 \bar\lg_J(\sigma)\edf \sigma^*(\theta^{1,0})^{0,1}_J.
 $$
Here $\theta^{1,0}$ is the holomorphic (1,0)-form on $G$ defined as the composition 
$$\theta\otimes\id_\C: T_G^\C\to \g^\C=\g^{1,0}\oplus\g^{0,1}\to \g^{1,0}.$$
It is useful to have an explicit formula for $\bar\lg_J$ with respect to a local trivialization (or, equivalently, local section) of $P$.  For a local section $\tau\in \Gamma(W,P)$ of class ${\cal C}^\infty$ of $P$ put 
$$\bar\lg_J^\tau(\sigma)\edf \tau^*(\bar\lg_J(\sigma))\in A^{0,1}(W,\g^{1,0})_\kappa.$$
We have (see \cite[Lemma 2.8]{Te2}):
\begin{equation}\label{bar-lg-sigma}
\bar \lg^ \tau_J(\sigma)=\sigma_\tau^*(\theta^{1,0})^{0,1}+(\Ad_{\sigma_\tau^{-1}}-\id)(\alpha_J^\tau),	
\end{equation}
where $\sigma_\tau\edf\sigma\circ\tau\in {\cal C}^{\kappa+1}(W,G)$.
Note the following useful formula for the composition $\bar\kg_J\circ\bar\lg_J$  associated with a bundle ACS $J$ of class ${\cal C}^1$. For any $\sigma\in {\cal C}^2_\iota(P,G)$ we have  \cite[Corollary 2.11]{Te2}:
\begin{equation}
\bar\kg_J\circ\bar\lg_J(\sigma)=(\Ad_{\sigma^{-1}}-\id)(\fg_J).	
\end{equation}

Let $J$, $J'\in {\cal J}^\kappa_P$ and $\sigma \in {\cal C}^{\kappa+1}_\iota(P,G)$. We have
$$
\alpha_{J'\cdot\sigma}-\alpha_{J\cdot\sigma}=\alpha_{J'}\circ\tilde\sigma_*-\alpha_{J}\circ\tilde\sigma_*=(\alpha_{J'}-\alpha_J)\circ\tilde\sigma_*.
$$
Since $\alpha_{J'}-\alpha_J$ is a tensorial form of type $\Ad$ (hence it vanishes  on vertical tangent and is  $\Ad$-equivariant) we obtain the formula
\begin{equation}\label{J'sigma-Jsigma}
J'\cdot\sigma-J\cdot\sigma=\Ad_{\sigma^{-1}}(J'-J),	
\end{equation}
which shows that the group ${\cal C}^{\kappa+1}_\iota(P,G)$ acts on ${\cal J}^\kappa_P$ by affine transformations and the induced linear action on the model vector space $A^{0,1}_{\Ad}(P,\g^{1,0})_\kappa$ is 
$$
(\beta,\sigma)\mapsto \Ad_{\sigma^{-1}}(\beta).
$$
We will need:
\begin{lm}\label{kgJlgJ}  Let $J$ be a bundle ACS of class ${\cal C}^1$ on $P$. Then
\begin{enumerate}
\item For any $\sigma_0$, $\sigma_1\in {\cal C}^1_\iota(P,G)$ we have
\begin{equation}\label{lJ-sigma1-sigma0}
\bar\lg_J(\sigma_1\sigma_0)=\Ad_{\sigma_0^{-1}}(\bar\lg_J(\sigma_1))+\bar\lg_J(\sigma_0).	
\end{equation}

\item For any $\sigma\in \Gamma^2(U,\iota(P))$ and $\beta\in \Gamma^1(U,\extp^{0,1}\otimes\Ad(P))$ we have: 
\begin{equation}
\bar\kg_J(\Ad_\sigma(\beta-\bar\lg_J(\sigma))=\Ad_\sigma(\bar\kg_J(\beta))+(\Ad_\sigma-\id)(\fg_J).	
\end{equation}
\end{enumerate}
\end{lm}

\begin{proof}
(1) By (\ref{Jsigma-J=lgsigma}) and (\ref{J'sigma-Jsigma}) we have
$$
\lg_J(\sigma_1\sigma_0)=J\cdot(\sigma_1\sigma_0)-J=J\cdot(\sigma_1\sigma_0)-J\cdot\sigma_0+J\cdot\sigma_0-J=
$$
$$
=(J\cdot\sigma_1)\cdot\sigma_0-J\cdot\sigma_0+\bar\lg_J(\sigma_0)=\Ad_{\sigma_0^{-1}}(J\cdot\sigma_1-J)+\bar\lg_J(\sigma_0)=\Ad_{\sigma_0^{-1}}(\bar\lg_J(\sigma_1))+\bar\lg_J(\sigma_0).
$$
(2) Using (\ref{fg-(J+b)}), (\ref{fg-Jsigma}), and (\ref{J'sigma-Jsigma}) we obtain:
\begin{equation}\label{TTT}
\begin{split}
\Ad_\sigma(\bar\kg_J(\beta))+(\Ad_\sigma-\id)(\fg_J)&=\Ad_\sigma(\bar\kg_J(\beta)+\fg_J)-\fg_J=\Ad_\sigma(\fg_{J+\beta})-\fg_J\\
&=\fg_{(J+\beta)\cdot\sigma^{-1}}-\fg_J=\bar\kg_J((J+\beta)\cdot\sigma^{-1}-J)\\
&=\bar\kg_J((J+\beta)\cdot\sigma^{-1}-J\cdot\sigma^{-1}+J\cdot\sigma^{-1}-J)\\
&=\bar\kg_J(\Ad_\sigma(\beta)+J\cdot\sigma^{-1}-J).	
\end{split}	
\end{equation}
On the other hand:
$$
J\cdot\sigma^{-1}-J=J\cdot\sigma^{-1}-(J\cdot\sigma)\cdot\sigma^{-1}=
-\Ad_\sigma(\bar\lg_J\sigma).
$$
We used  (\ref{J'sigma-Jsigma}) with $J'= J\cdot\sigma$ and (\ref{Jsigma-J=lgsigma}). Taking into account (\ref{TTT}), this completes the proof.

\end{proof}

 \subsubsection{The formal integrability condition on manifolds with boundary}\label{FIC}
 
The definitions above  generalize  in an obvious way for a ${\cal C}^\infty$ principal $G$-bundle on a manifold with boundary. The regularity class of a bundle ACS $J^+$ on a bundle on a manifold with boundary $\bar U^+$ is defined taking into account the regularity class of the associated forms $\alpha^\tau_{J^+}$ in the sense of Definition \ref{CkappaDef}.

 Let now $S\subset U$ be a separating, oriented smooth real hypersurface in $U$ and $U=\bar U^-\cup \bar U^+$ the corresponding decomposition of $U$ as union of manifolds with boundary. Let $P$ be a principal $G$-bundle of class ${\cal C}^\infty$ on $U$ and let $P^\pm$, $P_S$ be the restrictions of $P$ to $\bar U^\pm$, $S$ respectively. Let $\kappa\in (0,+\infty]\setminus\N$, $k\edf[\kappa]$,  $J$ be a bundle ACS of class ${\cal C}^\kappa$ on $P$, and $J^\pm$ be the restriction of $J$ to $P^\pm$.

Our problem: express the integrability condition on $J$ in terms of its restrictions  $J^\pm$ to $P^\pm$. By the Newlander-Nirenberg theorem for principal bundles (Theorem \ref{NNG}), the answer is obvious in the case $\kappa > 1$:  

 \begin{re}\label{fgJpm=0->J-is-int}
 Suppose $\kappa> 1$.  $J$ is integrable if and only if the forms 	$\fg_{J^\pm}\in A^{0,2}(\bar U^\pm, \Ad(P))_{\kappa-1}$ vanish. 
 \end{re}
The case $\kappa\in(0,1)$ is more delicate. In this case one can considering the restrictions $\cringle{J}^\pm$ of $J^\pm$ to the bundles $P_{U^\pm}$ over the open sets $U^\pm$ and the corresponding distributions $\fg_{\cringle{J}^\pm}$ on $U^\pm$, but one cannot expect the vanishing of these distributions to imply the integrability of $J$ (i.e. the vanishing of the distribution $\fg_J$ on $U$). The key observation here is:
\begin{re}\label{fg-J+forJC0}
Let $J^+$ be a continuous bundle ACS on $P^+$. Then the distribution $\fg_{\cringle{J}^+}\in {\cal D}'(U^+,\extp^{0,2}_{\;U^+}\otimes \Ad(P))$ extends as a continuous linear functional on the space $\Gamma^1_c(\bar U^+,\extp^{n,n-2}_{\;\bar U^+}\otimes\Ad(P)^*)$ of compactly supported sections of class ${\cal C}^1$ in the indicated bundle. If $J^+$ is of class ${\cal C}^1$, this extension coincides with the functional associated with the continuous form  $\fg_{J^+}$ on $\bar U^+$. 
\end{re}

 \begin{proof}
 
 Suppose first that  $J^+$ is of class ${\cal C}^1$. In this case $\fg_{J^+}$ is a continuous form on $\bar U^\pm$, and the associated linear functional on $\Gamma^1_c(\bar U^+,\extp^{n,n-2}_{\;\bar U^+}\otimes\Ad(P)^*)$ acts by
 $$
 \langle \fg_{J^+}, \varphi\rangle=\int_{\bar U^+} \langle \fg_{J^+}\wedge \varphi\rangle. 
 $$ 
 Let $W\stackrel{\hb{\tiny open}}{\subset} U$  and $\tau:\bar W^+\edf W\cap\bar U^+\to P^+$  be a local section of class ${\cal C}^2$ of $P^+$.  The associated form $\alpha_{J^+}^\tau$  belongs to $\Gamma^1(\bar W^+,\extp^{0,1}_{\;\bar W^+}\otimes\g)$.  
 
For any $\varphi\in\Gamma^1_c(\bar W^+, \extp^{n,n-2}_{\;\bar U^+}\otimes\Ad(P)^*)$ let $\varphi^\tau\in \Gamma^1_c(\bar W^+, \extp^{n,n-2}_{\;\bar U^+}\otimes\g^*)$ be the $\g^*$-valued form associated with $\varphi$ with respect to $\tau$. Using Stokes theorem,
\begin{align}\label{formula-for-fJ+}
\langle \fg_{J^+}, \varphi\rangle& = \langle \fg_{J^+}^\tau, \varphi^\tau\rangle=\int_{\bar W^+}\big \langle  \big(\bp \alpha_{J^+}^\tau+\frac{1}{2}[\alpha_{J^+}^\tau\wedge\alpha_{J^+}^\tau]\big)\wedge \varphi^\tau\big\rangle\nonumber \\
&=	\int_{\bar W^+} d\langle\alpha_{J^+}^\tau\wedge\varphi^\tau\rangle +\int_{\bar W^+}\big(\langle \alpha_{J^+}^\tau\wedge\bp \varphi^\tau\rangle +\frac{1}{2}\langle [\alpha_{J^+}^\tau\wedge\alpha_{J^+}^\tau] \wedge \varphi^\tau\rangle\big) \nonumber \\
&=\int_{\p \bar W^+} \langle\alpha_{J^+}^\tau\wedge\varphi^\tau\rangle +\int_{\bar W^+}\big(\langle \alpha_{J^+}^\tau\wedge\bp \varphi^\tau\rangle +\frac{1}{2}\langle [\alpha_{J^+}^\tau\wedge\alpha_{J^+}^\tau] \wedge \varphi^\tau\rangle\big).	
\end{align}
The right hand expression in (\ref{formula-for-fJ+}) has obviously sense and is continuous with respect to $\varphi^\tau$ (in the ${\cal C}^1$-topology) even if $J^+$ is only of class ${\cal C}^0$ and $\tau$ is only of class ${\cal C}^1$, because under these weaker assumptions the form $\alpha^\tau_{J^+}$ remains continuous.  Moreover, for $J^+$ of class ${\cal C}^0$ fixed, this expression gives a well defined (independent of $\tau$) linear functional on $\Gamma^1_c(\bar W^+,\extp^{n,n-2}_{\;\bar U^+}\otimes\Ad(P)^*)$. Indeed, we claim that for any $\tau$, $\tau'\in\Gamma^1(\bar W^+,P^+)$  we have
\begin{equation}\label{tau-tauf}
\begin{split}
&\int_{\p \bar W^+} \langle\alpha_{J^+}^\tau\wedge\varphi^\tau\rangle +\int_{\bar W^+}\big(\langle \alpha_{J^+}^\tau\wedge\bp \varphi^\tau\rangle +\frac{1}{2}\langle [\alpha_{J^+}^\tau\wedge\alpha_{J^+}^\tau] \wedge \varphi^\tau\rangle\big)\\
=&\int_{\p \bar W^+} \langle\alpha_{J^+}^{\tau'}\wedge\varphi^{\tau'}\rangle +\int_{\bar W^+}\big(\langle \alpha_{J^+}^{\tau'}\wedge\bp \varphi^{\tau'}\rangle +\frac{1}{2}\langle [\alpha_{J^+}^{\tau'}\wedge\alpha_{J^+}^{\tau'}] \wedge \varphi^{\tau'}\rangle\big).	
\end{split}	
\end{equation}
By (\ref{formula-for-fJ+}), this equality is clear when $J^+$ is of class ${\cal C}^1$ and $\tau$, $\tau'$ are of class ${\cal C}^2$. Fixing $\varphi$ and writing 
$$J^+=\lim_{n\to\infty}J^+_n\hb{ (in the ${\cal C}^0$-topology)},\ \tau=\lim_{n\to\infty}\tau_n,\ \tau'=\lim_{n\to\infty}\tau'_n \hb{ (in the ${\cal C}^1$-topology)}$$
  with $J_n^+$, $\tau_n$, $\tau'_n$ of class ${\cal C}^\infty$, we conclude  that  (\ref{tau-tauf}) also holds for  $J^+$  of class ${\cal C}^0$ and $\tau$, $\tau'$ of class ${\cal C}^1$. The same formula can be used to show that the linear functionals associated with two sections   $\tau\in\Gamma^1(\bar W^+,P^+)$,  $\tau'\in\Gamma^1(\bar W'^+,P^+)$ agree on  
$$\Gamma^1_c(\bar W^+\cap \bar W'^+,\extp^{n,n-2}_{\;\bar U^+}\otimes\Ad(P)^*),$$
so we obtain a well defined linear functional on $\Gamma^1_c(\bar U^+,\extp^{n,n-2}_{\;\bar U^+}\otimes\Ad(P)^*)$, obviously extending the distribution $\fg_{\cringle{J}^+}$.
\end{proof}

For a bundle ACS $J^+$ of class ${\cal C}^0$ on $P^+$  we will use the notation $\fg_{J^+}$ for the linear functional provided by Remark \ref{fg-J+forJC0}. Note that $\fg_{J^+}$ can be regarded as an element of the  space $\dot{{\cal D}}'(\bar U^+,\extp^{0,2}_{\;\bar U^+}\otimes \Ad(P))$ of $\extp^{0,2}_{\;\bar U^+}\otimes \Ad(P)$-valued distributions supported by $\bar U^+$  (see \cite[section I.1]{Me}).
The map $J^+\mapsto\fg_{J^+}$ is functorial with respect to ${\cal C}^1$-isomorphisms of principal bundles on $\bar U^+$, in particular:
\begin{re}\label{equiv-formula-for-C0}
The equivariance formula (\ref{fg-Jsigma})  generalizes to a bundle ACS $J^+$ of class ${\cal C}^0$ on $P^+$ and a gauge transformation $\sigma\in \Gamma^1(\bar U^+,\iota(P^+))$.
\end{re}

\begin{dt}\label{DefFormIntegrC0} Let  $P^+$ be a principal $G$-bundle on $\bar U^+$.
A bundle ACS $J^+$ of class ${\cal C}^1$ on $P^+$  will be called formally integrable, if $\fg_{J^+}=0$ in the space of $\Ad(P)$-valued continuous (0,2) forms on $\bar U^+$.

More generally, a bundle ACS $J^+$ of class ${\cal C}^0$  on  $P^+$  will be called formally integrable, if $\fg_{J^+}=0$ in the space of $\extp^{0,2}_{\;\bar U^+}\otimes \Ad(P)$-valued distributions supported by $\bar U^+$.	 \end{dt}
 With Definition \ref{DefFormIntegrC0} we have the following   generalization of Remark \ref{fgJpm=0->J-is-int}:

\begin{pr}  \label{(J-,J+)->Jintegrable}
Let $J$ be a bundle ACS of class ${\cal C}^0$ on $P$ and  $J^\pm$  its restriction to $P^\pm$. \begin{enumerate}
	\item If  $J^\pm$   are formally integrable, then $\fg_J=0$ in distribution sense. 
	\item Suppose $J\in {\cal J}^\kappa_P$ with $\kappa\in (0,+\infty]\setminus\N$. Then $J$ is integrable iff and only if $J^\pm$   are formally integrable.
\end{enumerate}
  \end{pr}
\begin{proof}
(1) Let $\tau\in \Gamma^1(W,P)$ be a local section of $P$ and $\varphi\in\Gamma^1_c( W, \extp^{n,n-2}_{\;  U}\otimes\Ad(P)^*)$.	Put $\bar W^\pm\edf W\cap\bar U^\pm$, $\varphi_\pm\edf \varphi|_{\bar W^\pm}$. We have
\begin{equation}\label{fgJ-+fgJ+}
\begin{split}
\langle \fg_{J^-}&,\varphi_-\rangle+\langle \fg_{J^+},\varphi_+\rangle=\\
&=\int_{\p \bar W^-} \langle\alpha_{J}^\tau\wedge\varphi^\tau_-\rangle +\int_{\bar W^-}\big(\langle \alpha_{J}^\tau\wedge\bp \varphi^\tau_-\rangle +\frac{1}{2}\langle [\alpha_{J}^\tau\wedge\alpha_{J}^\tau] \wedge \varphi^\tau_-\rangle\big)\\
&+\int_{\p \bar W^+} \langle\alpha_{J}^\tau\wedge\varphi^\tau_+\rangle +\int_{\bar W^+}\big(\langle \alpha_{J}^\tau\wedge\bp \varphi^\tau_+\rangle +\frac{1}{2}\langle [\alpha_{J}^\tau\wedge\alpha_{J}^\tau] \wedge \varphi^\tau_+\rangle\big)	.
\end{split}	
\end{equation}

We obviously have $\varphi_-^\tau|_{S\cap W}=\varphi_+^\tau|_{S\cap W}=\varphi^\tau|_{S\cap W}$. Taking into account that the oriented boundaries $\p \bar W^-$, $\p \bar W^+$ coincide with $S\cap W$ endowed with opposite orientations, it follows that 
$\int_{\p \bar W^-} \langle\alpha_{J}^\tau\wedge\varphi^\tau_-\rangle +\int_{\p \bar W^+} \langle\alpha_{J}^\tau\wedge\varphi^\tau_+\rangle=0$.
By (\ref{fgJ-+fgJ+}),
$$
\langle \fg_{J^-},\varphi_-\rangle+\langle \fg_{J^+},\varphi_+\rangle=\int_{ W}\big(\langle \alpha_{J}^\tau\wedge\bp \varphi^\tau\rangle +\frac{1}{2}\langle [\alpha_{J}^\tau\wedge\alpha_{J}^\tau] \wedge \varphi^\tau\rangle\big)=\langle \fg_J,\varphi\rangle,
$$
so the vanishing of $\fg_{J^\pm}$  as distributions supported by $\bar U^\pm$ implies the vanishing of the distribution $\fg_J$.
\vspace{2mm}\\
(2) If $J^\pm$ are formally integrable, then $\fg_J=0$ in distribution sense, so $J$ is integrable by Theorem \ref{NNG}. Conversely, if $J$ is integrable, then around any point $x\in U$ there exists a local section $\tau:W\to P$ of class ${\cal C}^{\kappa+1}$ which is $J$-pseudo-holomorphic. Therefore $\alpha^\tau_J=0$. Put $\tau^\pm\edf\tau|_{W\cap\bar U^\pm}:W\cap\bar U^\pm\to P^\pm$. We have $\alpha_{J^\pm}^{\tau^\pm}=\alpha^\tau_J|_{W\cap\bar U^\pm}=0$, so the restriction of $\fg_{J^\pm}$ (regarded as distribution supported by $\bar U^\pm$) to $W\cap\bar U^\pm$ vanishes. Therefore $\fg_{J^\pm}=0$, so $J^\pm$ are formally integrable.
\end{proof}
\begin{re}
In the special case $G=\GL(r,\C)$ we obtain the formal integrability condition for Dolbeault operators  on a vector bundle $E^+$ on   $\bar U^+$: A Dolbeault operator $\delta^+$ of class ${\cal C}^1$ on $E^+$ is formally integrable if the continuous $\End(E)$-valued form $F_\delta$ vanishes on $\bar U^+$.  A Dolbeault operator $\delta^+$ of class ${\cal C}^0$ on $E^+$ is formally integrable if $F_\delta=0$ in the space of $\extp^{0,2}_{\;\bar U^+}\otimes\End(E^+)$-valued distributions supported by $\bar U^+$. With this definition, the analogue for vector bundles of Proposition \ref{(J-,J+)->Jintegrable} holds.
\end{re}

 \subsection{Intrinsic higher order  differentials}
 \label{IntrinsicDiff}

Let  $k\in\N$ and $l\in\Z$ with $l\leq k$. Let $M$ be a differentiable manifold, $F$ a finite dimensional real vector space, $f\in {\cal C}^{k}(M,F)$, and $x\in M$. The condition
$$
\begin{array}{c}
\hb{\it With respect to a local chart around $x$, all partial derivatives of order $\leq l$ }	\\
\hb{\it   of $f$ at $x$ vanish}
\end{array}
$$
is independent of the chosen chart. This follows from the composition formula \cite[Section I.6]{Ma}. If this condition is satisfied, we we'll say that the order $l$ jet of $f$ at $x$ vanishes, and we shall write $j^l_xf=0$.  For negative $l$ the condition $j^l_xf=0$ becomes superfluous (satisfied by any $f\in {\cal C}^{k}(M,F)$). Note that, for $l\geq 0$ we have $j^l_xf=0$ iff and only if $f(x)=0$ and $j^{l-1}_xdf=0$.

\begin{lm} \label{ABLemma}
Let $F_1$, $F_2$, $F$ be finite dimensional real vector spaces, $k\in \N$ and $l$, $l_1$, $l_2\in\Z$,  such that $l\leq k$, $l_1+l_2+1\leq k$. Let $x\in M$.
\begin{enumerate}
\item Let   $b:F_1\times F_2\to F$ be a bilinear map and    $f_i\in 	{\cal C}^k(M,F_i)$ with $j^{l_i}_x f_i=0$. Then $j^{l_1+l_2+1}_xb(f_1,f_2)=0$. In particular, for $l_1=-1$, we have the implication 
$$
j^{l}_x f_2=0\Rightarrow j^{l}_x\, b(f_1,f_2)=0.
$$

\item Let $f\in {\cal C}^k(M, F_1)$ with $j^l_x f=0$, $V_1$ an open neighborhood of $\im(f)$ in $F_1$ and $g\in {\cal C}^k(V_1, F_2)$ such that $g(0)=0$. Then $j^l_x(g\circ f)=0$.
\item Let  $f_i\in {\cal C}^k(M, F_i)$ and $\Phi\in {\cal C}^k(F_1\times F_2, F)$. If  $j^l_{f_1(x)}(y_1\mapsto \Phi(y_1,0))=0$ and $j_{x}^l f_2=0$, then $j_{x}^l\Phi(f_1,f_2)=0$.
\end{enumerate}

\end{lm}
\begin{proof} We may suppose $M=\R^n$, $F_i=\R^{m_i}$, $F=\R^m$. \vspace{2mm}\\
(1) The claim follows easily using the Leibniz rule.
\vspace{2mm}\\
(2) The claim follows by induction using the formula $d(g\circ f)(y)=dg(f(y))\, df(y)$ and (1) taking $l_1=l$ and $l_2=-1$.
\vspace{2mm}\\
(3) Induction with respect to $l$: 
For $l=0$, taking into account the assumptions, we have $\Phi(f_1(x),f_2(x))=\Phi(f_1(x),0)=0$.
Let $l\geq 1$ and suppose that the statement is true for $l-1$. For $u\in M=\R^n$ we have:
$$
d\Phi(f_1,f_2)(u)=\partial_{1}\Phi(f_1(u),f_2(u))\,  df_1(u)+\partial_{2} \Phi (f_1(u),f_2(u))\,  df_2(u).$$
The assumption  $j^l_{f_1(x)}(y_1\mapsto \Phi(y_1,0))=0$ gives $j^{l-1}_{f_1(x)}(y_1\mapsto \partial_{1}\Phi(y_1,0))=0$. We also have $j_x^{l-1} f_2=0$ (because $j_x^l f_2=0$), so, the induction hypothesis applied to  $(f_1,f_2,\partial_{1}\Phi)$ gives $j_{x}^{l-1}\partial_1\Phi(f_1,f_2))=0$. Therefore $j^{l-1}_x(\partial_{1}\Phi(f_1 ,f_2 )\, df_1)=0$ by (1). On the other hand, the hypothesis  $j_x^l f_2=0$ implies  $j_x^{l-1} df_2=0$, so again by  (1) we obtain  $j^{l-1}_x(\partial_{2} \Phi (f_1,f_2)\, df_2)=0$. Therefore $j^{l-1}_x(d\Phi(f_1,f_2))=0$, so, since $\Phi(f_1(x),f_2(x))=0$, we have $j^{l}_x\Phi(f_1,f_2)=0$ as claimed.

\end{proof}

Suppose that $f\in {\cal C}^{k}(M,F)$  with $j^l_xf=0$ where $0\leq l<k$. Using the same composition formula cited above it follows that, for tangent vectors $v_1,\dots,v_{l+1}\in T_xM$  the element
$$
d^{l+1}_{h(x)}(f\circ h^{-1})(h_*(v_1),\dots,h_*(v_{l+1})) 
$$
(where $x\in W\textmap{h} W'\subset\R^n$ is a chart around $x$) of $F$ depends only on	$v_1,\dots,v_{l+1}$, not on $h$. Therefore, if $j^l_xf=0$, we obtain a well defined symmetric $(l+1)$-linear map
$$
D_x^{l+1}f:T_{M,x}^{l+1}\to F
$$
which will be called the {\it intrinsic differential of order} $(l+1)$ of $f$ at $x$. 

Let now $S\subset M$ be smooth hypersurface. If $j^l_xf=0$ for any $x\in S$ we'll say that the  order $l$ jet of $f$ along $S$ vanishes, and we'll write $j^l_Sf=0$. 

\begin{re}
Suppose that $j^l_Sf=0$, where $0\leq l<k$,  and let $x\in S$. Then $D^{l+1}_xf(v_1,\dots,v_{l+1})=0$  if one of the tangent vectors $v_i$ belongs to $T_xS$.
\end{re}
\begin{proof}   

We may suppose $M=\R^n$, $S=\R^{n-1}\times\{0\}$. It suffices to prove that $\p^\alpha f (x',0)$ for any $x'\in\R^{n-1}$ and  any $\alpha\in \N^n$ with $|\alpha|=l+1$ for which there exists $i\in\{1,\dots,n-1\}$ with $\alpha_i>0$. Let $\alpha\in \N^n$ with $|\alpha|=l+1$ and let  $i\in\{1,\dots,n-1\}$ with $\alpha_i>0$.  Denote by $(e_1,\dots,e_n)$ be the canonical basis of $\R^n$.   Putting $\beta\edf \alpha-e_i$ we have $\beta\in\N^n$, $|\beta|=l$ and
$$
\p^\alpha f(x',0)=\p_i (\p^\beta f)(x',0).
$$
The right hand term vanishes because, since we assumed  $j^l_{y}f=0$ for any $y\in S$, we have  $\p^\beta f(x'+te_i,0)=0$ for any $t\in\R$.
\end{proof}

Therefore, if $j^l_Sf=0$ and $x\in S$, then $D^{l+1}_xf(v_1,\dots,v_{l+1})$ depends only  on the images of $v_i$ in the normal line $n_{S,x}=T_{M,x}/T_{S,x}$, so the family $(D_x^{l+1}f)_{x\in S}$   defines a section
$$
D^{l+1}_Sf\in \Gamma^0(S,n_S^{*\otimes(l+1)}\otimes F),
$$
which will be called the {\it intrinsic differential of order} $(l+1)$ of $f$ along $S$.

Le now $E$ be a real vector bundle of rank $r$ and class ${\cal C}^\infty$ on $M$, and let $\sigma\in \Gamma^{l+1}(M,E)$ be a section of class ${\cal C}^{l+1}$ of $E$. Let $x\in M$. For a local trivialization $\theta:E_W\to W\times\R^r$ put $\sigma^\theta\edf p_{\R^r}\circ\theta\in {\cal C}^{l+1}(W,\R^r)$. The condition  
$$
\begin{array}{c}
\hb{\it With respect to a local trivialization $\theta$ around $x$ we have $j^l_x(\sigma^\theta)=0$ }	
\end{array}
$$
is independent of $\theta$. If this condition is satisfied,  we'll say that the order $l$ jet of $\sigma$ at $x$ vanishes, and  we'll write $j^l_x\sigma=0$. If this is the case (and $l\geq 0$), we obtain a well defined intrinsic differential  $D_x^{l+1}\sigma:T_{M,x}^{l+1}\to E_x$ of  order $(l+1)$.

If $j^l_x\sigma=0$ for any $x\in S$, we'll say that  the  order $l$ jet of $\sigma$ along $S$ vanishes,   we'll write $j^l_Sf=0$, and (if $l\geq 0$) we obtain a well defined intrinsic order $(l+1)$ differential 
$$
D^{l+1}_S\sigma\in \Gamma^0(S,n_S^{*\otimes(l+1)}\otimes E_S)
$$
of $\sigma$ along $S$. If $E$ is a complex vector bundle, we can regard $D^{l+1}_S\sigma$ as an element of $\Gamma^0(S,\eta_S^{\otimes(l+1)}\otimes E_S)$,  where $\eta_S$ is the complexified  conormal line bundle of $S$.

\begin{lm}\label{Dls0-3}

Let $U$, $V$, $F$ be finite dimensional complex  vector spaces, $S\subset U$  a smooth real hypersurface, and $f\in  {\cal C}^{k}(U,V)$ be such that $j_S^{l}f=0$, where $l<k$. 

Put $U_\C\edf U\otimes_\R\C$. For $0\leq s\leq l+1$ regard the  order $s$ differential $d^{s}f$ of $f$ on $U$ as a map $U\to U_\C^{*\otimes s}\otimes V$ of class ${\cal C}^{l+1-s}$ which   takes values in  $L^{s}_{\rm sym}(U_\C,V)\subset U_\C^{*\otimes s}\otimes V$.

Let $\omega \in A^{1,0}(V,F)$ be an $F$-valued $(1,0)$ form of class ${\cal C}^\infty$ on $V$ regarded as element in  ${\cal C}^\infty(V,\Hom_\C(V,F))$ and put 
$$\omega^f\edf\omega\circ f\in{\cal C}^{k}(U,\Hom_\C(V,F)).$$
The $F$-valued forms $f^*(\omega)$,  $f^*(\omega)^{0,1}$ on $U$ will be regarded as 
elements of the spaces  ${\cal C}^{k-1}(U,U^*_\C\otimes F)$,  ${\cal C}^{k-1}(U,U^{*0,1}_\C\otimes F)$ respectively. 

Then $j^{l-1}_S(d f)=0$, $j^{l-1}_S(\bp f)=0$, $j^{l-1}_S(f^*(\omega))=0$, 	$j^{l-1}_S(f^*(\omega)^{0,1})=0$   and  the intrinsic order $l$ differentials of $df$, $\bp f$, $f^*(\omega)$ and $f^*(\omega)^{0,1}$ along $S$ are given  by the following formulae:
\begin{alignat}{4}
\label{Dls0}
D^l_S(d f)&=\big(\id_{\eta_S}^{\otimes l}\otimes(\id_{\eta_S}\otimes\id_V)\big)(D^{l+1}_Sf)&&\in  \Gamma^0(U,\eta_S^{\otimes l}\otimes (\eta_S\otimes V)) \\
  &&&\subset  \Gamma^0(U,\eta_S^{\otimes l}\otimes (U^*_\C\otimes V)),	\nonumber\\
\label{Dls1}
 D^l_S(\bp f) &=\big(\id_{\eta_S}^{\otimes l}\otimes(\psi_S\otimes\id_V)\big)(D^{l+1}_Sf)&&\in    \Gamma^0(U,\eta_S^{\otimes l}\otimes (\eta_S^{0,1}\otimes V))\\
&&&\subset  \Gamma^0(U,\eta_S^{\otimes l}\otimes (U_\C^{*0,1}\otimes V)),	\nonumber\\
\label{Dls2}
D^l_S(f^*(\omega)) &=  \big(\id_{\eta_S}^{\otimes l}\otimes\omega^f_S\cdot(\id_{\eta_S}\otimes  \id_V) \big)(D^{l+1}_Sf)&&\in    \Gamma^0(S,\eta_S^{\otimes l}\otimes(\eta_S\otimes F)) \\
&&&\subset  \Gamma^0(S,\eta_S^{\otimes l}\otimes (U^*_\C\otimes F)),\nonumber \\
\label{Dls3}
D^l_S(f^*(\omega)^{0,1}) &=  \big(\id_{\eta_S}^{\otimes l}\otimes\omega^f_S\cdot(\psi_S\otimes  \id_V) \big)(D^{l+1}_Sf)&&\in  \Gamma^0(S,\eta_S^{\otimes l}\otimes(\eta_S^{0,1}\otimes F))\\
&&&\subset \Gamma^0(S,\eta_S^{\otimes l}\otimes (U^{*01}_\C\otimes F)),\nonumber
\end{alignat}
where, on the right:
\begin{itemize}
\item $\eta_S$ ($\eta_S^{0,1}$) is   regarded as a line subbundle of the trivial bundle with fibre $U_\C^{*}$ (respectively $U_\C^{*0,1}$)  on $S$,
\item $\omega^f_S\cdot$ denotes the  morphism 
$$S\times \Hom(U_\C,V)\textmap{}  S\times \Hom(U_\C,F)$$
of trivial bundles on $S$ defined by pointwise composition with $\omega^f$, and also the induced bundle morphisms on $S$:
$$
\eta_S\otimes V\to \eta_S\otimes F,\ \eta^{0,1}_S\otimes V\to \eta^{0,1}_S\otimes F.
$$	
\end{itemize}
  
\end{lm}

\begin{proof}
The recursive definition of the higher order differentials gives for  $0\leq s\leq l$
\begin{equation}\label{ds(df)}
d^s(df)=\big(\id_{U^*_\C}^{\otimes s}\otimes (\id_{U^*_\C}\otimes\id_V)\big)(d^{s+1}f)\in {\cal C}^0(U,U_\C^{*\otimes s}\otimes (U_\C^{*}\otimes V)).	
\end{equation}
This implies
\begin{equation}\label{dlbpf}
d^s(\bp f)=\big(\id_{U^*_\C}^{\otimes s}\otimes (p^{0,1}\otimes\id_V)\big)(d^{s+1}f)\in {\cal C}^0(U,U_\C^{*\otimes s}\otimes (U_\C^{*0,1}\otimes V)	),
\end{equation}
where $p^{0,1}:U_\C^*\to U_\C^{*0,1}$ is the obvious projection,  shows that the condition $j^l_Sf=0$  implies  $j^{l-1}_S(df)=0$, $j^{l-1}_S(\bp f)=0$, and proves formulae (\ref{Dls0}), (\ref{Dls1}).

The  forms  $f^*(\omega)$, $f^*(\omega)^{0,1}$   are given by
\begin{equation}\label{Omegaf-0-1}
f^*(\omega)=\omega^f\cdot df \in {\cal C}^l(U,U^*_\C\otimes F),\ 
f^*(\omega)^{0,1}=  \omega^f\cdot \bp f\in {\cal C}^l(U,U^{*0,1}_\C\otimes F),
\end{equation}
where $\omega^f\cdot$ denotes the  morphism 
$$U\times \Hom(U_\C,V)\textmap{}  U\times \Hom(U_\C,F)$$
of trivial bundles on $U$ defined by pointwise composition with $\omega^f$.
Since $j^{l-1}_S(df)=0$, $j^{l-1}_S(\bp f)=0$,  we obtain  $j^{l-1}_S(f^*(\omega))=0$, $j^{l-1}_S(f^*(\omega)^{0,1})=0$ and formulae (\ref{Dls2}), (\ref{Dls3})  follow from  (\ref{Omegaf-0-1}) using the Leibniz rule noting that $\omega^f_S$ is induced by $\omega^f$.

\end{proof}

\begin{lm}\label{Dl-delta-beta}
1. Let $U$, $F$ be complex vector spaces, $S\subset U$  a smooth real hypersurface, and $\beta$ an $F$-valued $(0,q)$ form with coefficients  ${\cal C}^{l+1}$ on $U$, regarded as element in ${\cal C}^{l+1}(U, U_\C^{*0,q}\otimes F)$. Suppose that   $j^l_S(\beta)=0$. Then $j^{l-1}_S(\bp\beta)=0$ (if $l\geq 1$) and the intrinsic order $l$ differential of $\bp\beta$ along $S$ is given by
\begin{equation}\label{Dl-bp-beta-new}
\begin{split}
D^l_S(\bp\beta)&=	\big(\id_{\eta_S}^{\otimes l}\otimes\wedge(\psi_S\otimes\id_{ U^{*0,q}\otimes F})\big)(D^{l+1}_S\beta)\\  
&\in\Gamma^0(U,\eta_S^{\otimes l}\otimes(\eta^{0,1}_S\wedge U_\C^{*0,q})\otimes   F))\subset \Gamma^0(U,( \eta_S^{\otimes l}\otimes (U_\C^{*0,q+1}\otimes   F)),
\end{split}		
\end{equation}
where, on the right, $\wedge$ denotes the bundle morphism   %
$$\eta^{0,1}_S\otimes (U^{*0,q}\otimes F) \to (\eta^{0,1}_S\wedge U^{*0,q})\otimes F\hookrightarrow S\times (U^{*0,q+1}\otimes F)$$
 on $S$ induced by the wedge product $\wedge: U_\C^{*0,1}\otimes (U^{*0,q}\otimes F)\to  U^{*0,q+1}\otimes F$.

2. More generally, let $U$ be a complex manifold, $E$  a complex vector bundle on $U$, $\delta$ a (not necessarily integrable) Dolbeault operator with coefficients in ${\cal C}^l$ on $E$, and $\beta\in \Gamma^{l+1}(U, \extp^{0,q}_{\; U}\otimes E)$ with $j^l_S(\beta)=0$. Then  $j^{l-1}_S(\delta\beta)=0$ (if $l\geq 1$), and
\begin{equation}\label{Dl-delta-beta-eq}
\begin{split}
 D^l_S(\delta\beta)=&	\big(\id_{\eta_S}^{\otimes l}\otimes\wedge(\psi_S\otimes\id_{\extp^{0,q}_{\,U|S}\otimes E_S})\big)(D^{l+1}_S\beta)\\
 \in\,& \Gamma^0(S, \eta_S^{\otimes l}\otimes (\eta_S^{0,1}\wedge \extp^{0,q}_{\;U|S})\otimes E_S)) \subset \Gamma^0(S,  \eta_S^{\otimes l}\otimes (\extp^{0,q+1}_{\;U|S}\otimes E_S)).
 \end{split}	
\end{equation}
\end{lm}
\begin{proof}  1. Regard  $\beta$ as an element   $\tilde\beta\in {\cal C}^{l+1}(U,U^{*0,q}_\C\otimes F)$. The explicit formula in coordinates for the operator $\bp$ on $(0,q)$ forms gives:
$$\bp (\sum_I \beta_Id\bar z^I)=\sum_I \bp\beta_I\wedge d\bar z^I=\wedge\big(\sum_I (d\beta^I)^{0,1}\otimes d\bar z^I\big)=$$
$$=\wedge(p^{0,1}\otimes \id_{U_\C^{*0,q}\otimes F})\big(\sum_I d\beta^I\otimes d\bar z^I\big)=\wedge(p^{0,1}\otimes \id_{U_\C^{*0,q}\otimes F})(d\tilde\beta),
$$
and (\ref{Dl-bp-beta-new}) follows from (\ref{Dls0}) applied to $\tilde\beta$.

2. For (\ref{Dl-delta-beta-eq}) we use the formula of the operator $\delta$ with respect to a local trivialization $\tau:E_W\to W\times\C^r$ of $E$. Identifying $\Gamma^{l+1}(W,E)$ with ${\cal C}^{l+1}(W,\C^r)$ via $\tau$, we have
$$
\delta(\beta)=\bp\beta+\alpha^\tau\wedge \beta 
$$
with $\alpha^\tau\in \Gamma^l(W,\extp^{0,1}_{\; U}\otimes \End(E))$. Since we assumed $j^l_S(\beta)=0$ we have $D^l_S(\alpha^\tau\wedge \beta)=0$, so (\ref{Dl-delta-beta-eq}) follows from (\ref{Dl-bp-beta-new}).

\end{proof}


\begin{thebibliography}{BBHM}



\bibitem[AHS]{AHS} M. Atiyah, N. Hitchin,  I. Singer, Self-duality in four-dimensional Riemannian geometry, Proc. R . Soc. Lond. A. 362 (1978), 425-461.


\bibitem[At]{At} M.  Atiyah, Vector bundles over an elliptic curve, Proc. London Math. Soc, 7 (1957), 414-452.

\bibitem[BGS]{BGS} R. Beals, P. C. Greiner, N. K. Stanton, $L^p$ and Lipschitz Estimates
for the $\bp$-Equation and the $\bp$-Neumann Problem, Math. Ann. 277 (1987), 185-196.

\bibitem[BBHM]{BBHM} P. Baldi, M. Berti, E. Haus, R. Montalto, Time quasi-periodic gravity water waves in finite depth, Invent. math. 214 (2018),  739–911.

\bibitem[Be]{Be} S. R. Bell, The Cauchy Transform, Potential Theory and Conformal Mapping, 2nd edition, CRC Press, Taylor\&Francis Group, Boca Raton, London, New York (2016).


\bibitem [Bor]{Bor} A. Borichev, private email, September 15, 2023.

\bibitem [Bot]{Bo} Th. Bothner, On the origins of Riemann–Hilbert
problems in mathematics, Nonlinearity 34 (2021) R1–R73.




\bibitem[Do]{Do}  S. Donaldson, Boundary value problems for Yang-Mills fields, Journal of Geometry and Physics 8 (1992) 89-122.

\bibitem[DK]{DK}   S. Donaldson, P. Kronheimer, The Geometry of Four-Manifolds, Oxford University Press (1990).

\bibitem[Fo]{Fo} G. B. Folland, The tangential Cauchy-Riemann complex on spheres, Transactions of the AMS, Vol. 171 (1972), 83-133.

\bibitem[FJW]{FJW} L. Frerick, E. Jordá, J, Wengenroth, Whitney extension operators without loss of derivatives, Rev. Mat. Iberoam. 32 (2016), no. 2, 377–390.

\bibitem[FK]{FK} G. B. Folland, J. J. Kohn, The Neumann Problem for the Cauchy-Riemann Complex, Annals of Mathematics Studies 75,  Princeton University Press, Year (1972).




\bibitem[Gra]{Gra} H. Grauert, Analytische Faserungen über holomorph-vollständigen Räumen. Math. Ann. 135 (1958),  263–273.

\bibitem[Gri]{Gri} Ph. Griffiths, The extension problem in complex analysis II. Embeddings with positive normal bundle, American Journal of Mathematics, Vol. 88 (1966),   366-446.


\bibitem[Gro]{Gro}  A. Grothendieck, Sur la classification des fibrés holomorphes sur La sphère de Riemann, American Journal of Mathematics, Vol. 79  (1957), 121-138.







\bibitem[GP]{GP} V. Guillemin, A. Pollack, Differential Topology, American Mathematical Society, Providence, Rhode Island (1974).

\bibitem[GS]{GS} P. C. Greiner, E. M. Stein, Estimates for the $\bp$-Neumann problem, Princeton Univ. Press, Princeton (1977).
 


\bibitem[Hil]{Hil} D. Hilbert, Grundzüge einer allgemeinen Theorie der linearen Integralgleichungen, Leipzig und Berlin Druck und Verlag  von B. G.Teubner (1912).



\bibitem[Hir]{Hir} M. W. Hirsch, Differential Topology, Graduate Texts in Mathematics 33, Springer-Verlag (991).

\bibitem[HiNa]{HiNa} C.  Hill, M. Nacinovich,
A collar neighborhood theorem for a complex manifold,
Rendiconti del Seminario Matematico della Università di Padova, tome 91 (1994),  23-30.

\bibitem[Hö]{Ho} L. Hörmander, Differential Operators of Principal Type, Math. Annalen 140, (1960) 124-146. 

\bibitem[It]{It} A. R. Its, The Riemann-Hilbert
Problem and Integrable Systems, Notices of the AMS, Vol. 50, No. 11 (2003), 1389-1400.


\bibitem[JW]{JW} A. Jonsson, H. Wallin, Function Spaces on Subsets of $\R^n$, Mathematical Reports Vol. 2, Part 1, Harwood Academic Publishers, Chur, London, Paris, Utrecht, New York (1984).

\bibitem[Ko]{Ko} S. Kobayashi, Differential Geometry of Complex Vector Bundles. Publications of the Mathematical Society of Japan 15, Princeton University Press, Princeton, NJ, (1987).

\bibitem[KN]{KN} S. Kobayashi, K. Nomizu, Foundations of Differential Geometry I, Interscience Publ., New York, Vol. I (1963).






\bibitem[LM]{LM} I. Lieb, J. Michel, The Cauchy-Riemann Complex. Integral formulae and Neumann problem, Aspects of Mathematics Vol. E34, Friedr. Vieweg \& Sohn, Braunschweig/Wiesbaden (2002).

\bibitem[LT]{LT} M. Lübke, A. Teleman, The Kobayashi-Hitchin correspondence,    World Scientific Publishing Co.   (1995).

\bibitem[LO]{LO} M. Lübke, Ch. Okonek, Moduli Spaces of Simple Bundles and Hermitian-Einstein Connections,  Math. Ann. 276   (1987), 663-674.


\bibitem[Ma]{Ma} B. Malgrange, Ideals of Differentiable Functions, Tata Institute of Fundamental Research, Bombay, Oxford University Press (1966). 


\bibitem[Me]{Me} R. B. Melrose, Transformation of boundary value problems, Acta Math. 147 (1981), 149-236.

\bibitem[Mu]{Mu} N. I. Muskhelishvili, Singular Intergral Equations. Boundary Problems of Function Theory and their Applications to Mathematical Physics, P. Noordhoff N. V., Groningen-Holland, 1953. 



\bibitem[Pa]{Pa} R. Palais, Foundations of Global Non-Linear Analysis, W. A. Benjamin, Inc., New York, Amsterdam (1968).

\bibitem[Ple]{Ple} J. Plemelj,  Riemannsche Funktionenscharen mit gegebener Monodromiegruppe, Monatsh. f. Mathematik und Physik 19 (1908), 211-45.

\bibitem[PS]{PS} A. Pressley and G. Segal, Loop Groups, Oxford University Press, Oxford, (1986).

\bibitem[Pe]{Pe} Peter Petersen, Manifold Theory, UCLA, \url{https://www.math.ucla.edu/~petersen/manifolds.pdf}.




%

\bibitem[St]{St} E. M. Stein, Singular Integrals and Differentiability Properties of Functions, Princeton University Press, Princeton (1970).

\bibitem[See]{See} R. T. Seeley, Extension of ${\cal C}^\infty$ functions defined in a half space, Proc. Amer. Math. Soc. 15 (1964), 625-626. 


\bibitem[Te1]{Te} A. Teleman, Holomorphic bundles on complex manifolds with boundary, \url{https://hal.science/hal-03614586v1}, to appear in Math. Res. Letters.

\bibitem[Te2]{Te2} A. Teleman, The Newlander-Nirenberg theorem for principal bundles, 
Math. Z. 306, 20 (2024), \url{https://doi.org/10.1007/s00209-023-03413-4}.

\bibitem[TT]{TT} A. Teleman, M. Toma, Holomorphic bundles framed along a real hypersurface, in preparation.


\bibitem[Ve]{Ve} Vekua, I. N., Generalized Analytic Functions,  Pergamon Press, Oxford-London-New York-Paris, 1962.

\bibitem[Wh]{Wh} H. Whitney, Analytic extensions of differentiable functions defined in closed sets, Transactions of the American Mathematical Society, American Mathematical Society, 36 (1) (1934), 63–89.

\bibitem[Xi]{Xi} Z. Xi, 
Hermitian–Einstein metrics on holomorphic vector bundles over Hermitian manifolds, Journal of Geometry and Physics 53 (2005) 315-335.

\end{thebibliography}
\end{document}